\documentclass[a4paper,11pt]{article}

\usepackage[utf8]{inputenc}
\usepackage[T1]{fontenc}
\usepackage[a4paper]{geometry}
\usepackage{xcolor}
\usepackage{amsmath, amsfonts, amssymb, amsthm}
\usepackage{mathabx}
\usepackage{stmaryrd}
\usepackage[colorlinks=true]{hyperref}
\usepackage[english]{babel}
\usepackage{enumitem}
\usepackage[capitalise]{cleveref}
\crefname{equation}{}{} %default : eq., eqs.
\crefname{enumi}{}{} %default : eq., eqs.

		\newcommand{\set}[1]{\left\lbrace #1 \right\rbrace}
		\newcommand{\parths}[1]{\left( #1 \right)}
		
		\newcommand{\lrangle}[1]{\left< #1 \right>}
		\newcommand{\lrbracket}[1]{\left[ #1 \right]}

		\newcommand{\scalp}[2]{(#1,#2)}
		\newcommand{\scald}[2]{#1 \cdot #2}
		\newcommand{\seq}[2]{\left(#1\right)_{#2}}

		\newcommand{\floor}[1]{\left\lfloor #1 \right\rfloor}

		\newcommand{\abs}[1]{\ensuremath{\left| #1 \right|}}
		\newcommand{\norm}[1]{\ensuremath{\left\| #1 \right\|}}

		\newcommand{\proba}[1]{\ensuremath{\mathbb P\left(#1\right)}}
		\newcommand{\probal}[2]{\ensuremath{\mathbb P_{#1}\!\left(#2\right)}}
		\newcommand{\esp}[1]{\ensuremath{\mathbb E\left[#1\right]}}
		\newcommand{\espl}[2]{\ensuremath{\mathbb E_{#1}\!\!\left[#2\right]}}
		
			\usepackage{bbm}
			\newcommand{\indset}[1]{\ensuremath{\mathbbm{1}_{{#1}}}}
			\newcommand{\ind}[1]{\indset{\set{#1}}}
		\newcommand{\law}[1]{\ensuremath{\mathbb P_{#1}}}
	\renewcommand{\epsilon}{\varepsilon}
	\renewcommand{\phi}{\varphi}
	\renewcommand{\tilde}{\widetilde}

	\renewcommand{\bar}{\overline}

	\DeclareMathOperator{\id}{id}

	\DeclareMathOperator{\Tr}{Tr}

		\let\div\undefined
		\DeclareMathOperator{\div}{div}

	\DeclareMathOperator{\Lip}{Lip}

	\theoremstyle{plain} %résultats
		\newtheorem{theorem}{Theorem}[section]

		\newtheorem{proposition}[theorem]{Proposition}
		
		\newtheorem{lemma}[theorem]{Lemma}

	\theoremstyle{definition} %définitions
		\newtheorem{definition}{Definition}
		\newtheorem{hypothesis}{Assumption}
			\crefname{hypothesis}{Assumption}{Assumptions}
	\theoremstyle{remark} %remarques
		\newtheorem*{remark}{Remark}

\newcommand{\taue}{{\ensuremath{\tau^\epsilon}}}
\newcommand{\tauLe}{{\ensuremath{\tau_\Lambda^\epsilon}}}
\newcommand{\tauLze}{{\ensuremath{\tau_\Lambda(\zeta^\epsilon)}}}
\newcommand{\stopped}{\ensuremath{^{\epsilon,\taue}}}
\newcommand{\stoppedL}{\ensuremath{^{\epsilon,\tauLe}}}
\newcommand{\stoppedLz}{\ensuremath{^{\epsilon,\tauLze}}}
\newcommand{\tauei}{{\ensuremath{\tau^{\epsilon_i}}}}
\newcommand{\tauLei}{{\ensuremath{\tau_\Lambda^{\epsilon_i}}}}
\newcommand{\tauLzei}{{\ensuremath{\tau_\Lambda(\zeta^{\epsilon_i})}}}
\newcommand{\stoppedi}{\ensuremath{^{\epsilon_i,\tauei}}}
\newcommand{\stoppedLi}{\ensuremath{^{\epsilon_i,\tauLei}}}
\newcommand{\stoppedLzi}{\ensuremath{^{\epsilon_i,\tauLzei}}}
\newcommand{\tSet}{{\ensuremath{\mathbb R^+}}}
\newcommand{\xSet}{{\ensuremath{\mathbb T^d}}}
\newcommand{\vSet}{{\ensuremath{V}}}
\newcommand{\fSet}{{L^2(\mathcal M^{-1})}}
\newcommand{\scalf}[2]{\scalp{#1}{#2}_\fSet}
\newcommand{\scalrho}[2]{\scalp{#1}{#2}_{L^2_x}}

\newcommand{\ul}[1]{\underline{#1}}

\DeclareMathOperator{\psq}{quad}

\title{Diffusion limit for a stochastic kinetic problem with unbounded driving process}

\author{Shmuel Rakotonirina{-}-Ricquebourg\thanks{Univ Lyon, Université Claude Bernard Lyon 1, CNRS UMR 5208, Institut Camille Jordan, F-69622 Villeurbanne, France}}

\begin{document}
\maketitle

\paragraph{Abstract}

This paper studies the limit of a kinetic evolution equation involving a small parameter and driven by a random process which also scales with the
small parameter. In order to prove the convergence in distribution to the solution of a stochastic diffusion equation while removing a boundedness assumption on the driving random process, we adapt the method of perturbed test functions to work with stopped martingales problems.

\tableofcontents

\section{Introduction} \label{sec:introduction}

Our aim in this work is to study the scaling limit of a stochastic kinetic equation in the diffusion approximation regime, both in Partial Differential Equation (PDE) and probabilistic senses. For deterministic problems, this is a thoroughly studied field in the literature, starting historically with \cite{larsen1974asymptotic,bensoussan1979boundary}. Kinetic models with small parameters appear in various situations, for example when studying semi-conductors \cite{golse1992limite} and discrete velocity models \cite{lions1997diffusive} or as a limit of a particle system, either with a single particle \cite{goudon2009stochastic} or multiple ones \cite{poupaud2003classical}. It is important to understand the limiting equations, which are in general much easier to simulate numerically. For instance, in the asymptotic regime we study, the velocity variable disappears at the limit.

When a random term with the correct scaling (here $t/\epsilon^2$) is added to a differential equation, it is classical that, when $\epsilon \to 0$, the solution may converge in distribution to a diffusion process, which solves a Stochastic Differential Equation (SDE) driven by a white noise in time. This is a diffusive limit in the probabilistic sense. Such convergence has been proved initially by Has'minskii \cite{hasminskii1966limit,hasminskii1966stochastic} and then, using the martingale approach and perturbed test functions, in the classical article \cite{papanicolaou1977martingale} (see also \cite{kushner1984approximation,ethier1986markov,fouque2007wave,pavliotis2008multiscale,de2012diffusion}). The use of perturbed test functions in the context of PDEs with diffusive limits also concerns various situations, for instance in the context of viscosity solutions \cite{evans1989perturbed}, nonlinear Schrödinger equations \cite{marty2006splitting,de2010nonlinear,debussche2010quintic,de2012diffusion}, a parabolic PDE \cite{pardoux2003homogenization} or,, as in this paper, kinetic SPDEs \cite{debussche2011diffusion,debussche2017diffusion,debussche2020diffusion}.

In this article, we consider the following equation
\begin{gather}
	\partial_t f^\epsilon+ \frac{1}{\epsilon} \scald{a(v)}{\nabla_x f^\epsilon} = \frac{1}{\epsilon^2} Lf^\epsilon + \frac 1 \epsilon f^\epsilon \bar m^\epsilon, \label{eq:SPDE}\\
	f^\epsilon(0) = f^\epsilon_0. \label{eq:initial_condition}
\end{gather}
where $f^\epsilon$ is defined on $\tSet \times \xSet \times \vSet$, $L$ is a linear operator (see \cref{eq:def_interactions} below) and the source term $\bar m^\epsilon$ is a random process defined on $\tSet \times \xSet$ (satisfying assumptions given in \cref{subsec:driving_random_term}). The goal of this article is to study the limit $\epsilon \to 0$ of its solution $f^\epsilon$, and to generalize previous results of \cite{debussche2011diffusion}.

The solution $f^\epsilon(t,x,v)$ is interpreted as a probability distribution function of particles, having position $x$ and velocity $a(v)$ at time $t$. The variable $v$ belongs to a measure space $(\vSet, \mu)$, where $\mu$ is a probability measure. The function $a$ models the velocity.

The Bhatnagar-Gross-Krook operator $L$ expresses the particle interactions, defined on $L^1(\vSet,\mu)$ by
\begin{equation} \label{eq:def_interactions}
	Lf = \rho \mathcal M - f,
\end{equation}
where $\rho \doteq \int_\vSet f d\mu$ and $\mathcal M \in L^1(\vSet)$.

The source term $\bar m^\epsilon$ is defined as
\begin{equation} \label{eq:def_m_epsilon}
	\bar m^\epsilon(t,x) = \bar m(t/\epsilon^2,x),
\end{equation}
where $\bar m$ is a random process, not depending on $\epsilon$.

In the deterministic case $\bar m^\epsilon = 0$, such a problem occurs in various physical situations \cite{degond2000diffusion}. The density $\rho^\epsilon \doteq \int_\vSet f^\epsilon d\mu$ converges to the solution of the linear parabolic equation
\begin{equation} \label{eq:deterministic_limit}
	\partial_t \rho - \div(K \nabla \rho) = 0,
\end{equation}
on $\tSet \times \xSet$. This is a diffusive limit in the PDE sense, since the limit equation is a diffusion equation.

In this article, the diffusion limit of \cref{eq:SPDE} is considered simultaneously in the PDE and in the probabilistic sense. The main result, \cref{thm:main_result}, establishes that, under appropriate assumptions, the density $\rho^\epsilon = \int_\vSet f^\epsilon d\mu$ converges in distribution in $C([0,T],H^{-\sigma}(\xSet))$ for any $\sigma > 0$ and in $L^2([0,T],L^2(\xSet))$ to the solution of the stochastic linear diffusion equation
\begin{equation*}
	d\rho = \div (K\nabla \rho)dt + \rho \circ Q^{1/2} dW(t),
\end{equation*}
with $K$ as in \cref{eq:deterministic_limit}. The equation is written in Stratonovitch form and is driven by a cylindrical Wiener process $W$, the covariance operator $Q$ being trace-class. As usual in the context of diffusion limit, the stochastic equation involves a Stratonovitch product. The diffusive limit in the stochastic case has been first proved in \cite{debussche2011diffusion}, under a restrictive condition on the driving random term: $\bar m$ is bounded almost surely. The boundedness of $\bar m$ is a strong assumption, which is not satisfied by an Ornstein-Ulhenbeck process for instance. The contribution of this article is to relax this assumption: we prove the convergence under a moment bound assumption for the driving process.

The main tools of \cite{debussche2011diffusion} are the perturbed test function method and the concept of solution in the martingale sense. Our general strategy for the proof is similar, therefore those tools are also used here. The main novelty is the introduction of stopping times to obtain the estimates required to establish tightness and convergence. Indeed, relaxing the conditions on $\bar m$ implies that moments of the solutions are not controlled (exponential moments for $\bar m$ would be necessary). The strategy from \cite{debussche2011diffusion} needs to be substantially modified: the martingale problem approach is combined with the use of stopping times. At the limit, the stopping times persist, thus the limit processes solves the limit martingale problem only up to a stopping time. We manage to identify them nonetheless as a stopped version of the global solution.

This article is organized as follows: in \cref{sec:main_result}, we set some notation, the assumptions on the driving random term and the main result, \cref{thm:main_result}. \Cref{sec:preliminary_results} states some auxiliary results that are used in the later sections. In \cref{sec:martingale_problems}, we introduce the notion of martingale problem and the perturbed test function method that are used to prove the convergence. In \cref{sec:tightness}, we prove the tightness of the family of processes $\seq{(\rho\stoppedL,\zeta\stoppedL)}{\epsilon}$ stopped at the random time $\tauLe$. \Cref{sec:convergence} takes the limit when $\epsilon \to 0$ in the martingale problems and establishes the convergence of $\rho^\epsilon$ in $C([0,T],H^{-\sigma}(\xSet))$ . In \cref{sec:strong_tightness}, we prove the convergence in a stronger sense, namely in $L^2([0,T],L^2(\xSet))$), using an additional assumption and an averaging lemma.

\section{Assumptions and main result} \label{sec:main_result}

\begin{hypothesis} \label{hyp:L_BGK}
	The operator $L$ is defined on $L^1(V,\mu)$ by \cref{eq:def_interactions}, with $\mathcal M \in L^1(V,\mu)$ such that $\inf_V \mathcal M > 0$ and $\int_V \mathcal Md\mu = 1$.
\end{hypothesis}

Let us define the spaces $\fSet$ and $L^2_x$ and the associated inner products:
\begin{gather*}
	\fSet \doteq L^2(\xSet \times \vSet,dx \mathcal M^{-1}(v) d\mu(v)), \scalf{f}{g} \doteq \int_\xSet \int_\vSet \frac{f(x,v)g(x,v)}{\mathcal M(v)} d\mu(v) dx,\\
	{L^2_x} \doteq L^2(\xSet,dx), \scalrho{f}{g} \doteq \int_\xSet f(x)g(x)dx.
\end{gather*}
We also define the norms $\norm{\cdot}_\fSet$ and $\norm{\cdot}_{L^2_x}$ associated with these inner products.

Note that $L$ is an orthogonal projection in $\fSet$, hence
\begin{equation*}
	\forall f \in \fSet, \norm{Lf}_\fSet \leq \norm{f}_\fSet.
\end{equation*}

\begin{hypothesis} \label{hyp:a1}
	The function $a$ is bounded ($a  \in L^\infty(\vSet,\mu;\mathbb R^d)$), centered for $\mathcal M d\mu$, namely
	\begin{equation} \label{eq:aM_centered}
		\int_\vSet a(v) \mathcal M(v) d\mu(v) = 0,
	\end{equation}
	and the following matrix is symmetric and positive definite
   \begin{equation} \label{eq:def_K}
	   K \doteq \int_\vSet a(v) \otimes a(v) \mathcal M(v) d\mu(v) > 0.
   \end{equation}
\end{hypothesis}

The following assumption is not required to get the convergence in $C([0,T],H^{-\sigma}_x)$ but is used in \cref{sec:strong_tightness} to retrieve a stronger convergence (in $L^2([0,T],L^2_x$). It is exactly the assumption of Theorem 2.3 in \cite{bouchut1999averaging}.
\begin{hypothesis} \label{hyp:a2}
	We have $(V,d\mu) = (\mathbb R^n,\psi(v)dv)$ for some function $\psi \in H^1(\mathbb R^n)$, $a \in \Lip_{loc}(\mathbb R^n;\mathbb R^d)$ and there exists $C \geq 0$ and $\sigma^* \in (0,1]$ such that
	\begin{equation*}
		\forall u \in S^{d-1}, \forall \lambda \in \mathbb R, \forall \delta > 0, \int_{\lambda < \scald{a(v)}{u} < \lambda + \delta} (\abs{\psi(v)}^2 + \abs{\nabla \psi(v)}^2) dv \leq C \delta^{\sigma^*}.
	\end{equation*}
	If $\psi$ is not compactly supported, assume moreover that $\nabla a$ is globally bounded.
\end{hypothesis}

\begin{hypothesis} \label{hyp:f0}
	We have
	\begin{equation}
		\sup_{\epsilon \in (0,1)} \esp{\norm{f^\epsilon_0}_\fSet^{24}} < \infty.
	\end{equation}
	and $\rho^\epsilon_0$ converges in distribution in $L^2_x$ to $\rho_0$.
\end{hypothesis}

\begin{remark}
	The moments of order 24 in \cref{hyp:f0} are useful in \cref{subsec:perturbed_test_functions,subsec:generator_martingale}.
\end{remark}

\subsection{Driving random term} \label{subsec:driving_random_term}

Consider the normed space
\begin{equation*}
	E \doteq C^{2\floor{d/2}+4}(\xSet),
\end{equation*}
where the norm is given by
\begin{equation*}
	\norm{\cdot}_E = \sum_{\abs{\beta} \leq 2\floor{d/2}+4} \sup_{x \in \xSet} \abs{\frac{\partial^{\abs{\beta}} \cdot}{\partial x^\beta}},
\end{equation*}
where $\beta \in \mathbb N^d$, $\abs{\beta} = \sum_{i=1}^d \beta_i$ and
\begin{equation*}
	\frac{\partial^{\abs{\beta}}}{\partial x^\beta} = \frac{\partial^{\abs{\beta}}}{\partial x_1^{\beta_1} ... \partial x_d^{\beta_d}}.
\end{equation*}

\begin{hypothesis} \label{hyp:stationarity} \label{hyp:m_centered}
	The family of process $\seq{m(\cdot,n)}{n \in E}$ is a $E$-valued, càdlàg, stochastically continuous and homogeneous Markov process with initial condition $m(0,n) = n$. It admits a unique centered stationary distribution $\nu$
	\begin{equation*}
		\int_E \norm{n}_E d\nu(n) < \infty \mbox{ and } \int_E n d\nu(n) = 0.
	\end{equation*}
	The driving process $\bar m$ is the stationnary Markov process associated with $\seq{m(\cdot,n)}{n \in E}$, meaning that for all $t \in \tSet$, the distribution of $\bar m(t)$ is $\nu$. It is adapted to a filtration $\seq{\mathcal F_t}{t \in \tSet}$ satisfying the usual conditions (complete and right-continuous).
\end{hypothesis}

For $\theta \in L^1(E) \doteq L^1(E,\nu)$, set
\begin{equation*}
	\lrangle{\theta} \doteq \int_E \theta d\nu.
\end{equation*}

Note that most of the arguments below only require $\bar m(t) \in C^1(\xSet)$. However, in \cref{sec:tightness}, we use the compact embedding $H^{\floor{d/2}+2}(\xSet) \subset C^1(\xSet)$ and in \cref{subsec:covariance_operator}, we need $\bar m(t) \in C^{2s}(\xSet)$ with $H^s(\xSet) \subset C^1(\xSet)$, hence $s = \floor{d/2}+2$.

\begin{definition} \label{def:m_epsilon}
For $\epsilon > 0$, the random process $\bar m^\epsilon$ is defined by \cref{eq:def_m_epsilon}
where $\bar m$ is defined by \cref{hyp:stationarity}. Let $\mathcal F^\epsilon_t = \mathcal F_{t/\epsilon^2}$ so that $\bar m^\epsilon$ is adapted to the filtration $\seq{\mathcal F^\epsilon_t}{t \in \tSet}$.
\end{definition}

\subsubsection{Assumption on moments}

From now on, we depart from the setting of \cite{debussche2011diffusion}. In the previous works \cite{debussche2011diffusion,debussche2017diffusion}, it is assumed that there exists $C_* \in \mathbb R^+$ such that, almost surely,
\begin{equation*}
	\forall t \in \tSet, \norm{\bar m(t)}_E \leq C_*.
\end{equation*}
The main novelty of this article is that we relax this assumption into \cref{hyp:m_L2_all_starting_point,hyp:m_moments} concerning moments.
\begin{hypothesis} \label{hyp:m_moments}
	There exists $\gamma \in (4,\infty)$ such that
	\begin{equation*} \label{eq:hyp:m_moments}
		\esp{\sup_{t \in [0,1]} \norm{\bar m(t)}_E^\gamma} < \infty.
	\end{equation*}
\end{hypothesis}

The condition $\gamma > 4$ is required below in \cref{hyp:m_L2_all_starting_point}, where we also assume that the moments on $m(t,n)$ depend polynomially on $n$.
\begin{hypothesis} \label{hyp:m_L2_all_starting_point}
	There exists $b \in [0,\frac \gamma 2 - 2)$ such that
	\begin{equation*}
		\sup_{n \in E} \sup_{t \in \tSet} \frac{\esp{\norm{m(t,n)}_E^2}^{\frac 1 2}}{1 +  \norm{n}_E^b} < \infty,
	\end{equation*}
	and such that $\nu$ has a finite $8(b+2)$-order moment, namely
	\begin{equation*}
		\int_E \norm{n}_E^{8(b+2)} d\nu(n) < \infty.
	\end{equation*}
\end{hypothesis}

For instance, if $m$ is an Ornstein-Uhlenbeck process
\begin{equation*}
dm(t) = -\theta m(t) dt + \sigma dW(t),
\end{equation*}
with $W$ a $E$-valued Wiener process, then $m$ satisfies \cref{hyp:m_moments,hyp:m_L2_all_starting_point}.

Moreover, any process satisfying the boundedness assumption in \cite{debussche2011diffusion} also satisfies \cref{hyp:m_L2_all_starting_point,hyp:m_moments}.

\subsubsection{Mixing property}

\begin{hypothesis}[Mixing property] \label{hyp:mixing}
	There exists a nonnegative integrable function $\gamma_{mix} \in L^1(\tSet)$ such that, for all $n_1, n_2 \in E$, there exists a coupling $(m^*(\cdot,n_1),m^*(\cdot,n_2))$ of $(m(\cdot,n_1),m(\cdot,n_2))$ such that
	\begin{equation*}
		\forall t \in \tSet, \esp{\norm{m^*(t,n_1) - m^*(t,n_2)}_E^2}^{1/2} \leq \gamma_{mix}(t) \norm{n_1-n_2}_E.
	\end{equation*}
\end{hypothesis}
Typically, $\gamma_{mix}$ is expected to be of the form $\gamma_{mix}(t) = C_{mix} e^{-\beta_{mix} t}$ for some $\beta_{mix} >0$. In the example where $m$ is an Ornstein-Uhlenbeck process, consider $m^*(\cdot,n_1)$ and $m^*(\cdot,n_2)$ driven by the same Wiener process $W$. Owing to Gronwall's Lemma, it is straightforward to prove that this coupling satisfies \cref{hyp:mixing} and that $\gamma_{mix}$ decays exponentially fast.

We also need \cref{hyp:B_continuity,hyp:B_subpolynomial} concerning the transition semi-group associated to the homogeneous Markov process $\seq{m(\cdot,n)}{n \in E}$. Since those assumptions are quite technical, we postpone their statement in \cref{subsubsec:resolvent}.

\subsection{Main result}

For $x,y \in \xSet$, define the kernel
\begin{equation} \label{eq:def_k}
	k(x,y) = \esp{\int_{\mathbb R} \bar m(0)(x) \bar m(t)(y) dt},
\end{equation}
and for $f \in L^2_x$ and $x \in \xSet$, let us recall from \cite{debussche2011diffusion}
\begin{equation} \label{eq:def_Q}
	Qf(x) = \int_\xSet k(x,y)f(y)dy.
\end{equation}

\begin{theorem} \label{thm:main_result}
	Let \cref{hyp:L_BGK,hyp:a1,hyp:f0,hyp:stationarity,hyp:m_centered,hyp:m_moments,hyp:mixing,hyp:m_L2_all_starting_point,hyp:B_continuity,hyp:B_subpolynomial} be satisfied. 
	Let $W$ be a cylindrical Wiener process on $L^2_x$, $\rho_0$ be a random variable in $L^2_x$ and $\rho$ be the weak solution of the linear stochastic diffusion equation
	\begin{equation} \label{eq:limit_equation_strato}
		d\rho = \div (K\nabla \rho)dt + \rho Q^{1/2} \circ dW(t),
	\end{equation}
	with initial condition $\rho(0) = \rho_0$, in the sense of \cref{def:limit_equation}. Also assume that $\rho^\epsilon(0)$ converges in distribution to $\rho_0$ in $L^2_x$. Then, for all $\sigma > 0$ and $T > 0$, the density $\rho^\epsilon$ converges in distribution in $C([0,T],H^{-\sigma}_x)$ to $\rho$.

	Let \cref{hyp:L_BGK,hyp:a1,hyp:a2,hyp:f0,hyp:stationarity,hyp:m_centered,hyp:m_moments,hyp:mixing,hyp:m_L2_all_starting_point,hyp:B_continuity,hyp:B_subpolynomial} be satisfied. Then $\rho^\epsilon$ also converges in distribution in $L^2([0,T],L^2_x)$ to $\rho$.
\end{theorem}

The noise in \cref{eq:limit_equation_strato} involves a Stratonovitch product, which is usual in the context of diffusion limit. Written with a Itô product, the limit becomes
\begin{equation} \label{eq:limit_equation_rho}
	d\rho = \div (K\nabla \rho)dt + \frac 1 2 F \rho dt + \rho Q^{1/2} dW(t),
\end{equation}
where $F$ is the trace of $Q$, namely
\begin{equation} \label{eq:def_F}
	F(x) = k(x,x).
\end{equation}
This equation is well-posed, as discussed after \cref{def:limit_equation}.

This is the same limit than in \cite{debussche2011diffusion}. Compared with \cite{debussche2011diffusion}, we obtain a stronger convergence result, namely a convergence in $L^2([0,T],L^2_x)$ under additional assumptions.

\subsection{Strategy of the proof of \cref{thm:main_result}}

A standard strategy to prove the convergence of $\rho^\epsilon$ when $\epsilon \to 0$ (see \cite{debussche2011diffusion,debussche2017diffusion,debussche2020diffusion}) is first to establish the tightness of the family $\seq{\rho^\epsilon}{\epsilon>0}$, and then the uniqueness of the limit point of this family and solves \cref{eq:limit_equation}. The tightness usually comes from estimates on moments of trajectories. It is the case in \cite{debussche2011diffusion}, where the boundedness of $\bar m$ is used to get an estimate on $\esp{\sup_{t \in [0,T]} \norm{f^\epsilon(t)}_{L^2}^p}$ for all $T > 0$ and $p \geq 1$. However, without an almost sure bound on $\bar m$, we do not manage to get this estimate. Instead, we introduce a stopping time $\tauLe$ depending on a parameter $\Lambda$ such that the estimate holds for $f\stoppedL \doteq f^\epsilon(t \wedge \tauLe)$.

More precisely, define a first stopping time
\begin{equation} \label{eq:def_taue}
	\taue \doteq \inf \set{t \in \tSet \mid \norm{\bar m^\epsilon(t)}_E > \epsilon^{-\alpha}},
\end{equation}
for some parameter $\alpha$. Let $C^1_x \doteq C^1(\xSet)$ and define the hitting time of a threshold $\Lambda$ by $z \in C([0,T],C^1_x)$
\begin{equation} \label{eq:def_tauLze}
	\tau_\Lambda(z) \doteq \inf\set{t \in \tSet \mid \norm{z(t)}_{C^1_x} \geq \Lambda}.
\end{equation}
Then, define the auxiliary process
\begin{equation} \label{eq:def_zeta_epsilon}
	\zeta^\epsilon(t) = \frac {1} \epsilon \int_0^t \bar m^\epsilon(s) ds = \epsilon \int_0^{t/\epsilon^2} \bar m(s) ds \in E \subset C^1_x.
\end{equation}
Observe that $\frac 1 \epsilon \bar m^\epsilon = \partial_t \zeta^\epsilon$.

We can now define
\begin{equation} \label{eq:def_tauLe}
	\tauLe \doteq \taue \wedge \tauLze.
\end{equation}
The times $\taue$ and $\tauLze$ have different asymptotic behaviors. On the one hand, \cref{lem:taue_to_infty} states that $\taue \to \infty$ in probability. On the other hand, \cref{subsec:convergence_of_the_auxiliary_process} establishes that $\zeta^\epsilon$ converges in distribution, when $\epsilon \to 0$, to a Wiener process $\zeta$. Thus, we prove that, for all $\Lambda$ outside of a countable set, $\tauLze$ converges in distribution to $\tau_\Lambda(\zeta)$. Hence, $\tauLe$ converges in distribution to $\tau_\Lambda(\zeta)$.

In \cref{subsec:bound_L2}, we prove an estimate on $f\stoppedL$ depending only on $T$, $\Lambda$ and $f^\epsilon_0$. This estimate leads to prove the tightness of the family of stopped processes $\seq{\rho\stoppedL}{\epsilon>0}$. Then, we identify the limit points of this family using the notions of martingale problems and perturbed test functions, and we deduce the convergence of the stopped process to a limit $\rho_\Lambda$.

Since $\taue \to \tau_\Lambda(\zeta)$ and we expect $\rho^\epsilon \to \rho$, it is convenient to study the process $(\rho^\epsilon,\zeta^\epsilon)$ to be able to write the limit of $\rho\stoppedL$ as $\rho^{\tau_\Lambda(\zeta)}$. Moreover, to prove that $\rho^\epsilon$ indeed converges to $\rho$ and that $\rho$ satisfies \cref{eq:limit_equation_strato}, we need $\tau_\Lambda(\zeta)$ to be a stopping time for the limit process. Thus, we need to consider the convergence in distribution of the couple $(\rho^\epsilon,\zeta^\epsilon)$ in $C([0,T],H^{-\sigma}_x) \times C([0,T],C^1_x)$ to the solution $(\rho,\zeta)$ of
\begin{equation} \label{eq:limit_equation}
	\left\{
	\begin{split}
	d\rho &= \div (K\nabla \rho)dt + \frac 1 2 F \rho dt + \rho Q^{1/2} dW(t)\\
	d\zeta &= Q^{1/2} dW(t),
	\end{split}
	\right.
\end{equation}
with initial condition $\rho(0) = \rho_0$ and $\zeta(0) = 0$. In this framework, $\tauLze$ is a stopping time for $(\rho^\epsilon,\zeta^\epsilon)$ and $\tau_\Lambda(\zeta)$ is a stopping time for the limit $(\rho,\zeta)$.

We first state in \cref{sec:preliminary_results} some consequences of our assumptions in \cref{subsec:driving_random_term} and introduce the stopping times. In \cref{sec:martingale_problems}, we define the martingale problem solved by the process $(\rho^\epsilon,\zeta^\epsilon)$ and set up the perturbed test functions strategy.

In \cref{sec:tightness}, we prove the tightness of the stopped process in $C([0,T],H^{-\sigma}_x) \times C([0,T],C^1_x)$, using the perturbed test functions of \cref{sec:martingale_problems}. Then, in \cref{sec:convergence}, we establish the convergence of the martingale problems when $\epsilon \to 0$ to identify the limit as a solution of a stopped martingale problem, and deduce the convergence of the original process $(\rho^\epsilon,\zeta^\epsilon)$ in $C([0,T],H^{-\sigma}_x) \times C([0,T],C^1_x)$.

In \cref{sec:strong_tightness}, we prove the tightness of the stopped process in $L^2([0,T],L^2_x)$ under the assumptions of \cref{thm:main_result}, using an averaging lemma. Combined with the previous results, we deduce the convergence of the original process $\rho^\epsilon$ in $L^2([0,T],L^2_x)$.

\section{Preliminary results} \label{sec:preliminary_results}

\subsection{Resolvent operator} \label{subsec:pseudolin_pseudoquad}

\subsubsection{Additional assumptions} \label{subsubsec:resolvent}

Denote by $\seq{P_t}{t \in \tSet}$ the transition semi-group on $E$ associated to the homogeneous Markov process $\bar m$ and let $B$ denote its infinitesimal generator
\begin{equation*}
	\forall n \in E, B\theta(n) = \lim_{t \to 0} \frac{P_t \theta(n) - \theta(n)}{t}.
\end{equation*}

The usual framework for Markov processes and their transition semi-groups is to consider continuous bounded test functions $\theta \in C_b(E)$, so that $P_t \theta$ is a contraction semi-group (see \cite{ethier1986markov}). Here, we need unbounded test functions (see \cref{subsec:perturbed_test_functions}), thus consider the action of the semi-group on $C(E) \cap L^1(E)$.
We also consider the domain of $B$
\begin{equation*}
	D(B) \doteq \set{\theta \in C(E) \cap L^1(E) \mid \forall n \in E, B\theta(n) \mbox{ exists and } B\theta \in C(E) \cap L^1(E)}.
\end{equation*}

We need a continuity property for the semi-group $\seq{P_t}{t \in \tSet}$. Define first the resolvent operator.
\begin{definition} \label{def:resolvent}
	For $\lambda \in [0,\infty)$ and $\theta \in C(E) \cap L^1(E)$ such that $\int_0^\infty \abs{P_t \theta(n)} dt < \infty$ for all $n \in E$, define the resolvent: for all $n \in E$
	\begin{equation*}
		R_\lambda \theta(n) \doteq \int_0^\infty e^{-\lambda t} P_t \theta (n) dt.
	\end{equation*}
\end{definition}

\begin{hypothesis} \label{hyp:B_continuity}
	The family $\seq{P_t}{t \in \tSet}$ is a semi-group on $C(E) \cap L^1(E)$. Moreover, for all $\seq{\lambda_i}{1 \leq i \leq 4} \in [0,\infty)^4$ and $\seq{\theta_i}{1 \leq i \leq 4} \in C(E) \cap L^1(E)^4$ such that $R_{\lambda_i} \theta_i$ are well-defined by \cref{def:resolvent}, we have
	\begin{equation*}
		\forall j \in \llbracket 1 , 4 \rrbracket, \Pi_{i=1}^{j} R_{\lambda_i} \theta_i \in D(B).
	\end{equation*}
	In addition, we assume that for $\lambda \in [0,\infty)$ and $\theta$ such that $R_\lambda \theta \in D(B)$, the commutation formula holds
	\begin{equation*}
		B \int_0^\infty e^{-\lambda t} P_t \theta(\cdot) dt = \int_0^\infty e^{-\lambda t} B P_t \theta(\cdot) dt.
	\end{equation*}
\end{hypothesis}

The second part of \cref{hyp:B_continuity} is satisfied under a continuity property for the semi-group $\seq{P_t}{t \in \tSet}$. Indeed, consider the following computations
\begin{multline*}
	\lim_{s \to 0} \frac{P_s - \id}{s} \int_0^\infty e^{-\lambda t} P_t \theta(\cdot) dt = \lim_{s \to 0} \int_0^\infty e^{-\lambda t} \frac{P_{s+t} - P_t}{s} \theta(\cdot) dt\\
	= \int_0^\infty e^{-\lambda t} \lim_{s \to 0} \frac{P_{s+t} - P_t}{s} \theta(\cdot) dt
\end{multline*}
To justify the first equality, it is sufficient to assume point-wise continuity of $P_t$ for all $t$ on the space $C(E) \cap L^1(E)$. The second equality is a consequence of the bounded convergence theorem.

Note that by means of \cref{hyp:B_continuity}, $-R_0$ is the inverse of $B$. Indeed, for $\theta$ such that $R_0 \theta \in D(B)$, we have
\begin{equation*}
	B \int_0^\infty P_t \theta(\cdot)dt = \int_0^\infty B P_t \theta(\cdot)dt = \int_0^\infty \partial_t P_t \theta(\cdot) dt = - \theta.
\end{equation*}

We sometimes use functions having at most polynomial growth. Our last assumption is that $B$ preserves this property.
\begin{hypothesis} \label{hyp:B_subpolynomial}
	If $\theta \in D(B)$ has at most polynomial growth, then $B\theta$ has at most polynomial growth with the same degree. Namely, there exists $C_B \in (0,\infty)$ such that, for any $\theta \in D(B)$ and $k \in \mathbb N$,
	\begin{equation*}
		\sup_{n \in E} \frac{\abs{B\theta(n)}}{1 + \norm{n}_E^k} \leq C_B \sup_{n \in E} \frac{\abs{\theta(n)}}{1 + \norm{n}_E^k}
	\end{equation*}
\end{hypothesis}

\subsubsection{Results on the resolvent operator} \label{subsubsec:pseudolin_pseudoquad}

We introduce a class of pseudo-linear (respectively pseudo-quadratic) functions, which behave like linear (respectively quadratic) functions for our purposes.
\begin{definition} \label{def:pseudo-linear_pseudo-quadratic}
	A function $\theta \in \Lip(E)$ such that $\lrangle{\theta} = 0$, is called pseudo-linear. Denote by $\lrbracket{\theta}_{\Lip}$ its Lipschitz constant.

	A function $\theta : E \to \mathbb R$ is called pseudo-quadratic if there exists a function $b_\theta : E^2 \to \mathbb R$ satisfying
	\begin{itemize}
		\item for all $n \in E$, $\theta(n) = b_\theta(n,n)$,
		\item for all $n \in E$, $b_\theta(n,\cdot)$ and $b_\theta(\cdot,n)$ are Lipschitz continuous,
		\item the mappings $n \mapsto \lrbracket{b_\theta(n,\cdot)}_{\Lip}$ and $n \mapsto \lrbracket{b_\theta(\cdot,n)}_{\Lip}$ have at most linear growth.
	\end{itemize}

	If $\theta$ is a pseudo-quadratic function, then let
	\begin{equation*}
		\lrbracket{\theta}_{\psq} \doteq \sup_{n_1 \neq n_2 \in E} \frac{\abs{\theta(n_2) - \theta(n_1)}}{(1 + \norm{n_1}_E + \norm{n_2}_E) \norm{n_2-n_1}_E} < \infty.
	\end{equation*}
\end{definition}

Let $E^*$ denote the dual space of $E$. Any element $\theta \in E^*$ is pseudo-linear.

A consequence of the mixing property (\cref{hyp:mixing}) is that the pseudo-linear and the pseudo-quadratic functions introduced in \cref{def:pseudo-linear_pseudo-quadratic} satisfy the conditions of \cref{def:resolvent}.
\begin{lemma} \label{lem:Pt_lip_Pt_quad}
	Let $\theta$ be a pseudo-linear function. Then, for all $\lambda \geq 0$, $R_\lambda \theta$ is well-defined and is pseudo-linear. Moreover, let
	\begin{equation*}
	    C_\lambda = \int_0^\infty e^{-\lambda t} \gamma_{mix}(t) dt \mbox{ and } C'_\lambda = \parths{1 \vee \int \norm{n_2}_E d\nu(n_2)} C_\lambda.
	\end{equation*}
	Then, we have
	\begin{equation*}
		\lrbracket{R_\lambda \theta}_{\Lip} \leq \lrbracket{\theta}_{\Lip} C_\lambda,
	\end{equation*}
	and for $n \in E$,
	\begin{equation} \label{eq:lem:Pt_lip_Pt_quad}
		\abs{R_\lambda \theta(n)} \leq C'_\lambda \lrbracket{\theta}_{\Lip} (1 + \norm{n}_E).
	\end{equation}

	Let $\theta$ be a pseudo-quadratic function. Then, for $\lambda \geq 0$, $R_\lambda\lrbracket{\theta - \lrangle{\theta}}$ is well-defined. Moreover, there exists $C''_\lambda \in (0,\infty)$ depending only on $C_\lambda$ and $b$ such that, for $n \in E$,
	\begin{equation*}
		\abs{R_\lambda\lrbracket{\theta - \lrangle{\theta}}(n)} \leq C''_\lambda \lrbracket{\theta}_{\psq} (1 + \norm{n}_E^{b+1})
	\end{equation*}
	where $b$ is defined in \cref{hyp:m_L2_all_starting_point}.
\end{lemma}

\begin{proof}
	Let $n_1, n_2 \in E$ and denote by $(m^*(\cdot,n_1),m^*(\cdot,n_2))$ the coupling introduced in \cref{hyp:mixing}. If $\theta$ is Lipschitz continuous, then for all $t \in \tSet$, \cref{hyp:mixing} leads to
	% \begin{align} % garder quelque part ces calculs pour la rédaction de la thèse
	% 	\abs{P_t\theta(n_1) - P_t\theta(n_2)}
	% 		%&\leq \abs{\esp{\theta(m^*(t,n_1))} - \esp{\theta(m^*(t,n_2))}} \nonumber\\
	% 		%&\leq \esp{\abs{\theta(m^*(t,n_1)) - \theta(m^*(t,n_2))}} \nonumber\\
	% 		%&\leq \lrbracket{\theta}_{\Lip} \esp{\norm{m^*(t,n_1) - m^*(t,n_2)}_E} \nonumber\\
	% 		%&\leq \lrbracket{\theta}_{\Lip} \esp{\norm{m^*(t,n_1) - m^*(t,n_2)}_E^2}^{1/2} \nonumber\\
	% 		&\leq \lrbracket{\theta}_{\Lip} \gamma_{mix}(t) \norm{n_1 - n_2}_E, \label{eq1:lem:Pt_lip_Pt_quad}
	% \end{align}
	\begin{equation} \label{eq1:lem:Pt_lip_Pt_quad}
		\abs{P_t\theta(n_1) - P_t\theta(n_2)} \leq \lrbracket{\theta}_{\Lip} \gamma_{mix}(t) \norm{n_1 - n_2}_E.
	\end{equation}
	Recall that $P_t$ is $\nu$-invariant, i.e. $\nu P_t = \nu$, hence we have
	\begin{align} % garder quelque part ces calculs pour la rédaction de la thèse
		\abs{P_t\theta(n_1) - \lrangle \theta}
			%&= \abs{\int_E \parths{P_t\theta(n_1) - \theta(n_2)} d\nu(n_2)} \nonumber\\
			&= \abs{\int_E \parths{P_t\theta(n_1) - P_t\theta(n_2)} d\nu(n_2)} \nonumber\\
			&\leq \gamma_{mix}(t) \lrbracket{\theta}_{\Lip} \int_E \norm{n_1 - n_2}_E d\nu(n_2) \nonumber\\
			%&\leq \gamma_{mix}(t) \lrbracket{\theta}_{\Lip} (\int_E \norm{n_2}_E d\nu(n_2) + \norm{n_1}_E) \nonumber\\
			&\leq \parths{1 \vee \int \norm{n_2}_E d\nu(n_2)} \gamma_{mix}(t) \lrbracket{\theta}_{\Lip} (1 + \norm{n_1}_E). \label{eq2:lem:Pt_lip_Pt_quad}
	\end{align}

	Assume that $\theta$ is pseudo-linear. Since $\lrangle{\theta} = 0$, \cref{eq2:lem:Pt_lip_Pt_quad} implies that $R_\lambda \theta$ is well-defined for all $\lambda \in [0,\infty)$, and \cref{eq1:lem:Pt_lip_Pt_quad} implies that $R_\lambda \theta$ is $\parths{\lrbracket{\theta}_{\Lip} C_\lambda}$-Lipschitz continuous. Moreover, by means of Fubini's Theorem,
	\begin{align*}
		\int_E R_\lambda \theta(n) d\nu(n)
			&= \int_E \int_0^\infty e^{-\lambda t} P_t \theta(n) dt d\nu(n)\\
			&= \int_0^\infty e^{-\lambda t} \int_E P_t \theta(n) d\nu(n) dt\\
			&= \int_0^\infty e^{-\lambda t} \int_E \theta(n) d\nu(n) dt\\
			&= 0.
	\end{align*}
	This concludes the proof that $R_\lambda \theta$ is pseudo-linear. Finally, \cref{eq:lem:Pt_lip_Pt_quad} is a straightforward consequence of \cref{eq2:lem:Pt_lip_Pt_quad}.

	Now assume that $\theta$ is a pseudo-quadratic function: using \cref{hyp:m_L2_all_starting_point,hyp:mixing} and Cauchy-Schwarz inequality, we have for all $n_1, n_2 \in E$
	\begin{equation*}
		\abs{P_t\theta(n_1) - P_t\theta(n_2)} \leq C \lrbracket \theta_{\psq} \parths{1 + \norm{n_1}_E^b + \norm{n_2}_E^b} \gamma_{mix}(t) \norm{n_1 - n_2}_E,
	\end{equation*}
	for some constant $C$ depending on $b$. Since $P_t$ is $\nu$-invariant and $\nu$ has a finite moment of order $b+1$ by \cref{hyp:m_L2_all_starting_point}, we get
	\begin{equation*}
		\abs{P_t\theta(n) - \lrangle \theta} \leq C' \lrbracket \theta_{\psq} \gamma_{mix}(t) (1 + \norm n_E^{b+1}),
	\end{equation*}
	for some constant $C'$ depending on $b$. Integrating with respect to $t$ gives the announced result.
\end{proof}

Let us recall notation from \cite{debussche2011diffusion}. For $\rho,\rho' \in {L^2_x}$, denote by $\psi_{\rho,\rho'} \in E^*$ the continuous linear form
\begin{equation*}
	\forall n \in E, \psi_{\rho,\rho'}(n) \doteq \scalrho{\rho n}{\rho'}
\end{equation*}
The linear form $\psi_{\rho,\rho'}$ is pseudo-linear and
\begin{equation*}
	\lrbracket{\psi_{\rho,\rho'}}_{\Lip} = \norm{\psi_{\rho,\rho'}}_{E^*} \leq \norm{\rho}_{L^2_x} \norm{\rho'}_{L^2_x}.
\end{equation*}
Hence, by \cref{lem:Pt_lip_Pt_quad}, we have for $\rho, \rho' \in {L^2_x}$, $\lambda \geq 0$ and $n \in E$
\begin{equation*}
	\abs{R_\lambda \psi_{\rho,\rho'}(n)} \leq C'_\lambda \norm{\rho}_{L^2_x} \norm{\rho'}_{L^2_x} (1 + \norm{n}_E).
\end{equation*}

Thus, for all $n \in E$, $(\rho,\rho') \mapsto R_\lambda \psi_{\rho,\rho'}(n)$ is a continuous bilinear form on $L^2_x$. By means of Riesz Representation Theorem, there exists a continuous linear map $\tilde R_\lambda(n) : {L^2_x} \to {L^2_x}$ such that
\begin{equation*}
	\forall \rho,\rho' \in {L^2_x}, \forall n \in E, R_\lambda \psi_{\rho,\rho'}(n) = \scalrho{\tilde R_\lambda(n)\rho}{\rho'}.
\end{equation*}
By a slight abuse of notation, denote $R_\lambda(n) = \tilde R_\lambda(n)$. For $\phi \in C^1(L^2_x)$, the linear mapping $D\phi(\rho)$ can also be identified as an element of $L^2_x$:
\begin{equation*}
	\forall \rho, h \in L^2_x, D\phi(\rho)(h) = \scalrho{h}{D\phi(\rho)},
\end{equation*}
so that we can define $D\phi(\rho)(R_\lambda(n)h)$ for $\rho,h \in L^2_x$.

Now consider $\rho$ and $\rho'$ in dual Sobolev spaces $H^k_x$ and $H^{-k}_x$ (for $k \in \mathbb N$ such that $E \subset C^{k}_x$, namely $k \leq 2 \floor{d/2} +2$). We also may define $R_\lambda(n) : H^k_x \to H^k_x$ in a compatible way.

% For $k \in \llbracket 0, 2 \floor{d/2} + 2 \rrbracket$, and for $\rho \in H^k_x$, Riesz's Representation Theorem also yields
% \begin{equation*}
% 	\norm{R_\lambda(n)\rho}_{H^k_x} = \norm{R_\lambda \psi_{\rho,\cdot}(n)}_{H^{k}_x},
% \end{equation*}
% thus owing to \cref{lem:Pt_lip_Pt_quad}, we get
% \begin{equation} \label{eq:ctrl_resolvent}
% 	\norm{R_\lambda(n)\rho}_{H^k_x} \leq C'_\lambda \norm{\rho}_{H^k_x} (1 + \norm{n}_E).
% \end{equation}

\subsection{Properties of the covariance operator} \label{subsec:covariance_operator}

Recall that $k$, $F$ and $Q$ are defined by equations \cref{eq:def_k,eq:def_F,eq:def_Q}.

\begin{lemma} \label{lem:kFQ}
	The kernel $k$ is symmetric and in $L^\infty(\xSet \times \xSet)$. Moreover, $Q$ is a bounded, self-adjoint, compact and non-negative operator on $L^2_x$.
\end{lemma}

\begin{proof}
	Since $\bar m$ is stationary, we have the identity%we write \cref{eq:def_k} as in \cite{debussche2011diffusion}
	% \begin{align*}
	% 	k(x,y)
	% 		&= \esp{\int_{\tSet} \bar m(0)(x) \bar m(t)(y) dt} + \esp{\int_{\mathbb R_-} \bar m(0)(x) \bar m(t)(y) dt}\\
	% 		&= \esp{\int_{\tSet} \bar m(0)(x) \bar m(t)(y) dt} + \esp{\int_{\mathbb R_-} \bar m(-t)(x) \bar m(0)(y) dt}\\
	% 		&= \esp{\int_{\tSet} \bar m(0)(x) \bar m(t)(y) dt} + \esp{\int_{\tSet} \bar m(t)(x) \bar m(0)(y) dt}\\
	% 		&= \int_E \psi_x(n) R_0 \psi_y(n) d\nu(n) + \int_E R_0 \psi_x(n) \psi_y(n) d\nu(n),
	% \end{align*}
	\begin{equation} \label{eq:expr_k}
		k(x,y) = \int_E \psi_x(n) R_0 \psi_y(n) d\nu(n) + \int_E R_0 \psi_x(n) \psi_y(n) d\nu(n),
	\end{equation}
	with $\psi_x(n) = n(x)$ for all $n \in E$ and $x \in \xSet$. The functions $\psi_x$ and $\psi_y$ are continuous linear forms, thus we have
	\begin{equation*}
		\sup_{x,y \in \xSet, n \in E} \frac{\abs{\psi_x R_0 \psi_y(n)}}{1 + \norm{n}_E^2} < \infty.
	\end{equation*}
	Owing to \cref{hyp:m_moments}, $\int_E \norm{n}_E^2 d\nu(n) < \infty$, thus $k$ is well-defined, bounded and symmetric. It implies that $Q$ is a bounded operator on $L^2_x$ and is self-adjoint and compact (see for instance \cite{dunford1964linear} XI.6).

	The proof of non-negativity of $Q$ is given in \cite{debussche2011diffusion}.
\end{proof}

By means of \cref{lem:kFQ}, the operator $Q^{1/2}$ can be defined ($L^2_x \to L^2_x$). Note that $Q$ is trace-class, that $Q^{1/2}$ is Hilbert-Schmidt on $L^2_x$ and that
\begin{equation*}
	\norm{Q^{1/2}}_{\mathcal L_2}^2 = \Tr Q = \int_{\mathbb T^d} F(x) dx.
\end{equation*}

Let $\seq{F_i}{i}$ be a orthonormal and complete system of eigenvectors of $Q^{1/2}$, and $\seq{\sqrt{q_i}}{i}$ their associated eigenvalues.
\begin{lemma} \label{lem:qiFi_summable}
	For all $i$, $F_i \in C^1_x$ and
	\begin{equation*}
		\sum_i q_i \norm{F_i}_{C^1_x}^2 < \infty.
	\end{equation*}
\end{lemma}
\begin{proof}
	Let $s = \floor{\frac d 2} + 2$, so that we have the continuous embeddings $H^s_x \subset C^1_x$. It is thus sufficient to prove that $\sum_i q_i \norm{F_i}_{H^s_x}^2 < \infty$.

	Since $\bar m(t) \in E = C^{2s}_x$ and is mixing, is it straightforward to prove that $k \in H^{2s}_x$ using a differentiation under the integral sign in \cref{eq:def_k}.

	For $\abs \beta \leq s$, we multiply the identity $q_i F_i = Q F_i$ by $\frac{\partial^{2\abs{\beta}} F_i}{\partial x^{2\beta}}$ and integrate by parts both sides of the equality to get
	\begin{equation*}
		(-1)^{\abs \beta} q_i \norm{\frac{\partial^{\abs{\beta}} F_i}{\partial x^\beta}}_{L^2_x}^2 = \int\int \frac{\partial^{2\abs{\beta}} k(x,y)}{\partial x^{2\beta}} F_i(x) F_i(y) dx dy.
	\end{equation*}
	Using \cref{eq:expr_k}, we have
	\begin{multline*}
		\int_\xSet \int_\xSet \frac{\partial^{2\abs{\beta}} k(x,y)}{\partial x^{2\beta}} F_i(x) F_i(y) dx dy =\\
		\begin{split}
			&\int_E \int_\xSet \frac{\partial^{2\abs{\beta}} n}{\partial x^{2\beta}}(x) F_i(x) dx \int_\xSet R_0 \psi_y(n) F_i(y) dy d\nu(n)\\
			&+ \int_E \int_\xSet \frac{\partial^{2\abs{\beta}} (R_0 \psi_x(n))}{\partial x^{2\beta}} F_i(x) dx \int_\xSet n(y) F_i(y) dy d\nu(n).
		\end{split}
	\end{multline*}
	We sum with respect to $i$, use the Parseval identity and the Cauchy-Schwarz inequality to get
	\begin{align*}
		\sum_i q_i \norm{\frac{\partial^{\abs{\beta}} F_i}{\partial x^\beta}}_{L^2_x}^2
			&\leq 2 \int_E \norm{R_0 \psi_x(n)}_{H^{2 \abs \beta}_x} \norm{n}_{H^{2 \abs \beta}_x} d\nu(n)\\
			&\leq C \int_E (1+\norm{n}_E^2) d\nu(n),
	\end{align*}
	for some constant $C$, using \cref{lem:Pt_lip_Pt_quad}. This upper bound is finite by \cref{hyp:m_L2_all_starting_point}. Summing with respect to $\beta$ concludes the proof.
\end{proof}

\subsection{Behavior of the stopping time for the driving process}

Recall that $\taue$ is defined by \cref{eq:def_taue}: $\taue \doteq \inf \set{t \in \tSet \mid \norm{\bar m^\epsilon(t)}_E > \epsilon^{-\alpha}}$.

In this section, we establish \cref{lem:taue_to_infty} below. Its objectives are twofold. On the one hand, it shows that $\bar m^\epsilon$ is almost surely bounded on any interval $[0,T]$, which is useful to justify the well-posedness of \cref{eq:SPDE}. On the other hand, it gives us an estimate for $\epsilon^\alpha \norm{\bar m\stopped(t)}_E$ uniform in $t$ and $\epsilon$ for some small $\alpha$. This estimate will prove useful for \cref{subsec:generator_martingale,sec:tightness,sec:convergence}. Therefore, it is a key result of this paper.

\begin{lemma} \label{lem:taue_to_infty}
	Let \cref{hyp:m_moments,hyp:stationarity} be satisfied and let $T>0$. Then almost surely
	\begin{equation*}
		\sup_{t \in [0,T]} \norm{\bar m^\epsilon(t)}_E < \infty.
	\end{equation*}
	Moreover, let $\alpha > \frac 2 \gamma$ and define the $\seq{\mathcal F^\epsilon_t}t$-stopping time $\taue$ by \cref{eq:def_taue}.	Then, we have
	\begin{equation*}
		\forall T>0, \proba{\taue < T} \xrightarrow[\epsilon \to 0]{} 0.
	\end{equation*}
\end{lemma}

\begin{remark}
	Equation \cref{eq:def_taue} implies that for $t \in \tSet$,
	\begin{equation} \label{eq:taue_bounds_m}
		\norm{\bar m\stopped(t)}_E = \norm{\bar m^\epsilon(t \wedge \taue)}_E \leq \epsilon^{-\alpha} \vee \norm{\bar m(0)}_E.
	\end{equation}
	In particular, on the event $\set{\norm{\bar m(0)}_E > \epsilon^{-\alpha}}$, we have $\taue = 0$ and $\bar m^{\epsilon,\taue}(t) = \bar m(0)$. Thus, one does not necessarily have the estimate $\norm{\bar m^{\epsilon,\taue}(t)}_E \leq \epsilon^{-\alpha}$.
\end{remark}

\begin{proof}
	Let $\seq{S_k}{k}$ be the identically distributed random variables defined by
	\begin{equation*}
		\forall k \in \mathbb N_0, S_k \doteq \sup_{t \in [k,k+1]} \norm{\bar m(t)}_E.
	\end{equation*}
	By means of \cref{hyp:stationarity,hyp:m_centered,hyp:m_L2_all_starting_point,hyp:m_moments}, for all $k \in \mathbb N$, $\esp{S_k^\gamma} = \esp{S_0^\gamma} < \infty$. Thus, almost surely, $S_k < \infty$ for all $k \in \mathbb N_0$. This yields the first result:
	\begin{equation*}
		\mathbb P\mbox{-a.s.}, \forall T > 0, \sup_{t \in [0,T]} \norm{\bar m^\epsilon(t)}_E \leq \sup_{k \leq T\epsilon^{-2}+1} S_k < \infty.
	\end{equation*}

	Since $\alpha > \frac 2 \gamma$, there exists $\delta$ such that $ \frac \alpha 2 > \delta > \frac 1 \gamma$. Then, the Markov inequality yields
	\begin{equation*}
		\sum_{k \in \mathbb N} \proba{S_k \geq k^\delta} \leq \sum_{k \in \mathbb N} \frac{\esp{S_k^\gamma}}{k^{\delta \gamma}} = \esp{S_0^\gamma} \sum_{k \in \mathbb N} \frac{1}{k^{\delta \gamma}} < \infty.
	\end{equation*}
	By means of the Borel-Cantelli Lemma, almost surely, there exists a random integer $k_0 \in \mathbb N$ such that
	\begin{equation*}
		\mathbb P\mbox{-a.s.}, \forall k > k_0, S_k < k^\delta.
	\end{equation*}
	Define the random variable $Z \doteq \sup_{k \leq k_0} S_k$. Then $Z$ is almost surely finite and
	\begin{equation*}
		\mathbb P \mbox{-a.s.}, \forall t \in \tSet, \norm{\bar m(t)}_E \leq S_{\floor t} \leq Z + \floor{t}^\delta \leq Z + t^\delta.
	\end{equation*}
	Finally, for $T > 0$, using that $\alpha > 2 \delta$, we get
	\begin{equation*}
		\proba{\taue < T} = \proba{\sup_{t \in [0,T]} \norm{\bar m^\epsilon(t)}_E > \epsilon^{-\alpha}} \leq \proba{Z + (T \epsilon^{-2})^\delta > \epsilon^{-\alpha}} \xrightarrow[\epsilon \to 0]{} 0.
	\end{equation*}
\end{proof}

In the sequel, $\alpha$ will be required to satisfy the constraint
\begin{equation} \label{eq:hyp_alpha}
	\alpha < \frac{1}{b+2}.
\end{equation}
Combined with the condition $\alpha > \frac 2 \gamma$ appearing in \cref{lem:taue_to_infty}, this motivates the condition $\gamma > 2(b+2)$ in \cref{hyp:m_L2_all_starting_point}.

\subsection{Pathwise solutions}

By means of \cref{lem:taue_to_infty}, we are in position to prove the existence and uniqueness of pathwise solutions of \cref{eq:SPDE,eq:initial_condition} (namely solutions when $\omega$ is fixed).

\begin{proposition} \label{prop:pathwise_solution}
	Let \cref{hyp:m_moments,hyp:stationarity} be satisfied. Let $T > 0$ and $\epsilon > 0$. Then for any $f^\epsilon_0 \in \fSet$, there exists, almost surely, a unique solution $f^\epsilon$ of \cref{eq:SPDE} in $C([0,T];\fSet)$, in the sense that
	\begin{equation*}
		\mathbb P\mbox{-a.s.}, \forall t \in [0,T], f^\epsilon(t) = e^{- \frac t \epsilon A} f^\epsilon_0 + \int_0^t e^{- \frac{t-s}{\epsilon} A} \parths{\frac{1}{\epsilon^2} Lf^\epsilon(s) + \frac 1 \epsilon f^\epsilon(s) \bar m^\epsilon(s)}ds
	\end{equation*}
	where $A$ is the operator defined by
	\begin{gather*}
		D(A) \doteq \set{f \in \fSet \mid (x,v) \mapsto \scald{a(v)}{\nabla_x f(x,v)} \in \fSet}\\
		Af(x,v) \doteq \scald{a(v)}{\nabla_x f(x,v)}.
	\end{gather*}
	% Besides, if $f^\epsilon_0 \in L^2(H^1) = L^2(\vSet;H^1(\xSet))$, then, $\mathbb P$-a.s.
	% \begin{equation*}
	% 	f^\epsilon \in C^1([0,T];L^2) \cap C([0,T];L^2(H^1)).
	% \end{equation*}
\end{proposition}

Note that here $\epsilon$ is fixed. Thus, the proof is standard, based on a fixed-point theorem.

\begin{proof}
	Let $\omega \in \Omega$ and $\epsilon > 0$. For $f \in C([0,T],\fSet)$, let
	\begin{equation*}
		\Phi(f) = e^{- \frac t \epsilon A} f^\epsilon_0 + \int_0^t e^{- \frac{t-s}{\epsilon} A} \parths{\frac{1}{\epsilon^2} Lf(s) + \frac 1 \epsilon f(s) \bar m^\epsilon(s)}ds.
	\end{equation*}
	Owing to the Banach fixed-point theorem, it is sufficient to prove that $\Phi$ is a contraction for some Banach norm on $C([0,T],\fSet)$. For $r \in [0,\infty)$, we consider the following Banach norm
	\begin{equation*}
		\forall f \in C([0,T],\fSet), \norm{f}_r = \sup_{t \in [0,T]} e^{-r t} \norm{f(t)}_\fSet.
	\end{equation*}

	Since the semi-group associated to $A$ is given by
	\begin{equation*}
		\forall f \in \fSet, \forall x \in \xSet, \forall v \in \vSet, e^{tA}f(x,v) = f(x + t a(v),v),
	\end{equation*}
	for $f \in \fSet$, we have for all $t \in \mathbb R$ and $f \in \fSet$,
	\begin{equation*}
		\norm{e^{tA}f}_\fSet = \norm{f}_\fSet.
	\end{equation*}
	Thus, for $t \in [0,T]$, and $f, g \in C([0,T],\fSet)$, we get
	\begin{equation*}
		\begin{split}
		\norm{\Phi(f)(t) - \Phi(g)(t)}_\fSet
			&\leq \frac{1}{\epsilon^2} \int_0^t \norm{L(f-g)(s)}_\fSet ds\\
			&\quad + \frac 1 \epsilon \int_0^t \norm{(f-g)(s)\bar m^\epsilon(s)}_\fSet ds.
		\end{split}
	\end{equation*}

	By \cref{lem:taue_to_infty}, since $\omega$ is fixed, $\bar m^\epsilon$ is bounded in E on $[0,T]$. Since $\norm{Lh}_\fSet \leq \norm{h}_\fSet$ for $h \in \fSet$, we get
	\begin{equation*}
		\norm{\Phi(f)(t) - \Phi(g)(t)}_\fSet \leq \parths{\frac{1}{\epsilon^2} + \frac{1}{\epsilon} \sup_{t \in [0,T]} \norm{\bar m^\epsilon(t)}_E} \int_0^t e^{r s} \norm{f-g}_r ds.
	\end{equation*}
	Hence, we have
	\begin{equation*}
		e^{-r t} \norm{\Phi(f)(t) - \Phi(g)(t)}_\fSet \leq \parths{\frac{1}{\epsilon^2} + \frac{1}{\epsilon} \sup_{t \in [0,T]} \norm{\bar m^\epsilon(t)}_E} \frac{1 - e^{-r t}}{r} \norm{f-g}_r,
	\end{equation*}
	and
	\begin{equation*}
		\norm{\Phi(f)-\Phi(g)}_r \leq \frac{1}{r} \parths{\frac{1}{\epsilon^2} + \frac{1}{\epsilon} \sup_{t \in [0,T]} \norm{\bar m^\epsilon(t)}_E} \norm{f-g}_r.
	\end{equation*}

	By taking $r$ large enough, we get that $\Phi$ is contracting, which concludes the proof.
\end{proof}

% \begin{remark}
% 	Note that by integration by parts, $A$ is skew-adjoint on $\fSet$ and $L^2$.
% \end{remark}

\subsection{Estimate in $\fSet$} \label{subsec:bound_L2}

In this section, we obtain an upper bound on $\norm{f\stoppedL}_\fSet$. Note that in the case where the driving process $\bar m$ is bounded, \cite{debussche2011diffusion} establishes a similar upper bound without introducing a stopping time. Here, the unboundedness of the stopped process $\bar m\stopped$ requires more intricate arguments and an additional stopping time $\tauLze$. One of these arguments is the introduction of a weight $\mathcal M^\epsilon$ that depends on $\epsilon$.

\begin{proposition} \label{prop:L2_bound}
	Assume that $f^\epsilon_0 \in \fSet$. For $\Lambda > 0$ and $\epsilon > 0$, define $\zeta^\epsilon$ by \cref{eq:def_zeta_epsilon} and $\tauLze$ by \cref{eq:def_tauLze}.

	Then almost surely, for all $t \in [0,T]$ and $\epsilon \in (0,\parths{4 \norm{a}_{L^\infty} \Lambda}^{-1}]$,
	\begin{equation} \label{eq:L2_bound}
		\norm{f\stoppedLz(t)}_\fSet^2 + \frac{1}{\epsilon^2} \int_0^{t \wedge \tauLze} \norm{L f\stoppedLz(s)}_{\fSet}^2 ds \leq C_\Lambda(T) \norm{f^\epsilon_0}_\fSet^2,
	\end{equation}
	for some $C_\Lambda(T) > 0$ depending only on $\Lambda$, $\norm{a}_{L^\infty}$ and $T$.
\end{proposition}

Note that $\tauLze > 0$ almost surely since $\zeta^\epsilon (0) = 0$.

\begin{remark}
	The condition $\epsilon \in (0,\parths{4 \norm{a}_{L^\infty} \Lambda}^{-1}]$ only reads: we fix $\Lambda$, then take $\epsilon$ small enough ($\epsilon \to 0$) depending on the fixed $\Lambda$. From now on, we always assume $\epsilon \leq \parths{4 \norm{a}_{L^\infty} \Lambda}^{-1} < 1$. In particular, we denote by $\sup_\epsilon$ the supremum with respect to $\epsilon \in (0,\parths{4 \norm{a}_{L^\infty} \Lambda}^{-1}] \subset (0,1)$.
\end{remark}

In most of the paper, we neglect the integral term of the left-hand side of \cref{eq:L2_bound} and we only use
\begin{equation*}
	\norm{f\stoppedLz(t)}_\fSet^2 \leq C_\Lambda(T) \norm{f^\epsilon_0}_\fSet^2.
\end{equation*}
Equation \cref{eq:L2_bound} will prove useful in \cref{sec:strong_tightness}.

Let us introduce some notation. For any variable $u$, $x \lesssim_u y$ means that there exists $C$ such that $x \leq C y$ where $C$ depends only on $u$, $a$, $\mathcal M$, $B$, $\nu$, $\gamma$, $\alpha$, $\gamma_{mix}$, $b$ and $\law{f^\epsilon_0}$ the distribution of $f^\epsilon_0$. With this notation, \cref{eq:L2_bound} yields
\begin{equation*}
	\norm{f\stoppedLz(t)}_\fSet^2 \lesssim_\Lambda \norm{f^\epsilon_0}_\fSet^2.
\end{equation*}
and
\begin{equation*}
	\frac{1}{\epsilon^2} \int_0^{t \wedge \tauLze} \norm{L f\stoppedLz(s)}_{\fSet}^2 ds \lesssim_\Lambda \norm{f^\epsilon_0}_\fSet^2.
\end{equation*}

\begin{proof}
	Define the time-dependent weight
	\begin{equation*}
		\mathcal M^\epsilon(t,x,v) = e^{2 \zeta^\epsilon(t,x)} \mathcal M(v),
	\end{equation*}
	and the associate weighted norm
	\begin{equation*}
		\norm{f}_{L^2(\mathcal M^\epsilon(t)^{-1})} \doteq \parths{\int\int \frac{\abs{f^\epsilon(x,v)}^2}{\mathcal M^\epsilon(t,x,v)} dx d\mu(v)}^{\frac 1 2}.
	\end{equation*}
	We have, for $t \in \tSet$,
	\begin{align*}
		\begin{split}
		\frac 1 2 \partial_t \norm{f^\epsilon(t)}_{L^2(\mathcal M^\epsilon(t)^{-1})}^2
			&= \int\int \parths{\frac{f^\epsilon(t,x,v)}{\mathcal M^\epsilon(t,x,v)} \parths{- \frac{1}{\epsilon} Af^\epsilon+ \frac{1}{\epsilon^2} L f^\epsilon + \frac {1} \epsilon f^\epsilon \bar m^\epsilon}(t,x,v) \right.\\
			&\quad \left. - \frac{\abs{f^\epsilon(t,x,v)}^2}{2\abs{\mathcal M^\epsilon(t,x,v)}^2} \partial_t \mathcal M^\epsilon(t,x,v)} dx d\mu(v)
		\end{split}\\
			&= \mathcal A_\epsilon + \mathcal B_\epsilon
	\end{align*}
	with
	\begin{gather*}
		\mathcal A_\epsilon = \frac{1}{\epsilon^2} \int\int \frac{f^\epsilon(t,x,v)}{\mathcal M^\epsilon(t,x,v)} \parths{Lf^\epsilon + \epsilon f^\epsilon \bar m^\epsilon - \frac{\epsilon^2}{2} \frac{f^\epsilon}{\mathcal M^\epsilon} \partial_t \mathcal M^\epsilon}(t,x,v) dx d\mu(v)\\
		\mathcal B_\epsilon = - \frac 1 \epsilon \int\int \frac{f^\epsilon(t,x,v)}{\mathcal M^\epsilon(t,x,v)} Af^\epsilon(t,x,v) dx d\mu(v).
	\end{gather*}

	On the one hand, the weight $\mathcal M^\epsilon$ has been chosen in order to satisfy $\epsilon \bar m^\epsilon - \frac{\epsilon^2}{2} \frac{\partial_t \mathcal M^\epsilon}{\mathcal M^\epsilon} = 0$. Moreover, since $f^\epsilon = \rho^\epsilon \mathcal M - Lf^\epsilon$ and $\int_\vSet Lf^\epsilon d\mu = 0$, we get
	\begin{align*}
		\mathcal A_\epsilon
			&= \frac{1}{\epsilon^2} \int\int \frac{f^\epsilon(t,x,v)}{\mathcal M^\epsilon(t,x,v)} Lf^\epsilon(t,x,v) dx d\mu(v)\\
			&= \frac{1}{\epsilon^2} \int_\xSet e^{-2\zeta^\epsilon(t,x)} \rho^\epsilon(t,x) \int_\vSet Lf^\epsilon(t,x,v) d\mu(v) dx\\
			&\quad - \frac{1}{\epsilon^2} \int\int \frac{\abs{Lf^\epsilon(t,x,v)}^2}{\mathcal M^\epsilon(t,x,v)} d\mu(v) dx\\
			&= - \frac{1}{\epsilon^2} \int\int \frac{\abs{Lf^\epsilon(t,x,v)}^2}{\mathcal M^\epsilon(t,x,v)} d\mu(v) dx = - \frac{1}{\epsilon^2} \norm{L f^\epsilon(t)}_{L^2(\mathcal M^\epsilon(t)^{-1})}^2.
	\end{align*}

	On the other hand, by means of an integration by parts (we take a primitive of $f^\epsilon \partial_{x_i}f^\epsilon$ and a derivative of $\frac{1}{\mathcal M^\epsilon}$), we write
	\begin{align*}
		\mathcal B_\epsilon
			&= - \frac 1 \epsilon \int\int \scald{a(v)}{\frac{f^\epsilon(t,x,v) \nabla_x f^\epsilon(t,x,v)}{\mathcal M^\epsilon(t,x,v)}} dx d\mu(v)\\
			&= - \frac 1 \epsilon \int\int \scald{a(v)}{\frac{\frac 1 2 \abs{f^\epsilon(t,x,v)}^2  \nabla_x \mathcal M^\epsilon(t,x,v)}{\abs{\mathcal M^\epsilon(t,x,v)}^2}} dx d\mu(v)\\
			&= - \frac{1}{2\epsilon} \int\int \abs{\frac{f^\epsilon(t,x,v)}{\mathcal M^\epsilon(t,x,v)}}^2 A\mathcal M^\epsilon(t,x,v) dx d\mu(v).
	\end{align*}
	Then, by definition of $\mathcal M^\epsilon$ and $A$, we have
	\begin{equation*}
		\mathcal B_\epsilon = - \frac{1}{\epsilon} \int_\xSet \scald{\nabla_x \zeta^\epsilon(t,x)}{\int_\vSet \frac{\abs{f^\epsilon(t,x,v)}^2}{\mathcal M^\epsilon(t,x,v)} a(v) d\mu(v)} dx.
	\end{equation*}
	Using once again the identity $f^\epsilon = \rho^\epsilon \mathcal M - Lf^\epsilon$, we get
	\begin{align*}
		\mathcal B_\epsilon
			&= - \frac{1}{\epsilon} \int_\xSet e^{-2\zeta^\epsilon(t,x)} \abs{\rho^\epsilon(t,x)}^2 \scald{\nabla_x \zeta^\epsilon(t,x)}{\int_\vSet a(v) \mathcal M(v) d\mu(v)} dx\\
			&\quad - \frac{1}{\epsilon} \int_\xSet \scald{\nabla_x \zeta^\epsilon(t,x)}{\int_\vSet \frac{\abs{Lf^\epsilon(t,x,v)}^2}{\mathcal M^\epsilon(t,x,v)} a(v) d\mu(v)} dx\\
			&\quad + \frac{2}{\epsilon} \int\int e^{-2\zeta^\epsilon(t,x)} \scald{\nabla_x \zeta^\epsilon(t,x)}{a(v)} \rho^\epsilon(t,x) Lf^\epsilon(t,x,v) d\mu(v) dx\\
			&= \mathcal B_\epsilon^1 + \mathcal B_\epsilon^2 + \mathcal B_\epsilon^3.
	\end{align*}
	\begin{itemize}
		\item Since $a$ is centered for $\mathcal M \mu$, $\mathcal B_\epsilon^1 = 0$.

		\item For $t \leq \tauLze$, we have $\norm{\zeta^\epsilon(t)}_{C^1_x} \leq \Lambda$ and we assumed $\Lambda \leq \parths{4 \norm{a}_{L^\infty} \epsilon}^{-1}$. Thus, we get
		\begin{equation*}
			\forall t \leq \tauLze, \abs{\mathcal B_\epsilon^2} \leq \frac{1}{4\epsilon^2} \norm{L f^\epsilon(t)}_{L^2(\mathcal M^\epsilon(t)^{-1})}^2.
		\end{equation*}

		\item Using the Young inequality, we have
		\begin{multline*}
			\abs{\mathcal B_\epsilon^3} \leq \frac{1}{4\epsilon^2} \norm{L f^\epsilon(t)}_{L^2(\mathcal M^\epsilon(t)^{-1})}^2\\
			+ 4 \norm{a}_{L^2(\mathcal M)}^2 \norm{\nabla_x \zeta^\epsilon(t)}_{C(\xSet)}^2 \int_\xSet e^{-2\zeta^\epsilon(t,x)}  \abs{\rho^\epsilon(t,x)}^2 dx,
		\end{multline*}
		with $\norm{a}_{L^2(\mathcal M)}^2 \doteq \int_V \abs{a(v)}^2 \mathcal M(v) d\mu(v)$. Now using the Cauchy-Schwarz inequality and the identity $\int_\vSet \mathcal M(v) d\mu(v) = 1$, we have
		\begin{equation*}
			\abs{\rho^\epsilon(t,x)}^2 \leq \int_\vSet \frac{\abs{f^\epsilon(t,x,v)}^2}{\mathcal M^\epsilon(t,x,v)} d\mu(v) \int_\vSet \mathcal M^\epsilon(t,x,v) d\mu(v) = \norm{f^\epsilon(t)}_{L^2(\mathcal M^\epsilon(t)^{-1})}^2 e^{2\zeta^\epsilon(t,x)},
		\end{equation*}
		hence
		\begin{equation*}
			\abs{\mathcal B_\epsilon^3} \leq \frac{1}{4\epsilon^2} \norm{L f^\epsilon(t)}_{L^2(\mathcal M^\epsilon(t)^{-1})}^2 + 4 \norm{a}_{L^2(\mathcal M)}^2 \norm{\nabla_x \zeta^\epsilon(t)}_{C(\xSet)}^2 \norm{f^\epsilon(t)}_{L^2(\mathcal M^\epsilon(t)^{-1})}^2.
		\end{equation*}
	\end{itemize}

	We finally get, for $t \leq \tauLze$,
	\begin{multline*}
		\partial_t \norm{f^\epsilon(t)}_{L^2(\mathcal M^\epsilon(t)^{-1})}^2 \leq - \frac{1}{2\epsilon^2} \norm{L f^\epsilon(t)}_{L^2(\mathcal M^\epsilon(t)^{-1})}^2\\
		+ 4 \norm{a}_{L^2(\mathcal M)}^2 \norm{\nabla_x \zeta^\epsilon(t)}_{C(\xSet)}^2 \norm{f^\epsilon(t)}_{L^2(\mathcal M^\epsilon(t)^{-1})}^2.
	\end{multline*}
	For $t \leq \tauLze$, Gronwall's Lemma implies
	\begin{multline*}
		\norm{f^\epsilon(t)}_{L^2(\mathcal M^\epsilon(t)^{-1})}^2 + \int_0^t \frac{1}{2\epsilon^2} \norm{L f^\epsilon(t)}_{L^2(\mathcal M^\epsilon(t)^{-1})}^2 dt\\
		\leq \norm{f^\epsilon_0}_\fSet^2 e^{4 \norm{a}_{L^2(\mathcal M)}^2 \int_0^t \norm{\nabla_x \zeta^\epsilon(s)}_{C(\xSet)}^2 ds}.
	\end{multline*}
	Since, for $t \in \tSet$, we have
	\begin{equation*}
		\norm{\cdot}_{L^2(\mathcal M^\epsilon(t)^{-1})}^2
		%= \int\int \frac{\abs{f^\epsilon(t,x,v)}^2}{\mathcal M(v)} e^{-2\zeta^\epsilon(t,x)} dx d\mu(v)
		\geq \norm{\cdot}_\fSet^2 e^{-2\norm{\zeta^\epsilon(t)}_{C(\xSet)}},
	\end{equation*}
	we get, for $t \leq \tauLze$,
	\begin{multline*}
		\norm{f^\epsilon(t)}_\fSet^2 + \int_0^t \frac{1}{2\epsilon^2} \norm{L f^\epsilon(t)}_\fSet^2 dt\\
		\leq \norm{f^\epsilon_0}_\fSet^2 \exp\parths{2 \sup_{s \in [0,t]} \norm{\zeta^\epsilon(s)}_{C(\xSet)} + 4 \norm{a}_{L^2(\mathcal M)}^2 \int_0^t \norm{\nabla_x \zeta^\epsilon(s)}_{C(\xSet)}^2 ds}.
	\end{multline*}
	To conclude, it is sufficient to recall that for $t \leq \tauLze$, we have $\norm{\zeta^\epsilon(t)}_{C^1_x} \leq \Lambda$.
\end{proof}

%\section{Martingale problems} \label{sec:martingale_problems}
\section{Martingale problems and perturbed test functions} \label{sec:martingale_problems}

The proof of \cref{thm:main_result} heavily relies on the notion of martingale problems as introduced in \cite{stroock2006multidimensional}. To identify a limit point of $\seq{\law{\rho^\epsilon}}{\epsilon > 0}$, we characterize it by a family of martingales and take the limit when $\epsilon \to 0$ in their martingale properties.

The characterization of the distribution of a solution of a SPDE in terms of martingales is based on the Markov property satisfied by this solution. However, we expect a limit point $\rho_\Lambda$ of the stopped process $\rho\stoppedL$ to be stopped at some $\tau_\Lambda(\zeta)$, as mentioned in \cref{subsec:bound_L2}. Since $\tau_\Lambda(\zeta)$ is not a stopping time for the filtration generated by $\rho_\Lambda$, this latter process should not be Markov. Thus, we need to consider the convergence of the couple $(\rho^\epsilon,\zeta^\epsilon)$ instead of just $\rho^\epsilon$. We will see more precisely in \cref{sec:convergence} at which point this matter occurs.

\subsection{Generator and martingales} \label{subsec:generator_martingale}

Also note that $(f^\epsilon,\zeta^\epsilon)$ is not a Markov process. As in \cite{debussche2011diffusion}, we consider the coupled process with $\bar m^\epsilon$ and thus consider the $\fSet \times C^1_x \times E$-valued Markov process $X^\epsilon \doteq (f^\epsilon,\zeta^\epsilon,\bar m^\epsilon)$.

Denote by $\mathcal L^\epsilon$ the infinitesimal generator of $X^\epsilon$. Since $f^\epsilon$ is solution of \cref{eq:SPDE} and since $\partial_t \zeta^\epsilon = \frac{1}{\epsilon} \bar m^\epsilon$, the infinitesimal generator has an expression of the type
\begin{equation} \label{eq:def_L_epsilon}
	\mathcal L^\epsilon = \frac 1 \epsilon \mathcal L_1 + \frac{1}{\epsilon^2} \mathcal L_2
\end{equation}
with
\begin{gather*}
	\mathcal L_1 \phi(f,z,n) = D_f\phi(f,z,n)(- Af + nf) + D_z\phi(f,z,n)(n)\\
	\mathcal L_2 \phi(f,z,n) = D_f\phi(f,z,n)(Lf) + B\phi(f,z,n),
\end{gather*}
where $B$ is the infinitesimal generator of $\bar m$. The domain of this generator contains the class of good test functions defined below. The terminology of "good test function" is inherited from \cite{debussche2011diffusion}, although our definition is a little more restrictive.
\begin{definition} \label{def:good_test_function}
	A function $\phi : \fSet \times C^1_x \times E \to \mathbb R$ is called a good test function if
	\begin{itemize}
		\item It is continuously differentiable on $\fSet \times C^1_x \times E$ with respect to the first and second variables.
		\item For $\ell \in \set{1,2}$, $B(\phi(f,z,\cdot)^\ell)$ is defined for all $(f,z) \in \fSet \times C^1_x$, and
		\begin{equation*}
			B(\phi^\ell) : \fSet \times C^1_x \times E \to \mathbb R
		\end{equation*}
		is continuous.
		\item If we identify the differential $D_f$ with the gradient, then for $f \in \fSet$, $z \in C^1_x$ and $n \in E$, we have
		\begin{equation} \label{eq:good_test_function_grad_H1}
			D_f \phi(f,z,n) \in H^1(\xSet \times \vSet,dx \mathcal M^{-1}(v) d\mu(v)).
		\end{equation}
		\item The functions $\phi$, $D_z \phi$, $D_f \phi$ and $AD_f \phi$ have at most polynomial growth in the following sense: there exists $C_\phi > 0$ such that for $f, h \in \fSet$, $z \in C^1_x$ and $n_1, n_2 \in E$, we have
		\begin{equation}
			\begin{split} \label{eq:estimates_good_test_function}
			\abs{\phi(f,z,n_1)} &\leq C_\phi \parths{1 + S_1^{3}} \parths{1 + S_2^{b+2}}\\
			\abs{D_f \phi(f,z,n_1)(Ah)} &\leq C_\phi \parths{1 + S_1^{3}} \parths{1 + S_2^{b+2}}\\
			\abs{D_f \phi(f,z,n_1)(n_2 h)} &\leq C_\phi \parths{1 + S_1^{3}} \parths{1 + S_2^{b+2}}\\
			\abs{D_f \phi(f,z,n_1)(Lh)} &\leq C_\phi \parths{1 + S_1^{3}} \parths{1 + S_2^{b+2}}\\
			\abs{D_z \phi(f,z,n_1)(n_2)} &\leq C_\phi \parths{1 + S_1^{3}} \parths{1 + S_2^{b+2}},
			\end{split}
		\end{equation}
		where $S_1 = \norm{f}_\fSet \vee \norm{h}_\fSet$ and $S_2 = \norm{n_1}_E \vee \norm{n_2}_E$.
	\end{itemize}
\end{definition}

See \cref{subsec:perturbed_test_functions} for a justification of the need to consider growth as appearing in \cref{eq:estimates_good_test_function}.

A consequence of \cref{eq:good_test_function_grad_H1} is that $A D_f \phi$ is well-defined. Thus, for $f,h \in \fSet$, $z \in C^1_x$ and $n \in E$, we can define
\begin{equation*}
	D_f \phi(f,z,n)(Ah) \doteq -\scalf{AD_f \phi(f,z,n)}{h},
\end{equation*}
even though $Ah$ is not necessarily defined in $\fSet$.

The class of test-function introduced in \cref{def:good_test_function} is chosen such that the \cref{prop:good_test_function_martingale} holds.

\begin{proposition} \label{prop:good_test_function_martingale}
	Let $\phi$ be a good test function in the sense of \cref{def:good_test_function}. Define for all $t \geq 0$
	\begin{equation} \label{eq:def_premartingale}
		M_\phi^\epsilon(t) \doteq \phi(X^\epsilon(t)) - \phi(X^\epsilon(0)) -  \int_0^t \mathcal L^\epsilon \phi (X^\epsilon(s))ds,
	\end{equation}
	and consider the stopping time $\tauLe$ defined by \cref{eq:def_tauLe}.

	Then $M\stoppedL_\phi$ is a càdlàg $\seq{\mathcal F^\epsilon_t}{t}$-martingale and
	\begin{align*}
		\forall t \in \tSet, \esp{\abs{M\stoppedL_\phi(t)}^2}
			&= \esp{\int_0^{t \wedge \tauLe} \parths{\mathcal L^\epsilon (\phi^2) - 2 \phi \mathcal L^\epsilon \phi}(X^\epsilon(s))ds}\\
			&= \frac{1}{\epsilon^2} \esp{\int_0^{t \wedge \tauLe} \parths{B (\phi^2) - 2 \phi B \phi}(X^\epsilon(s))ds}.
	\end{align*}
\end{proposition}

This result is expected to holds as in the standard framework \cite{debussche2017diffusion}. However, due to the presence of stopping times, the proof is very technical.

%% garder pour la thèse mais à retirer pour un article % Using the infinitesimal generator of a Markov process to construct a martingale of the form \cref{eq:def_premartingale} is a standard result (see \cite{ethier1986markov}). However, we have to write a specific proof due to the presence of a stopping time and of some unbounded quantities. It does not change the idea of the proof but prevents us from applying the standard theorem directly.

\begin{proof}
	Note that in this section, $\epsilon$ is fixed, it is therefore not required to prove bounds which are uniform with respect to $\epsilon$.

	Let $\phi$ be a good test function. Observe that $\phi$ and $\phi^2$ are in the domain of $\mathcal L^\epsilon$, by means of \cref{def:good_test_function}.

	Let $s,t \in \tSet$, $\delta > 0$ and let $s = t_1 < ... < t_n = t$ be a subdivision of $[s,t]$ such that $\max_i \abs{t_{i+1}-t_i} = \delta$. Let $g$ be a $\mathcal F^\epsilon_s$-measurable and bounded function. To simplify notation, let
	\begin{equation*}
		f_i \doteq f\stoppedL(t_i), \zeta_i \doteq \zeta\stoppedL(t_i), m_i \doteq \bar m\stoppedL(t_i).
	\end{equation*}
	Then, we have
	\begin{multline*}
		\esp{\parths{M\stoppedL_\phi(t) - M\stoppedL_\phi(s)}g}\\
		\begin{split}
		&= \esp{\parths{\phi(X\stoppedL(t)) - \phi(X\stoppedL(s)) - \int_{s \wedge \tauLe}^{t \wedge \tauLe} \mathcal L^\epsilon \phi(X^\epsilon(u))du}g}\\
		&= r_f + r_z + r_n,
		\end{split}
	\end{multline*}
	where
	\begin{multline*}
		r_f = \sum_{i=1}^{n-1} \esp{\parths{\phi(f_{i+1},\zeta_{i+1},m_{i+1}) - \phi(f_{i},\zeta_{i+1},m_{i+1}) \vphantom{\int_{t_i \wedge \tauLe}^{t_{i+1} \wedge \tauLe}} \right. \right.\\
		\quad \left. \left. - \int_{t_i \wedge \tauLe}^{t_{i+1} \wedge \tauLe} D_f \phi(X^\epsilon(u))(- \frac 1 \epsilon Af^\epsilon(u) + \frac{1}{\epsilon^2}Lf^\epsilon(u) + \frac 1 \epsilon f^\epsilon(u) \bar m^\epsilon(u)) du}g},
	\end{multline*}
	\begin{multline*}
		r_z = \sum_{i=1}^{n-1} \esp{\parths{\phi(f_{i},\zeta_{i+1},m_{i+1}) - \phi(f_{i},\zeta_{i},m_{i+1}) \vphantom{\int_{t_i \wedge \tauLe}^{t_{i+1} \wedge \tauLe}}\right. \right.\\
		\quad \left. \left. - \int_{t_i \wedge \tauLe}^{t_{i+1} \wedge \tauLe} D_z \phi(X^\epsilon(u))(\frac 1 \epsilon \bar m^\epsilon(u)) du}g},
	\end{multline*}
	and
	\begin{equation*}
		r_n = \sum_{i=1}^{n-1} \esp{\parths{\phi(f_{i},\zeta_{i},m_{i+1}) - \phi(f_{i},\zeta_{i},m_{i}) - \int_{t_i \wedge \tauLe}^{t_{i+1} \wedge \tauLe} \frac{1}{\epsilon^2} B\phi(X^\epsilon(u)) du}g}.
	\end{equation*}
	Straightforward computations lead to
	\begin{equation*}
		r_f = \esp{\parths{\int_s^t r'_f(u)du}g}, r_z = \esp{\parths{\int_s^t r'_z(u)du}g}
	\end{equation*}
	with
	\begin{align*}
		r'_f(u)
			&=\sum_{i=1}^{n-1} \indset{[t_i \wedge \tauLe,t_{i+1} \wedge \tauLe]}(u) \lrbracket{D_f \phi(f\stoppedL(u),\zeta_{i+1},m_{i+1}) - D_f \phi(X\stoppedL(u))}(\partial_t f\stoppedL(u)),\\
		r'_z(u)
			&=\sum_{i=1}^{n-1} \indset{[t_i \wedge \tauLe,t_{i+1} \wedge \tauLe]}(u) \lrbracket{D_z \phi(f_i,\zeta\stoppedL(u),m_{i+1}) - D_z \phi(X\stoppedL(u))}(\partial_t \zeta\stoppedL(u)).
	\end{align*}
	Let us now check that $r_n = \esp{\parths{\int_s^t r'_n(u)du}g}$ with
	\begin{equation*}
		r'_n(u) = \frac{1}{\epsilon^2} \sum_{i=1}^{n-1} \ind{[t_i \wedge \tauLe,t_{i+1} \wedge \tauLe]}(u) \lrbracket{B \phi(f_i,\zeta_i,\bar m\stoppedL(u)) - B \phi(X\stoppedL(u))}.
	\end{equation*}
	For $\theta \in C(E) \cap L^1(E)$, the Markov property for $\bar m$ yields
	\begin{equation*}
		\esp{\theta(\bar m(t)) \mid \mathcal F_s} = P_{t-s}\theta(\bar m(s)).
	\end{equation*}
	Usually, this property is written for $\theta$ deterministic, continuous and bounded, but it is straightforward to check that it is still satisfied when $\theta \in C(E) \cap L^1(E)$ $\mathcal F_s$-measurable. The standard proof to show that $\bar m$ solves the martingale problem associated to $B$ (see for example \cite{debussche2017diffusion}, Theorem B.3) can be applied, and we get that, for $\theta \in D(B)$,
	\begin{equation*}
		t \mapsto \theta(\bar m(t)) - \theta(\bar m(0)) - \int_0^t B\theta(\bar m(u))du
	\end{equation*}
	is an integrable $\mathcal F_t$-martingale. By rescaling the time to retrieve $\bar m^\epsilon$, stopping the martingale at $\tauLe$ and using a conditioning argument ($g$, $f_i$ and $\zeta_i$ are $\mathcal F_{t_i}$-measurable), we get
	\begin{equation*}
		\esp{\parths{\phi(f_{i},\zeta_{i},m_{i+1}) - \phi(f_{i},\zeta_{i},m_{i})}g} = \esp{g \int_{t_i \wedge \tauLe}^{t_{i+1} \wedge \tauLe} \frac{1}{\epsilon^2} B\phi(X^\epsilon(u)) du}.
	\end{equation*}
	Hence, we can write $r_n = \esp{\parths{\int_s^t r'_n(u)du}g}$ as claimed above.

	Since the estimates given by \cref{eq:taue_bounds_m,prop:L2_bound} are uniform for $t \in [0,T]$, we can use \cref{eq:estimates_good_test_function} with $S_1 \lesssim_\Lambda \norm{f^\epsilon_0}_\fSet$ and $S_2 \leq \epsilon^{-\alpha} \vee \norm{\bar m(0)}_E$. This leads to
	\begin{equation} \label{eq:subpol_condition}
		\sup_{u \in [0,T]} \abs{r'_*(u)}^2 \lesssim_{\phi,\Lambda,\epsilon} \parths{1 + \norm{f^\epsilon_0}_\fSet^{6}} \parths{1 + \norm{\bar m(0)}_E^{2(b+2)}},
	\end{equation}
	where $r'_* \in \set{r'_f,r'_z,r'_n}$. Hence, the Cauchy-Schwarz inequality, \cref{hyp:m_L2_all_starting_point,hyp:f0} yield
	\begin{align}
		\esp{\sup_{u \in [0,T]} \abs{r'_*(u)}^2}
			&\lesssim_{\phi,\Lambda,\epsilon} \esp{\parths{1 + \norm{f^\epsilon_0}_\fSet^{12}}}^{\frac 1 2} \esp{\parths{1 + \norm{\bar m(0)}_E^{4(b+2)}}}^{\frac 1 2} \nonumber\\
			&< \infty. \label{eq:subpol_conclusion}
	\end{align}
	Thus, the terms $r'_*$ are uniformly integrable with respect to $(u,\omega)$. Recall that $f\stoppedL$ and $\zeta\stoppedL$ are almost surely continuous and that $\bar m\stoppedL$ is stochastically continuous. Then, the terms $r'_*$ converge to $0$ in probability when $\delta \to 0$. By uniform integrability, the terms $r_*$ converge to $0$, which proves that $M\stoppedL_\phi$ is a $\seq{\mathcal F^\epsilon_t}{t}$-martingale. Note that we only used moments of order $12$ and $4(b+2)$, instead of $24$ and $8(b+2)$ as assumed in \cref{hyp:m_L2_all_starting_point,hyp:f0}. Hence, this proof can be adapted to establish that $M\stoppedL_{\phi^2}$ is also a $\seq{\mathcal F^\epsilon_t}{t}$-martingale.% ($\phi^2$ and its derivatives satisfy the same estimates as $\phi$ with twice the order of the moments).

	% Since $M\stoppedL_\phi$ is martingale for a filtration satisfying the usual conditions, it has a càdlàg modification $\tilde{M\stoppedL_\phi}$ (see for example \cite{protter2005stochastic}, Chapter I, Corollary of Theorem 9 or \cite{karatzas1991brownian}, Chapter 1, Theorem 3.13). % if we do not assume that \bar m is càdlàg, we need to take a càdlàg modification of this martingale, which exists by means of these references

	It remains to prove the formulas for the variance. This is done in several steps, following Appendix B of \cite{debussche2017diffusion}. Since $\phi$ and $\phi^2$ belong to the domain of $\mathcal L^\epsilon$, the process
	\begin{equation*}
		N^\epsilon(t) = \int_0^t \parths{\mathcal L^\epsilon (\phi^2) - 2 \phi \mathcal L^\epsilon \phi}(X^\epsilon(s))ds = \frac{1}{\epsilon^2} \int_0^t \parths{B (\phi^2) - 2 \phi B \phi}(X^\epsilon(s))ds,
	\end{equation*}
	is well-defined.

	The proof of the second equality is straightforward: since $D = \mathcal L^\epsilon - \frac{1}{\epsilon^2}B$ is a first order differential operator, we have $D(\phi^2) - 2 \phi D \phi = 0$.

	Let $0 = t_0 < t_1 < ... < t_n = T$ be a subdivision of $[0,T]$ of step $\max \abs{t_{i+1}-t_i} = \delta$.

	\paragraph{Step 1:} We claim that the following convergence is satisfied in $\mathbb L^2 \doteq \mathbb L^2(\Omega)$
	\begin{equation} \label{eq1:proof:prop:good_test_function_martingale}
		N\stoppedL(t) = \lim_{\delta \to 0} \sum_i \esp{N\stoppedL(t \wedge t_{i+1}) - N\stoppedL(t \wedge t_i) \mid \mathcal F^\epsilon_{t_i}}.
	\end{equation}
	Let $\Delta_i \doteq N\stoppedL(t \wedge t_{i+1}) - N\stoppedL(t \wedge t_i) - \esp{N\stoppedL(t \wedge t_{i+1}) - N\stoppedL(t \wedge t_i) \mid \mathcal F^\epsilon_{t_i}}$ so that \cref{eq1:proof:prop:good_test_function_martingale} is equivalent to $\sum_i \Delta_i \xrightarrow[\delta \to 0]{} 0$ in $\mathbb L^2$. Note that $\esp{\Delta_i \Delta_j} = 0$ for $i \neq j$. Hence, we have $\esp{\abs{\sum_i \Delta_i}^2} = \esp{\sum_i \abs{\Delta_i}^2}$. Using that a conditional expectation is an orthogonal projection in $\mathbb L^2$, we get
	\begin{align*}
		\esp{\abs{\Delta_i}^2}
			&\leq \esp{\abs{N\stoppedL(t \wedge t_{i+1}) - N\stoppedL(t \wedge t_i)}^2}\\
			&\lesssim_{\epsilon} \esp{\abs{\int_{t \wedge t_{i} \wedge \tauLe}^{t \wedge t_{i+1} \wedge \tauLe} \parths{B (\phi^2) - 2 \phi B \phi}(X^\epsilon(s))ds}^2},
	\end{align*}
	By means of \cref{def:good_test_function,hyp:B_subpolynomial},
	\begin{equation*}
		\abs{\parths{B (\phi^2) - 2 \phi B \phi}(X^\epsilon(s))}^2 \lesssim_{\phi} \parths{1 + \norm{f^\epsilon(s)}_\fSet^{12}} \parths{1 + \norm{\bar m^\epsilon(s)}_E^{4(b+2)}}.
	\end{equation*}
	As in \cref{eq:subpol_condition,eq:subpol_conclusion} (using moments of order $24$ and $8(b+2)$ instead of $12$ and $4(b+2)$), \cref{prop:L2_bound,eq:taue_bounds_m} lead to
	\begin{equation*}
		\esp{\sup_{s \in [0,T]} \abs{\parths{B (\phi^2) - 2 \phi B \phi}(X\stoppedL(s))}^2} \lesssim_{\phi,\Lambda,\epsilon} 1.
	\end{equation*}
	Since $t \wedge t_{i} \wedge \tauLe - t \wedge t_{i+1} \wedge \tauLe \leq t_{i+1} - t_i$, we get
	\begin{equation*}
		\esp{\abs{\Delta_i}^2} \lesssim_{\phi,\Lambda,\epsilon} (t_{i+1} - t_i)^2,
	\end{equation*}
	which then yields $\esp{\abs{\sum_i \Delta_i}^2} \lesssim_{\phi,\Lambda,\epsilon} T \delta \xrightarrow[\delta \to 0]{} 0$, which proves \cref{eq1:proof:prop:good_test_function_martingale}.

	\paragraph{Step 2:} We claim that
	\begin{equation} \label{eq2:proof:prop:good_test_function_martingale}
		\esp{\sum_i \abs {R_{t_i,t_{i+1}}}} \lesssim_{\phi,\Lambda,\epsilon} \delta^{1/2},
	\end{equation}
	where, for $0 \leq t < t' \leq T$,
	\begin{align}
		R_{t,t'}
			&= \abs{M\stoppedL_\phi(t') - M\stoppedL_\phi(t)}^2 - \abs{\phi(X\stoppedL(t')) - \phi(X\stoppedL(t))}^2 \label{eq:def_R_step2}\\
			&= \abs{\int_{t \wedge \tauLe}^{t' \wedge \tauLe} \mathcal L^\epsilon \phi (X^\epsilon(s))ds}^2 - 2 \parths{\phi(X\stoppedL(t')) - \phi(X\stoppedL(t))} \int_{t \wedge \tauLe}^{t' \wedge \tauLe} \mathcal L^\epsilon \phi (X^\epsilon(s))ds. \nonumber
	\end{align}
	We can write
	\begin{equation} \label{eq3:proof:prop:good_test_function_martingale}
		\begin{split}
		\abs{\phi(X\stoppedL(t')) - \phi(X\stoppedL(t))}^2
			&= M\stoppedL_{\phi^2}(t') - M\stoppedL_{\phi^2}(t)\\
			&\quad - 2\phi(X\stoppedL(t))\parths{M\stoppedL_{\phi}(t') - M\stoppedL_{\phi}(t)}\\
			&\quad + \int_{t \wedge \tauLe}^{t' \wedge \tauLe} \mathcal L^\epsilon (\phi^2)(X^\epsilon(s))ds\\
			&\quad - 2 \phi(X\stoppedL(t)) \int_{t \wedge \tauLe}^{t' \wedge \tauLe} \mathcal L^\epsilon \phi(X^\epsilon(s))ds.
		\end{split}
	\end{equation}
	As established in the first part of the proof, $M\stoppedL_{\phi^2}$ and $M\stoppedL_{\phi}$ are $\seq{\mathcal F^\epsilon_s}{s}$-martingales. Moreover, $\phi(X\stoppedL_t)$ is $\mathcal F^\epsilon_t$-measurable. Thus, taking the expectation in \cref{eq3:proof:prop:good_test_function_martingale} yields
	\begin{multline*}
		\esp{\abs{\phi(X\stoppedL(t')) - \phi(X\stoppedL(t))}^2} =\\
		\esp{\int_{t \wedge \tauLe}^{t' \wedge \tauLe} \parths{\mathcal L^\epsilon (\phi^2)(X^\epsilon(s))ds - 2 \phi(X\stoppedL(t)) \mathcal L^\epsilon \phi(X^\epsilon(s))}ds}.
	\end{multline*}
	As in Step 1, by \cref{def:good_test_function,hyp:B_subpolynomial,eq:taue_bounds_m,prop:L2_bound}, the integrand is bounded by $\parths{1 + \norm{f^\epsilon_0}_\fSet^{6}} \parths{1 + \norm{\bar m(0)}_E^{2(b+2)}}$ (up to a constant depending on $\phi$, $\Lambda$ and $\epsilon$) and we get
	\begin{equation} \label{eq4:proof:prop:good_test_function_martingale}
		\esp{\abs{\phi(X\stoppedL(t')) - \phi(X\stoppedL(t))}^2} \lesssim_{\phi,\Lambda,\epsilon} t'-t,
	\end{equation}
	owing to the Cauchy-Schwarz inequality, \cref{hyp:m_L2_all_starting_point,hyp:f0}. Young's inequality with a parameter $\eta > 0$ yields
	\begin{multline*}
		\esp{\abs{R_{t,t'}}} \lesssim_{\phi,\Lambda,\epsilon} (1+ \frac 1 \eta) \esp{\abs{\int_{t \wedge \tauLe}^{t' \wedge \tauLe} \mathcal L^\epsilon \phi (X^\epsilon(s))ds}^2}\\
		+ \eta \esp{\abs{\phi(X\stoppedL(t')) - \phi(X\stoppedL(t))}^2}.
	\end{multline*}
	Similarly, we get
	\begin{equation*}
		\esp{\abs{\int_{t \wedge \tauLe}^{t' \wedge \tauLe} \mathcal L^\epsilon \phi (X^\epsilon(s))ds}^2} \lesssim_{\phi,\Lambda,\epsilon} (t'-t)^2.
	\end{equation*}
	Choosing $\eta = (t'-t)^{1/2}$ yields $\esp{\abs{R_{t,t'}}} \lesssim_{\phi,\epsilon} (t'-t)^{3/2}$, which gives \cref{eq2:proof:prop:good_test_function_martingale}.

	\paragraph{Step 3:} We claim that $\esp{\abs{M\stoppedL_\phi(t)}^2} = \esp{N\stoppedL(t)}$.

	Taking conditional expectation in \cref{eq:def_R_step2} leads to
	\begin{multline*}
		\sum_i \esp{\abs{M\stoppedL_\phi(t \wedge t_{i+1}) - M\stoppedL_\phi(t \wedge t_i)}^2 \mid \mathcal F^\epsilon_{t_i}} =\\
		\sum_i \esp{R_{t \wedge t_i, t \wedge t_{i+1}} \mid \mathcal F^\epsilon_{t_i}}\\
		+ \sum_i \esp{\abs{\phi(X\stoppedL(t \wedge t_{i+1})) - \phi(X\stoppedL(t \wedge t_i))}^2 \mid \mathcal F^\epsilon_{t_i}}.
	\end{multline*}
	Using \cref{eq3:proof:prop:good_test_function_martingale} and the martingale property on $M\stoppedL_{\phi^2}$ and $M\stoppedL_{\phi}$, the last term can be rewritten as
	\begin{equation*}
		\sum_i \int_{t \wedge t_i \wedge \tauLe}^{t \wedge t_{i+1} \wedge \tauLe} \mathcal L^\epsilon (\phi^2)(X^\epsilon(s))ds - 2 \sum_i \phi(X\stoppedL(t \wedge t_i)) \int_{t \wedge t_i \wedge \tauLe}^{t \wedge t_{i+1} \wedge \tauLe} \mathcal L^\epsilon \phi(X^\epsilon(s))ds
	\end{equation*}
	Then, \cref{eq1:proof:prop:good_test_function_martingale} yields
	\begin{equation*}
		\sum_i \esp{\abs{M\stoppedL_\phi(t \wedge t_{i+1}) - M\stoppedL_\phi(t \wedge t_i)}^2 \mid \mathcal F^\epsilon_{t_i}} = N\stoppedL(t) + r_1 + r_2
	\end{equation*}
	where
	\begin{gather*}
		r_1 = \sum_i \esp{R_{t \wedge t_i,t \wedge t_{i+1}} \mid \mathcal F^\epsilon_{t_i}}\\
		r_2 = 2 \sum_i \esp{\int_{t \wedge t_i \wedge \tauLe}^{t \wedge t_{i+1} \wedge \tauLe} \parths{\phi(X\stoppedL(s)) - \phi(X\stoppedL(t \wedge t_i))} \mathcal L^\epsilon \phi (X^\epsilon(s))ds \mid \mathcal F^\epsilon_{t_i}}.
	\end{gather*}
	By means of \cref{eq2:proof:prop:good_test_function_martingale}, $r_1 \to 0$ in $\mathbb L^1$. For $r_2$, we have
	\begin{align*}
		\esp{\abs{r_2}}
			&\leq 2 \esp{\sum_i \int_{t \wedge t_i \wedge \tauLe}^{t \wedge t_{i+1} \wedge \tauLe} \abs{\phi(X\stoppedL(s)) - \phi(X\stoppedL(t_i))} \abs{\mathcal L^\epsilon \phi (X^\epsilon(s))}ds}\\
			&\leq 2 \esp{\sum_i \int_{t_i}^{t_{i+1}} \abs{\phi(X\stoppedL(s)) - \phi(X\stoppedL(t_i))} \abs{\mathcal L^\epsilon \phi (X\stoppedL(s))}ds}\\
			&\leq 2 \sum_i \int_{t_i}^{t_{i+1}} \esp{\abs{\phi(X\stoppedL(s)) - \phi(X\stoppedL(t_i))} \abs{\mathcal L^\epsilon \phi (X\stoppedL(s))}}ds\\
			&\leq 2 \sum_i \int_{t_i}^{t_{i+1}} \esp{\abs{\phi(X\stoppedL(s)) - \phi(X\stoppedL(t_i))}^2}^{1/2} \esp{\abs{\mathcal L^\epsilon \phi (X\stoppedL(s))}^2}^{1/2}ds.
	\end{align*}
	As above, one can show $\esp{\sup_s \abs{\mathcal L^\epsilon \phi (X\stoppedL(s))}^2} \lesssim_{\phi,\Lambda,\epsilon} 1$. Thus, \cref{eq4:proof:prop:good_test_function_martingale} yields
	\begin{align*}
		\esp{\abs{r_2}}
			&\lesssim_{\phi,\Lambda,\epsilon} \sum_i \int_{t_i}^{t_{i+1}} \esp{\abs{\phi(X\stoppedL(s)) - \phi(X\stoppedL(t_i))}^2}^{1/2}ds\\
			&\lesssim_{\phi,\Lambda,\epsilon} \sum_i \int_{t_i}^{t_{i+1}} \parths{s - t_i}^{1/2}ds\\
			&\lesssim_{\phi,\Lambda,\epsilon} \sum_i \parths{t_{i+1} - t_i}^{3/2}ds\\
			&\lesssim_{\phi,\Lambda,\epsilon} T \delta^{1/2} \xrightarrow[\delta \to 0]{} 0.
	\end{align*}

	Thus, in $\mathbb L^1$, we have
	\begin{equation} \label{eq:temptemptemp}
		\lim_{\delta \to 0} \sum_i \esp{\abs{M\stoppedL_\phi(t \wedge t_{i+1}) - M\stoppedL_\phi(t \wedge t_i)}^2 \mid \mathcal F^\epsilon_{t_i}} = N\stoppedL(t).
	\end{equation}
	In particular, the expectation converges. Then, the martingale property and the tower property $\esp{\esp{\cdot \mid \mathcal F_s}} = \esp{\cdot}$ yield
	\begin{align*}
		\esp{N\stoppedL(t)}
			&= \lim_{\delta \to 0} \esp{\sum_i \esp{\abs{M\stoppedL_\phi(t \wedge t_{i+1}) - M\stoppedL_\phi(t \wedge t_i)}^2 \mid \mathcal F^\epsilon_{t_i}}}\\
			&= \lim_{\delta \to 0} \esp{\sum_i \esp{\abs{M\stoppedL_\phi(t \wedge t_{i+1})}^2 - \abs{M\stoppedL_\phi(t \wedge t_i)}^2 \mid \mathcal F^\epsilon_{t_i}}}\\
			&= \lim_{\delta \to 0} \sum_i \esp{\abs{M\stoppedL_\phi(t \wedge t_{i+1})}^2 - \abs{M\stoppedL_\phi(t \wedge t_i)}^2}\\
			&= \esp{\abs{M\stoppedL_\phi(t)}^2}.
	\end{align*}
	 This conclude the proof that
 	\begin{equation*}
 		\forall t \in \tSet, \esp{\abs{M\stoppedL_\phi(t)}^2} = \esp{\int_0^{t \wedge \tauLe} \parths{\mathcal L^\epsilon (\phi^2) - 2 \phi \mathcal L^\epsilon \phi}(X^\epsilon(s))ds}.
 	\end{equation*}
	 and the proof of \cref{prop:good_test_function_martingale}.
\end{proof}

\begin{remark}
	Note that if $\bar m$ had continuous paths, then $M^\epsilon_\phi$ would be a continuous martingale and \cref{eq:temptemptemp} would mean that $N\stoppedL$ is the quadratic variation of $M\stoppedL_\phi$.
\end{remark}

A similar proof leads to the following Proposition, where we take weaker stopping times but add some conditions on $\phi$. The proof is omitted.
\begin{proposition} \label{prop:good_test_function_martingale_cases}
	Let $\phi$ be a good test function. The conclusion of \cref{prop:good_test_function_martingale} holds in the following cases.
	\begin{itemize}
		\item The function $\phi$ does not depend on $f$ and $\tauLe$ is replaced by $\taue$.
		\item The function $\phi$ is bounded uniformly in $n$ and does not depend on $z$ and $\tauLe$ is replaced by $\tauLze$.
		\item The function $\phi$ is bounded and depends only on $n$ and $\tauLe$ is replaced by $+\infty$.
	\end{itemize}
\end{proposition}

\subsection{The perturbed test functions method} \label{subsec:perturbed_test_functions}

We use the perturbed test functions method as in \cite{papanicolaou1977martingale} to exhibit a generator $\mathcal L$ such that a possible limit point $(\rho_\Lambda,\zeta_\Lambda)$ of $\seq{(\rho\stoppedL,\zeta\stoppedL)}{\epsilon}$ solves the martingale problem associated to $\mathcal L$ until some limit stopping time depending on $\Lambda$. Given a test function $\phi$, two corrector functions $\phi_1$ and $\phi_2$ are constructed, so that
\begin{equation} \label{eq:def_phi_epsilon}
	\forall (f,z,n) \in \fSet \times C^1_x \times E, \phi^\epsilon(f,z,n) = \phi(\rho,z) + \epsilon \phi_1(f,z,n) + \epsilon^2 \phi_2(f,z,n),
\end{equation}
satisfies
\begin{equation} \label{eq:precontrol_perturbed_generator}
	\mathcal L^\epsilon \phi^\epsilon = \mathcal L \phi + o(1),
\end{equation}
when $\epsilon \to 0$. Then, we prove that $\phi^\epsilon$ is a good test function and that we can take the limit when $\epsilon \to 0$ in the martingale problem associated to $\mathcal L^\epsilon$ (\cref{prop:good_test_function_martingale}) to obtain a stopped martingale problem solved by a limit point.

Based on the decomposition \cref{eq:def_L_epsilon}, a sufficient condition to prove \cref{eq:precontrol_perturbed_generator} for $\phi^\epsilon$ of the form \cref{eq:def_phi_epsilon} is to solve the following equations \cref{eq:order-2,eq:order-1,eq:order0} and to check that \cref{eq:order1} holds when $\epsilon \to 0$.
\begin{align}
	\mathcal L_2 \phi &= 0 \label{eq:order-2}\\
	\mathcal L_1 \phi + \mathcal L_2 \phi_1 &= 0 \label{eq:order-1}\\
	\mathcal L_1 \phi_1 + \mathcal L_2 \phi_2 &= \mathcal L\phi \label{eq:order0}\\
	\mathcal L_1 \phi_2 &= O(1). \label{eq:order1}
\end{align}
The properties of the resolvent operators $R_\lambda$ are employed to invert $\mathcal L_2$.

\subsubsection{Framework for the perturbed test functions method}

For a martingale problem to be relevant, it is sufficient that the class of test functions satisfying the martingale problem is separating, namely that if some random variables $X$ and $X'$ satisfy $\esp{\phi(X)} = \esp{\phi(X')}$ for all $\phi \in \Phi$, then we have $X \stackrel{d}{=} X'$. In this work, we use the following class
\begin{equation*} \label{eq:def_regular_test_functions}
	\Theta = \set{(\rho,z) \mapsto \psi\parths{\scalrho{\rho}{\xi}} \chi(z) \mid \psi \in C^3(\mathbb R), \psi'' \in C^1_b(\mathbb R), \xi \in H^3_x, \chi \in C^3_b(C^1_x)},
\end{equation*}
where $\rho = \int_V f d\mu$. The class $\Theta$ is indeed separating because it separates points (see \cite{ethier1986markov}, Theorem 4.5).

Note that the test functions depend only on $\rho$ and $z$, because we expect the limit equation to be satisfied by $\rho$ and $z$. It is confirmed by \cref{subsubsec:order-2}. To simplify the notation, for $\phi \in \Theta$, we sometimes write $\phi(f,z,n) \doteq \phi(\rho,z)$ and $\phi(\rho,z) = \Psi(\rho) \chi(z)$, where $\Psi(\rho) = \psi\parths{\scalrho{\rho}{\xi}}$.

\begin{proposition} \label{prop:corrector}
	There exists an operator $\mathcal L$ whose domain contains $\Theta$ and, for all $\phi \in \Theta$, there exist two good test functions $\phi_1$ and $\phi_2$ such that, for all $(f,z,n) \in \fSet \times C^1_x \times E$, we have
	\begin{gather}
		\label{eq:control_phi_1}
		\abs{\phi_1(f,z,n)} \lesssim_\phi (1 + \norm{f}_\fSet^2) (1 + \norm{n}_E)\\
		\label{eq:control_phi_2}
		\abs{\phi_2 (f,z,n)} \lesssim_\phi (1 + \norm{f}_\fSet^2) (1 + \norm{n}_E^{b+1})\\
		\label{eq:control_perturbed_generator}
		\abs{\mathcal L^\epsilon \phi^\epsilon - \mathcal L \phi}(f,z,n) \lesssim \epsilon (1 + \norm{f}_\fSet^3) (1 + \norm{n}_E^{b+2}).
	\end{gather}
	Moreover, $\phi^\epsilon = \phi + \epsilon \phi_1 + \epsilon^2 \phi_2$ is a good test function.

	Moreover, if $\phi$ depends only on $z$, then $\phi_1$, $\phi_2$ and $\phi^\epsilon$ depend only on $z$ and $n$.
\end{proposition}

\subsubsection{Consistency result} \label{subsubsec:order-2}

Since we already expect the limit equation to be satisfied by $\rho$, equation \cref{eq:order-2} will not give us extra information. Hence, this section only present a consistency result, namely that \cref{eq:order-2} forces $\phi$ to depend on $f$ through $\rho$.

In fact, let $\phi$ depend on $f$ and $z$ but not on $n$.
Since $\phi$ does not depend on $n$, $B\phi = 0$. Hence, \cref{eq:order-2} can be written, for all $f \in \fSet$ and $z \in C^1_x$,
\begin{equation} \label{eq:reexpression_of_order-2}
	D_f \phi(f,z,n)(Lf) = 0.
\end{equation}
For $t \in \tSet$ and $f \in \fSet$, define
\begin{equation} \label{eq:def_g}
	g(t,f) = \rho \mathcal M + e^{-t}(f-\rho \mathcal M),
\end{equation}
and observe that $\partial_t g(t,f) = Lg(t,f)$ with $g(0,f) = f$. Owing to \cref{eq:reexpression_of_order-2}, the mapping $t \mapsto \phi(g(t,f),z)$ is constant. Since $g(t,f) \xrightarrow[t \to \infty]{} \rho \mathcal M$, by continuity of $\phi$, we get $\phi(f,z,n) = \phi(\rho \mathcal M,z,n)$, which depends on $f$ only through $\rho$.

\subsubsection{Construction of the first corrector function $\phi_1$} \label{subsubsec:order-1}

The first corrector function $\phi_1$ is defined as the solution of \cref{eq:order-1}: the formal solution to Poisson equation will provide an expression for $\phi_1$, then we will check that this expression indeed solves \cref{eq:order-1}.

Let $g(t,f)$ be defined by \cref{eq:def_g} and $m(t,n)$ be defined in \cref{subsec:driving_random_term} (Markov process of infinitesimal generator $B$ starting from $n$). The process $\seq{(g(t,f),z,m(t,n))}{t \in \tSet}$ is a $\fSet \times C^1_x \times E$-valued Markov process of generator $\mathcal L_2$ starting from $(f,z,n)$. Denote by $\seq{Q_t}{t \in \tSet}$ its transition semi-group. Note that this semi-group does not have a unique invariant distribution, since for any $\rho$ fixed, $\delta_{\rho \mathcal M} \otimes \delta_z \otimes \nu$ is an invariant distribution. However on every space $\set{(f',z',n) \in \fSet \times C^1_x \times E \mid \int_\vSet f' d\mu = \rho, z'=z}$, this measure is the unique invariant distribution. Indeed, $\delta_{\rho \mathcal M}$, $\delta_z$ and $\nu$ are respectively the unique invariant distributions of each marginal process (on the corresponding subspaces), and $\delta_{\rho \mathcal M} \otimes \delta_z \otimes \nu$ is the only coupling of these three marginal distributions.

For $\Phi : \fSet \times C^1_x \times E \to \mathbb R$, denote by
\begin{equation*}
	\lrangle \Phi _{\rho,z} \doteq \int_E \Phi(\rho \mathcal M,z,n)d\nu(n) = \int_{\fSet \times C^1_x \times E} \Phi d(\delta_{\rho \mathcal M} \otimes \delta_z \otimes \nu)
\end{equation*}
the integral against this invariant distribution. For $\phi \in \Theta$, $\phi(\rho,z) = \Psi(\rho) \chi(z)$, let us compute $\mathcal L_1 \phi$.

We have
\begin{align*}
	\mathcal L_1\phi(f,z,n)
		&= D_f\phi(f,z,n)(- Af + nf) + D_z\phi(f,z,n)(n)\\
		&= D\Psi(\rho)(-\bar{Af} + n \rho) \chi(z) + \Psi(\rho) D\chi(z)(n),
\end{align*}
where $\bar h = \int_\xSet h(x)dx$.

Owing to \cref{eq:aM_centered}, one can write, for all $\rho \in {L^2_x}$, $\bar{A(\rho \mathcal M)} = 0$. Moreover, since $\nu$ is centered by \cref{hyp:m_centered}, any term linear in $n$ vanishes when integrating with respect to $\nu$. Hence, we have checked that
\begin{equation*}
	\forall \rho \in {L^2_x}, \forall z \in C^1_x, \lrangle{\mathcal L_1\phi}_{\rho,z} = 0.
\end{equation*}

Using the expansion of $\mathcal L_1 \phi$, for all $f, z, n$, we have
\begin{align*}
	\int_0^\infty Q_t \mathcal L_1 \phi(f,z,n) dt
		&= \int_0^\infty \esp{\mathcal L_1 \phi \parths{g(t,f),z,m(t,n)}} dt\\
	\begin{split}
		&= \int_0^\infty \parths{- D\Psi(\rho)(\bar{Ag(t,f)}) \chi(z) \right.\\
		&\quad + \esp{D\Psi(\rho)(\rho m(t,n))} \chi(z)\\
		&\quad \left. + \Psi(\rho) D\chi(z)(m(t,n)) \vphantom{- D\Psi(\rho)(\bar{Ag(t,f)}) \chi(z)}} dt,
	\end{split}
\end{align*}
owing to the identity $\bar{g(t,f)} = \rho$. Equation \cref{eq:def_g} yields
\begin{equation*}
	\bar{Ag(t,f)} = e^{-t} \bar{Af},
\end{equation*}
since $\bar{A\rho \mathcal M} = 0$. Thus, owing to \cref{def:resolvent}, we define
\begin{align}
	\phi_1(f,z,n)
		&= \int_0^\infty Q_t \mathcal L_1 \phi(f,z,n) dt \nonumber\\
		&= D\Psi(\rho)(- \bar{Af} + R_0(n)\rho) \chi(z) + \Psi(\rho) R_0\lrbracket{D\chi(z)}(n). \label{eq:def_phi_1}
\end{align}
It is straightforward to check that $\phi_1$ defined by \cref{eq:def_phi_1} solves \cref{eq:order-1}. Moreover, it satisfies the condition \cref{eq:control_phi_1}. It remains to prove that $\phi_1$ is a good test function. Owing to \cref{hyp:B_continuity} and \cref{eq:def_phi_1}, $\phi_1 \in D(B)$ and $\phi_1^2 \in D(B)$. For $h \in \fSet$, we have
\begin{multline*}
	D_f \phi_1(f,z,n)(h) = D^2\Psi(\rho)(- \bar{Af} + R_0(n)\rho,\bar h) \chi(z) + D\Psi(\rho)(- \bar{Ah} + R_0(n)\bar h) \chi(z)\\
	+ D\Psi(\rho)(\bar h) R_0\lrbracket{D\chi(z)}(n),
\end{multline*}
hence $D_f \phi_1(f,z,n)(Ah)$ is well-defined (as in \cref{def:good_test_function}) and $\phi_1$, $D_f \phi_1(f,z,n)(h)$ and $D_f \phi_1(f,z,n)(Ah)$ have at most polynomial growth in the sense of \cref{eq:estimates_good_test_function}. For $n_2 \in E$, we have
\begin{multline*}
	D_z \phi_1(f,z,n)(n_2) = D\Psi(\rho)(- \bar{Af} + R_0(n)\rho) D\chi(z)(n_2)\\
	+ \Psi(\rho) D \lrbracket{z' \mapsto R_0\lrbracket{D\chi(z')}(n)}(z)(n_2).
\end{multline*}
Using \cref{lem:Pt_lip_Pt_quad} and the assumption $\chi \in C^3_b(C^1_x)$, we write
\begin{align*}
	D \lrbracket{R_0\lrbracket{D\chi(\cdot)}(n)}(z)(n_2)
		&= D \lrbracket{z' \mapsto \int_0^\infty P_t D\chi(z')(n) dt}(z)(n_2)\\
		&= \int_0^\infty P_t \lrbracket{D^2 \chi(z)(\cdot,n_2)}(n)\\
		&= R_0\lrbracket{D^2\chi(z)(\cdot,n_2)}(n).
\end{align*}
This leads to
\begin{equation*}
	D_z \phi_1(f,z,n)(n_2) = D\Psi(\rho)(- \bar{Af} + R_0(n)\rho) D\chi(z)(n_2) + \Psi(\rho) R_0\lrbracket{D^2\chi(z)(\cdot,n_2)}(n).
\end{equation*}
Once again using \cref{lem:Pt_lip_Pt_quad} and that $\chi \in C^3_b(C^1_x)$, one checks that $D_z \phi_1$ has at most polynomial growth in the sense of \cref{eq:estimates_good_test_function}. Thus $\phi_1$ satisfies \cref{eq:estimates_good_test_function}.

\subsubsection{Construction of the second corrector function $\phi_2$}

The second corrector $\phi_2$ is defined as a solution of \cref{eq:order0}. To solve \cref{eq:order0}, we need the centering condition $\lrangle{\mathcal L\phi - \mathcal L_1 \phi_1}_{\rho,z} = 0$. This identity will be the definition of $\mathcal L \phi$.

First, let us compute $\mathcal L_1 \phi_1$. Using the derivative calculated in \cref{eq:def_phi_1}, $\mathcal L_1 \phi_1$ can be written as
\begin{equation*}
	\mathcal L_1 \phi_1(f,z,n) = c(f,z) + \ell(f,z,n) + q(f,z,n)
\end{equation*}
where $c$, $\ell$ and $q$ are defined by
\begin{gather}
	\begin{split} \label{eq:def_c}
	c(f,z)
		&= D^2\Psi(\rho)(\bar{Af},\bar{Af}) \chi(z)
		+ D\Psi(\rho)(\bar{A^2f}) \chi(z)
	\end{split}\\
	\begin{split} \label{eq:def_l}
	\ell(f,z,n)
		&= - D^2\Psi(\rho)(\bar{Af},R_0(n)\rho + n \rho) \chi(z)\\
		&\quad - D\Psi(\rho)(\bar{A(nf)} + R_0(n)(\bar{Af})) \chi(z)\\
		&\quad - D\Psi(\rho)(\bar{Af}) \parths{R_0\lrbracket{D\chi(z)}(n) + D\chi(z)(n)}
	\end{split}\\
	\begin{split} \label{eq:def_q}
	q(f,z,n)
		&= D^2\Psi(\rho)(n\rho, R_0(n)\rho) \chi(z)
		+ D\Psi(\rho)(R_0(n)(n\rho)) \chi(z)\\
		&\quad + D\Psi(\rho)(n\rho) R_0\lrbracket{D\chi(z)}(n)
		+ D\Psi(\rho)(R_0(n)\rho) D\chi(z)(n)\\
		&\quad + \Psi(\rho) R_0\lrbracket{D^2\chi(z)(\cdot,n)}(n).
	\end{split}
\end{gather}
Note that, for fixed $f$ and $z$, $c$ does not depend on $n$, $\ell$ is pseudo-linear in $n$ and $q$ is pseudo-quadratic in $n$ as introduced in \cref{def:pseudo-linear_pseudo-quadratic}.

The function $\ell(f,z,\cdot)$ is indeed pseudo-linear as a sum of continuous linear and pseudo-linear forms, yielding $\lrangle{\ell}_{\rho,z} = 0$ for all $\rho$ and $z$. Using also that $\bar{A\rho \mathcal M} = 0$, we get an explicit definition of $\mathcal L$:
\begin{equation} \label{eq:def_L_phi}
	\begin{split}
	\mathcal L \phi(\rho,z) \doteq \lrangle{\mathcal L_1 \phi_1}_{\rho,z}
		&= D\Psi(\rho)(\bar{A^2\rho \mathcal M}) \chi(z)\\
		&\quad + \int D^2\Psi(\rho)(n\rho,R_0(n)\rho) d\nu(n) \chi(z)\\
		&\quad + \int D\Psi(\rho)(R_0(n)(n\rho)) d\nu(n) \chi(z)\\
		&\quad + \int D\Psi(\rho)(n\rho) R_0\lrbracket{D\chi(z)}(n) d\nu(n)\\
		&\quad + \int D\Psi(\rho)(R_0(n)\rho) D\chi(z)(n) d\nu(n)\\
		&\quad + \Psi(\rho) \int R_0\lrbracket{D^2\chi(z)(\cdot,n)}(n) d\nu(n).
	\end{split}
\end{equation}
Note that by taking $\chi = 1$, we obtain the same expression of $\mathcal L$ as in \cite{debussche2011diffusion}.

Since the centering condition for the Poisson equation \cref{eq:order0} is satisfied by construction of $\mathcal L$, the second corrector function $\phi_2$ can be defined as follows: for all $f,z,n$,
\begin{align*}
	\phi_2 (f,z,n)
		&\doteq \int_0^\infty Q_t \parths{\mathcal L_1\phi_1 - \lrangle{\mathcal L_1 \phi_1}_{\rho,z}} (f,z,n) dt\\
	\begin{split}
		&= \int_0^\infty Q_t \parths{c - \lrangle c_{\rho,z}} (f,z,n) dt\\
		&\quad + \int_0^\infty Q_t \ell(f,z,n) dt\\
		&\quad + \int_0^\infty Q_t \parths{q - \lrangle q_{\rho,z}} (f,z,n) dt
	\end{split}\\
		&\doteq \phi_2^c(f,z,n) + \phi_2^\ell(f,z,n) + \phi_2^q(f,z,n).
\end{align*}
Once again, one can check that $\phi_2$ satisfies \cref{eq:order0}. It only remains to prove \cref{eq:control_phi_2}, \cref{eq:control_perturbed_generator} and that $\phi^\epsilon$ is a good test function. Since
\begin{equation} \label{eq:diff_Lphie_Lphi}
	\mathcal L^\epsilon \phi^\epsilon = \mathcal L \phi + \epsilon \mathcal L_1 \phi_2,
\end{equation}
equation \cref{eq:control_perturbed_generator} comes from an estimate on $\mathcal L_1 \phi_2(f,n)$ in terms of $f$, $n$, and $\phi$.

\subsubsection{Controls on the second corrector function} \label{subsubsec:control_phi_2}

The aim of this section is to prove some estimates for $\phi_2(f,z,n)$ and its derivatives to establish that \cref{eq:estimates_good_test_function,eq:control_perturbed_generator,eq:control_phi_2} are satisfied. Let $f, h \in \fSet$, $z \in C^1_x$ and $n, n_2 \in E$ and let $S_1 = \norm{f}_\fSet \vee \norm{h}_\fSet$ and $S_2 = \norm{n}_E \vee \norm{n_2}_E$.

\paragraph{Estimates on $\phi_2^c$}

We have, using $\lrangle{c}_{\rho,z} = c(\rho \mathcal M, z)$,
\begin{equation*}
	c(f,z) - c(\rho \mathcal M,z) = D^2\Psi(\rho)(\bar{Af},\bar{Af}) \chi(z)
	+ D\Psi(\rho)(\bar{A^2(f-\rho \mathcal M)}) \chi(z).
\end{equation*}
Recall that $\bar{Ag(t,f)} = e^{-t} \bar{Af}$. Hence, using \cref{eq:def_g}, we get
\begin{align*}
	Q_t \parths{c - \lrangle c_{\rho,z}} (f,z,n)
		&= \esp{c(g(t,f),z) - c(\rho \mathcal M,z)}\\
	\begin{split}
		&= e^{-2t} D^2\Psi(\rho)(\bar{Af},\bar{Af}) \chi(z)\\
		&\quad + e^{-t} D\Psi(\rho)(\bar{A^2(f-\rho \mathcal M)}) \chi(z).
	\end{split}
\end{align*}
By integration, we get
\begin{equation} \label{eq:phi_2_c}
	\phi_2^c (f,z,n) = \frac 1 2 D^2\Psi(\rho)(\bar{Af},\bar{Af}) \chi(z) + D\Psi(\rho)(\bar{A^2(f-\rho \mathcal M)}) \chi(z).
\end{equation}
Moreover, we obtain
\begin{gather*}
	\begin{split}
	D_f \phi_2^c (f,z,n)(h)
		&= \frac 1 2 D^3\Psi(\rho)(\bar{Af},\bar{Af},\bar h) \chi(z)\\
		&\quad + D^2\Psi(\rho)(\bar{Af},\bar{Ah}) \chi(z)\\
		&\quad + D^2\Psi(\rho)(\bar{A^2(f-\rho \mathcal M)},\bar h) \chi(z)\\
		&\quad + D\Psi(\rho)(\bar{A^2(h - \bar h \mathcal M)}) \chi(z),
	\end{split}\\
	\begin{split}
	D_z \phi_2^c (f,z,n)(n_2)
		&= \frac 1 2 D\Psi(\rho)(\bar{Af},\bar{Af}) D\chi(z)(n)\\
		&\quad + D\Psi(\rho)(\bar{A^2(f - \rho \mathcal M)}) D\chi(z)(n).
	\end{split}
\end{gather*}
Recall that $\norm{f - \rho \mathcal M}_\fSet^2 + \norm{\rho}_{L^2_x}^2 = \norm{f}_\fSet^2$, hence $\norm{f - \rho \mathcal M}_\fSet \leq \norm{f}_\fSet$. Then, since $\Psi(\rho) = \psi\parths{\scalrho{\rho}{\xi}}$ and $\psi'' \in C^1_b(\mathbb R)$, we get that $\phi_2^c$ satisfies \cref{eq:estimates_good_test_function}. More precisely, the following estimates hold:
\begin{gather*}
	\abs{\phi_2^c (f,z,n)} \lesssim_\phi 1 + \norm{f}_\fSet^2\\
	\abs{\mathcal L_1 \phi_2^c (f,z,n)} \lesssim_\phi (1 + \norm{f}_\fSet^3) (1+\norm{n}_E).
\end{gather*}

\paragraph{Estimates on $\phi_2^\ell$}

Using \cref{eq:def_g}, \cref{eq:def_l} and that $\bar{A\rho \mathcal M} = \bar{A(n\rho) \mathcal M} = 0$, we get
\begin{equation*}
	\forall (f,z,n), \ell(g(t,f),z,n) = e^{-t} \ell(f,z,n).
\end{equation*}
Thus, we have
\begin{align*}
	Q_t \ell(f,z,n) = \esp{\ell(g(t,f),z,m(t,n))} = e^{-t} \esp{\ell(f,z,m(t,n))} = e^{-t} P_t \ell(f,z,n),
\end{align*}
and by integrating with respect to $t$, we get
\begin{equation*}
	\phi_2^\ell (f,z,n) = R_1 \ell(f,z,n).
\end{equation*}
Moreover, from \cref{lem:Pt_lip_Pt_quad,eq:def_l}, it is straightforward to check that
\begin{equation*}
	\lrbracket{\ell(f,\cdot)}_{\Lip} \lesssim_\phi (1 + \norm{f}_\fSet^2).
\end{equation*}
Hence, \cref{lem:Pt_lip_Pt_quad} yields
\begin{equation*}
	\abs{\phi_2^\ell (f,z,n)} \lesssim_\phi (1 + \norm{f}_\fSet^2) (1 + \norm{n}_E).
\end{equation*}

Since the operator $R_1$ acts only on the variable $n$, it commutes with the derivatives $D_f$ and $D_z$ in the following sense:
\begin{gather*}
	D_f \lrbracket{R_1 \ell}(f,z,n)(h) = R_1 \lrbracket{D_f \ell(f,z,\cdot)(h)}(n)\\
	D_z \lrbracket{R_1 \ell}(f,z,n)(n_2) = R_1 \lrbracket{D_z \ell(f,z,\cdot)(n_2)}(n).
\end{gather*}
Thus, after calculating the derivatives of $\ell$, we get estimates on the derivatives of $\phi_2^\ell$ the same way we got estimates on $\phi_2^\ell$. This leads to
\begin{gather*}
	\abs{D_f \phi_2^\ell (f,z,n)(Ah)} \lesssim_\phi (1 + S_1^3) (1 + S_2)\\
	\abs{D_f \phi_2^\ell (f,z,n)(n_2 f)} \lesssim_\phi (1 + S_1^3) (1 + S_2^2)\\
	\abs{D_z \phi_2^\ell (f,z,n)(n_2)} \lesssim_\phi (1 + S_1^2) (1 + S_2^2),
\end{gather*}
hence $\phi_2^\ell$ satisfies \cref{eq:estimates_good_test_function}. Finally, the following estimates hold
\begin{gather*}
	\abs{\phi_2^\ell (f,z,n)} \lesssim_\phi (1 + \norm{f}_\fSet^2) (1 + \norm{n}_E)\\
	\abs{\mathcal L_1 \phi_2^\ell (f,z,n)} \lesssim_\phi (1 + \norm{f}_\fSet^3) (1+\norm{n}_E^2).
\end{gather*}

\paragraph{Estimates on $\phi_2^q$}

The function $q$ depends of $f$ only through $\rho$. Since $\bar{g(t,f)} = \rho$ does not depend on $t$, we get $Q_t q = P_t q$ and
\begin{equation*}
	\phi_2^q (f,z,n) = R_0 \lrbracket{q - \lrangle{q}_{\rho,z}} (f,z,n)
\end{equation*}

It is straightforward to compute the derivatives of $q$ with respect to $f$ and $z$ from \cref{eq:def_q}. One can deduce estimates for $\lrbracket{q(f,z,\cdot)}_{\psq}$ and for the first order derivatives $\lrbracket{D_f q(f,z,\cdot)(n_2 f)}_{\psq}$, $\lrbracket{D_f q(f,z,\cdot)(Af)}_{\psq}$ and $\lrbracket{D_z q(f,z,\cdot)(n_2)}_{\psq}$. Reasoning as for $R_1$, the resolvent $R_0$ acts only on $n$, and thus commutes with $D_f$ and $D_z$. Thus, \cref{lem:Pt_lip_Pt_quad} with $\lambda = 0$ proves that $\phi_2^q$ satisfies \cref{eq:estimates_good_test_function}. Finally, the following estimates hold
\begin{gather*}
	\abs{\phi_2^q (f,z,n)} \lesssim_\phi (1 + \norm{f}_\fSet^2) (1 + \norm{n}_E^{b+1})\\
	\abs{\mathcal L_1 \phi_2^q (f,z,n)} \lesssim_\phi (1 + \norm{f}_\fSet^3) (1 + \norm{n}_E^{b+2}).
\end{gather*}

This concludes the proof that $\phi_2$ satisfies \cref{eq:estimates_good_test_function} and the proof of the estimates of \cref{prop:corrector} on $\phi_2$ and $\mathcal L_1 \phi_2$.

\subsubsection{Good test function property}

It only remains to prove that $\phi^\epsilon$ is a good test function. The estimates \cref{eq:estimates_good_test_function} are satisfied by $\epsilon \phi_1$ and $\epsilon^2 \phi_2$, hence by their sum $\phi^\epsilon$. Moreover, using the notation introduced in \cref{subsubsec:resolvent}, $\phi^\epsilon$ can be written as
\begin{equation*}
	\begin{split}
	\phi^\epsilon(f,z,n)
		&= \phi(\rho,z) - \epsilon D\Psi(\rho)(\bar{Af}) \chi(z) + \epsilon R_0\lrbracket{D\Psi(\rho)(\cdot \rho)}(n) \chi(z)\\
		&\quad + \epsilon \Psi(\rho) R_0 \lrbracket{D\chi(z)}(n) + \epsilon^2 \phi_2^c(f,z) + \epsilon^2 R_1\ell(f,z,n)\\
		&\quad + \epsilon^2 R_0 \lrbracket{q - \lrangle{q}_{\rho,z}}(f,z,n).
	\end{split}
\end{equation*}
Observe that each term either does not depend on $n$ or can be written $R_\lambda \theta$ with $\theta$ as in \cref{def:resolvent}. As a consequence, owing to \cref{hyp:B_continuity}, any product of at most two of these terms belongs to $D(B)$. Thus, $\phi^\epsilon \in D(B)$ and $(\phi^\epsilon)^2 \in D(B)$. This concludes the proof that $\phi^\epsilon$ is a good test function, and the proof of \cref{prop:corrector}.

\section{Dynamics associated with the limiting equation} \label{sec:limit_equation}

In this section, we show that the operator $\mathcal L$ is the generator of the limit equation \cref{eq:limit_equation} and that the martingale problem associated to $\mathcal L$ characterizes the solution of \cref{eq:limit_equation}.

\begin{definition} \label{def:limit_equation}
	Let $\rho_0 \in L^2_x$ and let $\sigma > 0$. A process $(\rho,\zeta)$ is said to be a weak solution to \cref{eq:limit_equation} in $L^2_x$ if the following assertions are satisfied
	\begin{enumerate}[label=(\roman*)]
		\item $\rho(0) = \rho_0$,
		\item $\rho \in L^\infty([0,T],L^2_x) \cap C([0,T],H^{-\sigma}_x)$ a.s. and $\zeta \in C([0,T],C^1_x)$ a.s.,
		\item there exists $\seq{B_i}{i}$ a sequence of independent standard Brownian motions such that $(\rho,\zeta)$ is adapted to the filtration generated by $\seq{B_i}{i}$ and such that, for all $\xi \in L^2_x$ and $t \in [0,T]$, we have a.s.
		\begin{gather}
			\begin{split}
			\scalrho{\rho(t)}{\xi}
				&= \scalrho{\rho_0}{\xi} + \int_0^t \scalrho{\rho(s)}{\div(K \nabla \xi)} ds + \int_0^t \scalrho{\frac 1 2 F \rho(s)}{\xi} ds \label{eq:limit_equation_weak}\\
				&\quad + \sum_i \sqrt{q_i} \int_0^t \scalrho{F_i \rho(s)}{\xi} dB_i(s)
			\end{split}\\
			\zeta(t) = \sum_i \sqrt{q_i} F_i B_i(t). \label{eq:limit_equation_weak_zeta}
		\end{gather}
	\end{enumerate}
\end{definition}

Note that the sum in \cref{eq:limit_equation_weak_zeta} does converge in $C([0,T],C^1_x)$ owing to \cref{lem:qiFi_summable}.

The solution to this equation exists and is unique in distribution. The existence can be proved using energy estimates, Itô formula and regularization argument. The uniqueness comes from pathwise uniqueness which derives from the same arguments. We do not give details concerning existence and uniqueness, however, in the proof of the following Proposition, we established the aforementioned energy estimate \ref{eq:moment_limit_solution}.

\begin{proposition} \label{prop:limit_generator}
	Let $\sigma > 0$ and let $(\rho,\zeta) \in C([0,T],H^{-\sigma}_x) \times C([0,T],C^1_x)$.

	If $(\rho,\zeta)$ is the weak solution to \cref{eq:limit_equation} in $H^{-\sigma}_x$, then, for any test function $\phi \in \Theta$, the process
	\begin{equation*}
		M_\phi(t) = \phi(\rho(t),\zeta(t)) - \phi(\rho_0,0) - \int_0^t \mathcal L \phi(\rho(s),\zeta(s)) ds
	\end{equation*}
	is a martingale for the filtration generated by $(\rho,\zeta)$.

	Conversely, if for all $\phi \in \Theta$, $M_\phi$ and $M_{\phi^2}$ are martingales, then $(\rho,\zeta)$ is the weak solution of \cref{eq:limit_equation} in $H^{-\sigma}_x$.
\end{proposition}

\begin{proof}
	Let us first prove that $\mathcal L$ is the generator associated to \cref{eq:limit_equation}. The expression of $\mathcal L \phi$ is given by \cref{eq:def_L_phi}. First note that $\bar{A^2\rho \mathcal M} = \div(K \nabla \rho)$, which is the first term of \cref{eq:limit_equation}. The third term of \cref{eq:def_L_phi} is associated to the second term of \cref{eq:limit_equation}:
	\begin{align*}
		\int D\Psi(\rho)(R_0(n)(n\rho)) d\nu(n)
			&= \esp{\int_0^\infty D\Psi(\rho)(\rho \bar m(0) \bar m(t)) dt}\\
			&= \frac 1 2 \esp{\int_{\mathbb R} D\Psi(\rho)(\rho \bar m(0) \bar m(t)) dt}\\
			&= \frac 1 2 D\Psi(\rho)(\rho F).
	\end{align*}
	To rewrite the second term of \cref{eq:def_L_phi}, assume first that the bilinear form $D^2\Psi(\rho)$ on $L^2_x$ admits a kernel $k_\rho$. Then, we have
	\begin{align*}
		\int D^2\Psi(\rho)(n\rho,R_0(n)\rho) d\nu(n)
			&= \frac 1 2 \esp{\int_0^\infty D^2\Psi(\rho)(\rho \bar m(0), \rho \bar m(t)) dt}\\
			&= \frac 1 2 \esp{\int_0^\infty \int\int k_\rho(x,y) \rho(x) \bar m(0)(x) \rho(y) \bar m(t)(y) dxdydt}\\
			&= \frac 1 2 \int\int k_\rho(x,y) k(x,y) \rho(x) \rho(y) dxdy.
	\end{align*}
	Owing to Mercer's Theorem (see \cite{ferreira2009eigenvalues}), the kernel $k$ can be expressed in terms of the eigenvectors and eigenvalues of $Q$:
	\begin{equation*}
		\forall x, y, k(x,y) = \sum_i q_i F_i(x) F_i(y).
	\end{equation*}
	It is straightforward to check that $k^{(1/2)}(x,y) = \sum_i q_i^{1/2} F_i(x) F_i(y), x, y \in \xSet$, defines a kernel for $Q^{1/2}$ and satisfies $k(x,y) = \int k^{(1/2)}(x,z) k^{(1/2)}(y,z) dz$. Thus, we have
	\begin{align}
		\int D^2\Psi(\rho)(n\rho,R_0(n)\rho) d\nu(n)
			&= \frac 1 2 \int\int\int k_\rho(x,y) k^{(1/2)}(x,z) \rho(x) k^{(1/2)}(y,z) \rho(y) dxdydz \nonumber\\
			&= \frac 1 2 \Tr\lrbracket{(\rho Q^{1/2}) D^2 \Psi(\rho) (\rho Q^{1/2})^*}. \label{temp:eq:prop:limit_generator}
	\end{align}
	By density of the functions whose second derivative admits a kernel $k_\rho$ in $C^2$, this formula holds for all test functions $\phi \in \Theta$. Using similar reasoning for the three remaining terms, we get
	\begin{multline*}
		\mathcal L \phi(\rho,\zeta) = D\Psi(\rho)(\div(K \nabla \rho) + \frac 1 2 F \rho) \chi(\zeta)\\
		+ \frac 1 2 \Tr\lrbracket{(\rho Q^{1/2}, Q^{1/2})
		\begin{pmatrix}
			D^2 \Psi(\rho) \chi(\zeta) & D\Psi(\rho) \otimes D\chi(\zeta)\\
			D\Psi(\rho) \otimes D\chi(\zeta) & \Psi(\rho) D^2 \chi(\zeta)
		\end{pmatrix}
		(\rho Q^{1/2}, Q^{1/2})^*},
	\end{multline*}
	which is the generator of \cref{eq:limit_equation}. Once moment estimates for $\rho$ have been obtained in $L^2_x$, integrability of $M_\phi$ is ensured. In addition, estimates on $\phi(\rho(t),\zeta(t))$ and $\mathcal L \phi(\rho(t),\zeta(t))$ (uniformly in $t \in [0,T]$) are also obtained, since $\phi$ and $\mathcal L \phi$ have at most quadratic growth. Then, the proof that $M_\phi$ is a martingale follows the same strategy as for the proof of \cref{prop:good_test_function_martingale}. This proof is omitted. It thus remains to prove the moment estimates for $\rho$.

	We apply Itô's formula, equation \cref{eq:limit_equation_weak} and we take the expectation (so that the martingale part vanishes), to get
	% \begin{multline*}
	% 	\frac 1 2 \scalrho{\rho(t)}{\xi}^2 = \frac 1 2 \scalrho{\rho_0}{\xi}^2 + \int_0^t \scalrho{\rho(s)}{\div(K \nabla \xi)} \scalrho{\rho(s)}{\xi} ds + \int_0^t \scalrho{\frac 1 2 F \rho(s)}{\xi} \scalrho{\rho(s)}{\xi} ds\\
	% 	+ \sum_i \sqrt{q_i} \int_0^t \scalrho{F_i \rho(s)}{\xi} \scalrho{\rho(s)}{\xi} dB_i(s) + \frac 1 2 \sum_i q_i \int_0^t \scalrho{F_i \rho(s)}{\xi}^2 ds
	% \end{multline*}
	\begin{multline*}
		\frac 1 2 \esp{\scalrho{\rho(t)}{\xi}^2} = \frac 1 2 \esp{\scalrho{\rho_0}{\xi}^2} + \mathbb E{\int_0^t \scalrho{\rho(s)}{\div(K \nabla \xi)} \scalrho{\rho(s)}{\xi} ds}\\
		 + \mathbb E \int_0^t \scalrho{\frac 1 2 F \rho(s)}{\xi} \scalrho{\rho(s)}{\xi} ds + \frac 1 2 \sum_i q_i \mathbb E \int_0^t \scalrho{F_i \rho(s)}{\xi}^2 ds.
	\end{multline*}
	Then, we evaluate at $\xi = e_\ell$ with $\ell \in \mathbb Z^d$ and $e_\ell$ the Fourier basis $e_\ell(x) = \exp(2 i \pi \scald{\ell}{x})$. Let $\lambda_\ell = 4 \pi^2 \scald{\ell}{K\ell}$ so that $\div(K \nabla e_\ell) = - \lambda_\ell e_\ell$. We sum this formula for $\abs{\ell} \leq L$. Let $P_L$ be the orthogonal projector on the space generated by $\set{e_\ell \mid \abs{\ell} \leq L}$. Since $\lambda_\ell \geq 0$, we get
	\begin{align*}
		\begin{split}
		\frac 1 2 \esp{\norm{P_L \rho(t)}_{L^2_x}^2}
			&\leq \frac 1 2 \esp{\norm{P_L \rho_0}_{L^2_x}^2} + \mathbb E \int_0^t \frac 1 2 \norm{P_L (F\rho(s))}_{L^2_x} \norm{P_L \rho(s)}_{L^2_x} ds\\
			&\quad + \frac 1 2 \sum_i \mathbb E \int_0^t q_i \norm{P_L (F_i \rho(s))}_{L^2_x}^2 ds
		\end{split}\\
			&\leq \frac 1 2 \esp{\norm{P_L \rho_0}_{L^2_x}^2} + \frac 1 2 \parths{\norm{F}_{L^\infty} + \sum_i q_i \norm{F_i}_{L^\infty}^2} \mathbb E \int_0^t \norm{\rho(s)}_{L^2_x}^2 ds.
	\end{align*}
	Taking $L \to \infty$, using \cref{lem:qiFi_summable} and Gronwall's Lemma, we get
	\begin{equation} \label{eq:moment_limit_solution}
		\esp{\norm{\rho(t)}_{L^2_x}^2} \lesssim \esp{\norm{\rho_0}_{L^2_x}^2}.
	\end{equation}
	This concludes the proof of the moment estimates for $\rho$, hence the proof that $M_\phi$ is a martingale.

	Conversely, assume that for all $\phi \in \Theta$, $M_\phi$ and $M_{\phi^2}$ are martingales. It holds in particular for regular and bounded test functions $\phi$. It is then standard that a solution to this martingale problem is the Markov process of generator $\mathcal L$ (see for example chapter 4 of \cite{ethier1986markov}), based on Lévy's martingale representation theorem in Hilbert spaces (see \cite{da2014stochastic}, Theorem 8.2). This concludes the proof since we already proved that $\mathcal L$ is the generator associated to \cref{eq:limit_equation}.
\end{proof}

\section{Tightness of the coupled stopped process} \label{sec:tightness}

In this section, we prove the following Proposition.
\begin{proposition} \label{prop:tightness}
	Let $\Lambda \in (0,\infty)$. The family of processes $\seq{(\rho\stoppedLz,\zeta\stoppedLz)}{\epsilon}$ is tight in the space $C([0,T],H^{-\sigma}_x) \times C([0,T],C^1_x)$ for any $\sigma > 0$. Moreover, the family $\seq{\zeta^\epsilon}{\epsilon}$ is tight in $C([0,T],C^1_x)$.
\end{proposition}

To simplify the notation, we write $C_T H^{-\sigma}_x \times C_T C^1_x$ for $C([0,T],H^{-\sigma}_x) \times C([0,T],C^1_x)$.

Owing to Slutsky's Lemma (see \cite{billingsley1999convergence}, Theorem 4.1) and to \cref{lem:taue_to_infty}, \cref{prop:tightness} is equivalent to the tightness of $\seq{(\rho\stoppedL,\zeta\stoppedL)}{\epsilon}$ and $\seq{\zeta\stopped}{\epsilon}$.

Since these processes are pathwise continuous, we have the following inequality between the modulus of continuity $w$ for continuous functions and the modulus of continuity $w'$ for càdlàg functions (see \cite{billingsley1999convergence}, equation 14.11):
\begin{equation*}
	w_X(\delta) \leq 2 w'_X(\delta),
\end{equation*}
with, for a càdlàg function $X$,
\begin{gather*}
	w_X(\delta) \doteq \sup_{0 \leq t \leq s \leq t + \delta \leq T} \norm{X(s)-X(t)}\\
	w'_X(\delta) \doteq \sup_{\seq{t_i}{i}} \max_i \sup_{t_i \leq t \leq s < t_{i+1}} \norm{X(s)-X(t)},
\end{gather*}
where $\seq{t_i}{i}$ is a subdivision of $[0,T]$. Therefore, the tightness in the Skorokhod space $D_T H^{-\sigma}_x \times D_T C^1_x$ (respectively $D_T C^1_x$) implies the tightness in $C_T H^{-\sigma}_x \times C_T C^1_x$ (respectively in $C_T C^1_x$).

Owing to Theorem 3.1. of \cite{jakubowski1986on}, tightness in the Skorokhod space follows from the following claims, which are proved in \cref{subsec:zeta_moments,subsec:Aldous} respectively.
\begin{enumerate}[label=(\roman*)]
	\item \label{item:first_claim} For all $\eta > 0$, there exists some compact sets $K_\eta \subset H^{-\sigma}_x$ and $K'_\eta \subset C^1_x$ such that for all $\epsilon > 0$,
	\begin{gather}
		 \label{eq:aim_rho_tight}
		\proba{\forall t \in [0,T], \rho\stoppedL(t) \in K_\eta} > 1 - \eta\\
		 \label{eq:fake_aim_zeta_tight}
		\proba{\forall t \in [0,T], \zeta\stoppedL(t) \in K'_\eta} > 1 - \eta\\
		 \label{eq:aim_zeta_tight}
		\proba{\forall t \in [0,T], \zeta\stopped(t) \in K'_\eta} > 1 - \eta.
	\end{gather}
	\item \label{item:second_claim} If $\phi$ is a sum of a finite number of bounded functions $\phi_i \in \Theta$, then $\seq{\phi(\rho\stoppedL,\zeta\stoppedL)}{\epsilon}$ is tight in $D([0,T],\mathbb R)$.

	For any $\tilde \phi \in \Theta$ with $\psi = 1$, $\seq{\tilde \phi(\zeta\stopped)}{\epsilon}$ is tight in $D([0,T],\mathbb R)$.
\end{enumerate}

We ask of $\phi$ to be a finite sum of test functions because Theorem 3.1. of \cite{jakubowski1986on} requires the class of test functions to separate points and to be closed under addition, but $\Theta$ does not satisfies the latter condition.

\subsection{Proof of the first claim \cref{item:first_claim}} \label{subsec:zeta_moments}

Using \cref{prop:L2_bound} and the Markov inequality, we have for $K>0$,
\begin{equation*}
	\proba{\exists t \in [0,T], \norm{\rho\stoppedL(t)}_{L^2_x} > K} \leq \frac{\esp{\sup_{t \in [0,T]} \norm{\rho\stoppedL(t)}_{L^2_x}}}{K} \lesssim_\Lambda \frac{\esp{\norm{f^\epsilon_0}_\fSet}}{K}.
\end{equation*}
Note that stopping the processes at $\tauLze$ is necessary at this point. Owing to the compact embedding $L^2_x \subset H^{-\sigma}_x$ for $\sigma > 0$, we get \cref{eq:aim_rho_tight}.

Since \cref{eq:fake_aim_zeta_tight} is a consequence of \cref{eq:aim_zeta_tight}, it remains to prove \cref{eq:aim_zeta_tight}. Owing to Ascoli's Theorem, we have a compact embedding of the Hölder space $C^{1,\delta}_x \subset C^1_x$ for any $\delta > 0$. Moreover, with $s = \floor{d/2}+2$, we have a continuous embedding $H^s_x \subset C^{1,\delta}_x$ for any $\delta \in (0,s - \frac d 2 - 1]$. Then \cref{eq:aim_zeta_tight} is a consequence of \cref{prop:zeta_moments} below and of the Markov inequality.
\begin{proposition} \label{prop:zeta_moments}
	Recall that $\taue$ is defined by \cref{eq:def_taue}. Then, for all $T>0$, we have
	\begin{equation*}
		\sup_\epsilon \esp{\sup_{t \in [0,T]} \norm{\zeta\stopped(t)}_{H^{\floor{d/2}+2}_x}^2} < \infty.
	\end{equation*}
\end{proposition}
\begin{proof}
	The idea of this proof is to express $\zeta^\epsilon$ (and its derivatives) as a sum of a small term and a martingale, and then to estimate the martingale using Doob's Maximal Inequality. This argument is used two times in a row, and the estimates heavily rely on \cref{hyp:B_continuity,hyp:B_subpolynomial}.

	Since $\zeta^\epsilon(t) \in E = C^{2\floor{d/2}+4}_x \subset H^{\floor{d/2}+2}_x$, it is sufficient to prove that for all multi-indices $\beta$ of length $\abs{\beta} \leq \floor{d/2}+2$, we have
	\begin{equation*}
		\sup_\epsilon \esp{\sup_{t \in [0,T]} \norm{\frac{\partial^{\abs{\beta}} \zeta\stopped(t)}{\partial x^\beta}}_{L^2_x}^2} < \infty.
	\end{equation*}
	Fix such a $\beta$ and let $\epsilon > 0$. First note that
	\begin{equation} \label{eq1:prop_zeta_moment}
		\esp{\sup_{t \in [0,T]} \norm{\frac{\partial^{\abs{\beta}} \zeta\stopped(t)}{\partial x^\beta}}_{L^2_x}^2} \leq \int \esp{\sup_{t \in [0,T]} \abs{\frac{\partial^{\abs{\beta}} \zeta\stopped(t,x)}{\partial x^\beta}}^2}dx.
	\end{equation}
	For $x \in \xSet$, define $\theta_{x,\beta} \in E^*$ by
	\begin{equation*}
		\forall n \in E, \theta_{x,\beta}(n) = \frac{\partial^{\abs{\beta}} n}{\partial x^\beta}(x).
	\end{equation*}
	Since $\bar m^\epsilon$ is almost surely an $E$-valued càdlàg function, the derivative and the integral commute in the following computation:
	\begin{equation*}
		\frac{\partial^{\abs{\beta}} \zeta^\epsilon(t,x)}{\partial x^\beta} = \frac 1 \epsilon \int_0^t \frac{\partial^{\abs{\beta}} \bar m^\epsilon(s,x)}{\partial x^\beta} ds = \frac{1}{\epsilon} \int_0^t \theta_{x,\beta}(\bar m^\epsilon(s)) ds.
	\end{equation*}
	Owing to the identity $\lrangle{\theta_{x,\beta}} = 0$, \cref{lem:Pt_lip_Pt_quad,hyp:B_continuity}, the function $\psi_x \doteq - R_0\theta_{x,\beta}$ is well-defined, is Lipschitz continuous with $\lrbracket{\psi_x}_{\Lip} \lesssim \lrbracket{\theta_{x,\beta}}_{\Lip} = 1$ and $\psi_x, \psi_x^2 \in D(B)$. Therefore \cref{prop:good_test_function_martingale_cases} states that
	\begin{align*}
		M^\epsilon_{\epsilon \psi_x}(t)
			&= \epsilon \psi_x(\bar m^\epsilon(t)) - \epsilon \psi_x(\bar m(0)) - \frac{1}{\epsilon^2} \int_0^t \epsilon B \psi_x(\bar m^\epsilon(s))ds\\
			&= \epsilon \psi_x(\bar m^\epsilon(t)) - \epsilon \psi_x(\bar m(0)) - \frac{\partial^{\abs{\beta}} \zeta^\epsilon(t,x)}{\partial x^\beta}
	\end{align*}
	defines a square-integrable martingale such that
	\begin{equation} \label{eq:prop_zeta_variance_martingale_variance}
		\esp{\abs{M\stopped_{\epsilon \psi_x}(t)}^2} = \esp{\int_0^{t \wedge \taue} \parths{B(\psi_x^2) - 2 \psi_x B\psi_x}(\bar m^\epsilon(s))ds}.
	\end{equation}
	Since $\lrbracket{\psi_x}_{\Lip} \lesssim \lrbracket{\theta_{x,\beta}}_{\Lip} = 1$ and $\alpha < 1$ in \cref{eq:hyp_alpha}, we have
	\begin{equation*}
		\esp{\sup_{t \in [0,T]} \abs{\epsilon \psi_x(\bar m\stopped(t))}^2} \lesssim \epsilon^2 \esp{\epsilon^{-2\alpha} \vee \norm{\bar m(0)}_E^2} \lesssim 1,
	\end{equation*}
	and by Doob's Maximal Inequality, we get
	\begin{align}
		\esp{\sup_{t \in [0,T]} \abs{\frac{\partial^{\abs{\beta}} \zeta\stopped(t,x)}{\partial x^\beta}}^2}^{\frac 1 2}
			&\lesssim 1 + \esp{\sup_{t \in [0,T]} \abs{M\stopped_{\epsilon \psi_x}(t)}^2}^{\frac 1 2} \nonumber\\
			&\lesssim 1 + \esp{\abs{M\stopped_{\epsilon \psi_x}(T)}^2}^{\frac 1 2}. \label{eq0:prop_zeta_moment}
	\end{align}
	Owing to \cref{prop:good_test_function_martingale_cases}, we have
	\begin{equation*}
		\esp{\abs{M\stopped_{\epsilon \psi_x}(T)}^2} = \esp{\int_0^{T \wedge \taue} \parths{B(\psi_x^2) - 2 \psi_x B \psi_x}(\bar m^\epsilon(s))ds},
	\end{equation*}
	For now, we only know that the right-hand side is of order $\epsilon^{-2\alpha}$, by \cref{eq:def_taue,eq:prop_zeta_variance_martingale_variance}. To retrieve an estimate uniform in $\epsilon$, we use the same martingale argument as before. Let
	\begin{equation*}
	\tilde \theta_{x,\beta} \doteq B(\psi_x^2) - 2 \psi_x B \psi_x = B((R_0 \theta_{x,\beta})^2) + 2 \theta_{x,\beta} R_0 \theta_{x,\beta},
	\end{equation*}
	so that
	\begin{equation} \label{eq2:prop_zeta_moment}
		\esp{\abs{M\stopped_{\epsilon \psi_x}(T)}^2} = \esp{\int_0^{T \wedge \taue} \tilde \theta_{x,\beta} (\bar m^\epsilon(s))ds}.
	\end{equation}
	Since $\theta_{x,\beta}$ and $R_0 \theta_{x,\beta}$ are pseudo-linear functions, the function $\theta_{x,\beta} R_0 \theta_{x,\beta}$ is pseudo-quadratic. Thus, by \cref{lem:Pt_lip_Pt_quad,hyp:B_continuity}, the function
	\begin{equation*}
		\tilde \psi_x = (R_0 \theta_{x,\beta})^2 - 2 R_0 \lrbracket{\theta_{x,\beta} R_0 \theta_{x,\beta} - \lrangle{\theta_{x,\beta} R_0 \theta_{x,\beta}}},
	\end{equation*}
	is well-defined and satisfies $\tilde \psi_x, \tilde \psi_x^2 \in D(B)$ and $B \tilde \psi_x = \tilde \theta_{x,\beta} - 2 \lrangle{\theta_{x,\beta} R_0 \theta_{x,\beta}}$. As before, introduce the martingale process
	\begin{align*}
		M^\epsilon_{\epsilon^2 \tilde \psi_x}(t)
			&= \epsilon^2 \tilde \psi_x(\bar m^\epsilon(t)) - \epsilon^2 \tilde \psi_x(\bar m(0)) - \frac{1}{\epsilon^2} \int_0^t \epsilon^2 B \tilde \psi_x(\bar m^\epsilon(s))ds\\
			&= \epsilon^2 \tilde \psi_x(\bar m^\epsilon(t)) - \epsilon^2 \tilde \psi_x(\bar m(0)) - \int_0^t \tilde \theta_{x,\beta} (\bar m^\epsilon(s))ds + 2 t \lrangle{\theta_{x,\beta} R_0 \theta_{x,\beta}}.
	\end{align*}
	Owing to \cref{lem:Pt_lip_Pt_quad}, we have
	\begin{gather*}
		\forall n \in E, \abs{\tilde \psi_x(n)} \lesssim (1 + \norm{n}_E^{b+1} + \norm{n}_E^2)\\
		\lrangle{\theta_{x,\beta} R_0 \theta_{x,\beta}} \lesssim 1.
	\end{gather*}
	Using the conditions $\alpha (b+1) < 2$ and $\alpha <  1$ in \cref{eq:hyp_alpha}, and using the finiteness of moments of order $2(b+1)$ and $4$ of $\bar m(0)$ in \cref{hyp:m_L2_all_starting_point}, we get
	\begin{equation} \label{eq3:prop_zeta_moment}
		\esp{\int_0^{T \wedge \taue} \tilde \theta_{x,\beta} (\bar m^\epsilon(s))ds} \lesssim_T 1 + \esp{\abs{M\stopped_{\epsilon^2 \tilde \psi_x}(T)}},
	\end{equation}
	where, owing to \cref{prop:good_test_function_martingale_cases},
	\begin{equation*}
		\esp{\abs{M\stopped_{\epsilon^2 \tilde \psi_x}(T)}^2} = \epsilon^2 \esp{\int_0^{T \wedge \taue} \parths{B(\tilde \psi_x^2) - 2 \tilde \psi_x B \tilde \psi_x}(\bar m^\epsilon(s))ds}.
	\end{equation*}
	Owing to \cref{hyp:B_subpolynomial}, we have
	\begin{equation*}
	\forall n \in E, \abs{\parths{B(\tilde \psi_x^2) - 2 \tilde \psi_x B \tilde \psi_x}(n)} \lesssim (1 + \norm{n}_E^{2(b+1)} + \norm{n}_E^4).
	\end{equation*}
	Since $\alpha (b+1) < 1$ and $2 \alpha < 1$ in \cref{eq:hyp_alpha} and since $\bar m(0)$ has finite moments of order $2(b+1)$ and $4$ in \cref{hyp:m_L2_all_starting_point}, we get
	\begin{equation} \label{eq4:prop_zeta_moment}
		\esp{\abs{M\stopped_{\epsilon^2 \tilde \psi_x}(T)}^2} \lesssim_T 1.
	\end{equation}
	Gathering the estimates \cref{eq0:prop_zeta_moment}, \cref{eq2:prop_zeta_moment}, \cref{eq3:prop_zeta_moment} and \cref{eq4:prop_zeta_moment}, we obtain the required result
	\begin{equation*}
		\sup_\epsilon \sup_{x \in \xSet} \esp{\sup_{t \in [0,T]} \abs{\frac{\partial^{\abs{\beta}} \zeta\stopped(t,x)}{\partial x^\beta}}^2} \lesssim_T 1.
	\end{equation*}
	This concludes the proof by \cref{eq1:prop_zeta_moment}.
\end{proof}

\Cref{prop:zeta_moments}, together with the compact embedding $H^s_x \subset C^1_x$ and the Markov inequality, proves that \cref{eq:aim_zeta_tight} holds, hence \cref{eq:fake_aim_zeta_tight}. This concludes the proof of \cref{item:first_claim}.

\subsection{Proof of the second claim \cref{item:second_claim}} \label{subsec:Aldous}

As in \cite{debussche2017diffusion}, we prove \cref{item:second_claim} using the Aldous criterion (\cite{jacod2003limit}, Theorem 4.5 p356).

Let $\phi = \sum_i \phi_i$ be the sum of a finite number of bounded functions $\phi_i \in \Theta$. We set $X^\epsilon = (f^\epsilon,\zeta^\epsilon,\bar m^\epsilon)$ and $\bar X^\epsilon = (\rho^\epsilon,\zeta^\epsilon)$. Recall that if $\tilde \phi \in \Theta$ depends only on $z$, then the perturbed test function $\tilde \phi^\epsilon$ defined by \cref{prop:corrector} depends only on $n$ and $z$. Using \cref{prop:good_test_function_martingale_cases}, this allows us to stop the processes only at $\taue$ instead of $\tauLe$ while keeping the same estimates. Therefore, the proof of the tightness of $\seq{\tilde \phi(\zeta\stopped)}{\epsilon}$ is the same as of $\seq{\phi(\bar X\stoppedL)}{\epsilon}$, and is thus omitted. It only remains to prove $\seq{\phi(\bar X\stoppedL)}{\epsilon}$ is tight.

The Aldous criterion gives a sufficient condition for the tightness of the family $\seq{\phi(\bar X\stoppedL)}{\epsilon}$ in $D([0,T],\mathbb R)$: since $\phi$ is bounded, is it sufficient to prove that
\begin{equation} \label{eq:aldous}
	\forall \eta > 0, \lim_{\delta \to 0} \limsup_{\epsilon \to 0} \sup_{\substack{\tau_1,\tau_2 \leq T\\\tau_1 \leq \tau_2 \leq \tau_1 + \delta}} \proba{\abs{\phi(\bar X\stoppedL(\tau_2)) - \phi(\bar X\stoppedL(\tau_1))} > \eta} = 0,
\end{equation}
where $\tau_1,\tau_2$ are any $\seq{\mathcal F^\epsilon_t}{t \in \tSet}$-stopping times.

Define the perturbed test function $\phi^\epsilon = \sum_i \phi_i^\epsilon$. This sum satisfies the estimates \cref{eq:control_phi_1,eq:control_phi_2,eq:control_perturbed_generator}. Then, define
\begin{align}
	\theta^\epsilon(t)
		&= \phi(\bar X^\epsilon(0)) + \phi^\epsilon(X^\epsilon(t)) - \phi^\epsilon(X^\epsilon(0)) \label{eq:def_thetae}\\
		&= \phi(\bar X^\epsilon(0)) + \int_0^t \mathcal L^\epsilon \phi^\epsilon (X^\epsilon(s))ds + M^\epsilon_{\phi^\epsilon}(t), \label{eq:expr_thetae}
\end{align}
where $M^\epsilon_{\phi^\epsilon}$ is defined by \cref{prop:good_test_function_martingale}, so that
\begin{equation*}
	\begin{split}
	\phi(\bar X\stoppedL(\tau_2)) - \phi(\bar X\stoppedL(\tau_1))
		&= \parths{\theta\stoppedL(\tau_2) - \theta\stoppedL(\tau_1)} - \parths{\phi^\epsilon(X\stoppedL(\tau_2)) - \phi(\bar X\stoppedL(\tau_2))}\\
		&\quad + \parths{\phi^\epsilon(X\stoppedL(\tau_1)) - \phi(\bar X\stoppedL(\tau_1))}.
	\end{split}
\end{equation*}
%Using \cref{eq:control_phi_1,eq:control_phi_2} in \cref{prop:corrector}, \cref{eq:taue_bounds_m} and \cref{prop:L2_bound}, we get
Using \cref{eq:taue_bounds_m}, \cref{prop:corrector,prop:L2_bound}, we get
\begin{multline*}
	\abs{\phi^\epsilon(X\stoppedL(t)) - \phi(\bar X\stoppedL(t))}\\
	\lesssim_{\phi,\Lambda} (1 + \norm{f^\epsilon_0}_\fSet^2) (\epsilon (1 + \epsilon^{-\alpha} \vee \norm{\bar m(0)}) + \epsilon^2 (1 + \epsilon^{-\alpha(b+1)} \vee \norm{\bar m(0)}^{b+1})).
\end{multline*}
Since $\alpha < 1$ and $\alpha (b+1) < 2$ in \cref{eq:hyp_alpha}, we get
\begin{equation*}
	\esp{\sup_{t \in [0,T]} \abs{\phi^\epsilon(X\stoppedL(t)) - \phi(\bar X\stoppedL(t))}} \xrightarrow[\epsilon \to 0]{} 0,
\end{equation*}
hence, when $\epsilon \to 0$,
\begin{equation*}
	\sup_{\tau_1,\tau_2} \esp{\abs{\phi(\bar X\stoppedL(\tau_2)) - \phi(\bar X\stoppedL(\tau_1))}} \leq \sup_{\tau_1,\tau_2} \esp{\abs{\theta\stoppedL(\tau_2) - \theta\stoppedL(\tau_1)}} + o(1).
\end{equation*}
Using the Markov inequality, we get
\begin{equation*}
	\sup_{\tau_1,\tau_2} \proba{\abs{\phi(\bar X\stoppedL(\tau_2)) - \phi(\bar X\stoppedL(\tau_1))} > \eta} \leq \sup_{\tau_1,\tau_2} \frac{\esp{\abs{\theta\stoppedL(\tau_2) - \theta\stoppedL(\tau_1)}}}{\eta} + o(1).
\end{equation*}
Therefore, it is sufficient to prove that
\begin{equation} \label{eq:aim_aldous2}
	\sup_{\tau_1, \tau_2, \epsilon} \esp{\abs{\theta\stoppedL(\tau_2) - \theta\stoppedL(\tau_1)}} \xrightarrow[\delta \to 0]{} 0,
\end{equation}
to deduce \cref{eq:aldous} and then to use Aldous criterion.

Owing to \cref{eq:expr_thetae}, we have
\begin{equation} \label{eq:temp_new_aim_final_cut}
	\abs{\theta\stoppedL(\tau_2) - \theta\stoppedL(\tau_1)} \leq \int_{\tau_1 \wedge \tauLe}^{\tau_2 \wedge \tauLe} \abs{\mathcal L^\epsilon \phi^\epsilon(X^\epsilon(s))} ds + \abs{M\stoppedL_{\phi^\epsilon}(\tau_2) - M\stoppedL_{\phi^\epsilon}(\tau_1)}.
\end{equation}
% Owing to \cref{eq:control_perturbed_generator} in \cref{prop:corrector}, to \cref{eq:taue_bounds_m} and to \cref{prop:L2_bound}, we have
Using once again \cref{eq:taue_bounds_m}, \cref{prop:corrector,prop:L2_bound}, we get
\begin{align*}
	\abs{\mathcal L^\epsilon \phi^\epsilon (X\stoppedL(t))}
		&\lesssim_{\phi,\Lambda} \abs{\mathcal L \phi(\bar X\stoppedL(s))} + \epsilon (1 + \norm{f^\epsilon_0}_\fSet^3)(1 + \epsilon^{-\alpha(b+2)} \vee \norm{\bar m(0)}^{b+2})\\
		&\lesssim_{\phi,\Lambda} 1 + \norm{f^\epsilon_0}_\fSet^2 + \epsilon (1 + \norm{f^\epsilon_0}_\fSet^3)(1 + \epsilon^{-\alpha(b+2)} \vee \norm{\bar m(0)}^{b+2}).
\end{align*}
Using the Cauchy-Schwarz inequality, the condition $\alpha(b+2)<1$ in \cref{eq:hyp_alpha}, \cref{hyp:m_L2_all_starting_point,hyp:f0}, we get
\begin{equation} \label{eq:temp1}
	\esp{\sup_{t \in [0,T]} \abs{\mathcal L^\epsilon \phi^\epsilon (X\stoppedL(t))}} \lesssim_{\phi,\Lambda} 1.
\end{equation}
Thus, we get
\begin{align*}
	\sup_\epsilon \sup_{\tau_1,\tau_2} \esp{\int_{\tau_1 \wedge \tauLe}^{\tau_2 \wedge \tauLe} \abs{\mathcal L^\epsilon \phi^\epsilon(X^\epsilon(s))} ds}
		&\leq \sup_\epsilon \sup_{\tau_1,\tau_2} \delta \esp{\sup_{t \in [0,T]} \abs{\mathcal L^\epsilon \phi^\epsilon (X\stoppedL(t))}} \xrightarrow[\delta \to 0]{} 0.
\end{align*}
The last term of \cref{eq:temp_new_aim_final_cut} is controlled using martingale arguments. Owing to \cref{prop:good_test_function_martingale}, $M\stoppedL_{\phi^\epsilon}$ is indeed a square-integrable martingale and
\begin{align*}
	\esp{\abs{M\stoppedL_{\phi^\epsilon}(\tau_2) - M\stoppedL_{\phi^\epsilon}(\tau_1)}^2}
		&= \esp{\abs{M\stoppedL_{\phi^\epsilon}(\tau_2)}^2 - \abs{M\stoppedL_{\phi^\epsilon}(\tau_1)}^2}\\
		&= \frac{1}{\epsilon^2} \esp{\int_{\tau_1 \wedge \tauLe}^{\tau_2 \wedge \tauLe} \parths{B((\phi^\epsilon)^2) - 2 \phi^\epsilon B \phi^\epsilon}(X^\epsilon(s)) ds}\\
		&= \esp{\int_{\tau_1 \wedge \tauLe}^{\tau_2 \wedge \tauLe} \sum_{i = -2}^{2}\epsilon^i r_i(X^\epsilon(s)) ds}
\end{align*}
where the terms $r_i$ are obtained by writing $\phi^\epsilon = \phi + \epsilon \phi_1 + \epsilon^2 \phi_2$ and expanding $B((\phi^\epsilon)^2) - 2 \phi^\epsilon B \phi^\epsilon$. The terms containing $\phi$ vanish, using $B\phi = 0$, $B(\phi^2) = 0$ and $B \phi \phi_j = \phi B \phi_j$ (since $\phi$ does not depend on $n$). Using \cref{hyp:B_subpolynomial}, the remaining terms satisfy
\begin{gather*}
	r_{-2} = r_{-1} = 0,\\
	r_0(f,z,n) = \lrbracket{B(\phi_1^2) - 2 \phi_1 B \phi_1}(f,z,n) \lesssim_{\phi} (1 + \norm{f}_\fSet^2)(1 + \norm{n}_E^2),\\
	r_1(f,z,n) = \lrbracket{2B(\phi_1 \phi_2) - \phi_1 B \phi_2 - \phi_2 B \phi_1}(f,z,n) \lesssim_{\phi} (1 + \norm{f}_\fSet^3)(1 + \norm{n}_E^{b+2}),\\
	r_2(f,z,n) = \lrbracket{B(\phi_2^2) - 2 \phi_2 B \phi_2}(f,z,n) \lesssim_{\phi} (1 + \norm{f}_\fSet^4)(1 + \norm{n}_E^{2(b+1)}).
\end{gather*}
As for \cref{eq:temp1}, using that $\alpha(b+2)<1$ in \cref{eq:hyp_alpha}, we have for $i \in \set{1,2}$
\begin{equation*}
	\esp{\sup_{t \in [0,T]} \epsilon^i r_i(X\stoppedL(t))} \lesssim_{\phi,\Lambda} 1,
\end{equation*}
and
\begin{equation*}
	\sup_\epsilon \sup_{\tau_1,\tau_2} \esp{\int_{\tau_1 \wedge \tauLe}^{\tau_2 \wedge \tauLe} \epsilon^i r_i(X^\epsilon(s)) ds} \xrightarrow[\delta \to 0]{} 0.
\end{equation*}
We need to be more cautious when dealing with $r_0$, since there are no $\epsilon$ left to compensate the $\epsilon^{-2\alpha}$ that would appear from bounding $\bar m\stopped$ from above using \cref{prop:L2_bound}. The idea is to use estimates for $f\stoppedL$ and $\bar m^\epsilon$ (instead of $\bar m\stoppedL$), using that for $s \leq \tauLe$, $\bar m\stoppedL(s) = \bar m^\epsilon(s)$. We write
% \begin{remark}
% 	This idea can be used elsewhere in the paper, for example when dealing with $r_1$ and $r_2$. But it never leads to stronger results or weaker assumptions, so we only use it here.
% \end{remark}
\begin{align*}
	\esp{\int_{\tau_1 \wedge \tauLe}^{\tau_2 \wedge \tauLe} r_0(X^\epsilon(s))ds}
		&\lesssim_{\phi} \esp{\int_{\tau_1 \wedge \tauLe}^{\tau_2 \wedge \tauLe} (1 + \norm{f\stoppedL(s)}_\fSet^2)(1 + \norm{\bar m^\epsilon(s)}_E^2) ds}\\
		&\lesssim_{\phi,\Lambda} \esp{\int_{\tau_1}^{\tau_2} (1 + \norm{f^\epsilon_0}_\fSet^2)(1 + \norm{\bar m^\epsilon(s)}_E^2) ds}\\
		&\lesssim_{\phi,\Lambda} \int_0^T \esp{\indset{[\tau_1,\tau_2]}(s) (1 + \norm{f^\epsilon_0}_\fSet^2)(1 + \norm{\bar m^\epsilon(s)}_E^2)} ds.
\end{align*}
Then, we use the Hölder inequality to write
\begin{multline*}
	\esp{\int_{\tau_1 \wedge \tauLe}^{\tau_2 \wedge \tauLe} r_0(X^\epsilon(s))ds}\\
	\begin{split}
		&\lesssim_{\phi,\Lambda} \int_0^T \esp{\indset{[\tau_1,\tau_2]}(s)}^{\frac 1 3} \esp{1 + \norm{f^\epsilon_0}_\fSet^6}^{\frac 1 3} \esp{1 + \norm{\bar m^\epsilon(s)}_E^6}^{\frac 1 3} ds\\
		&\lesssim_{\phi,\Lambda} \int_0^T \esp{\indset{[\tau_1,\tau_2]}(s)}^{\frac 1 3} ds \esp{1 + \norm{f^\epsilon_0}_\fSet^6}^{\frac 1 3} \esp{1 + \norm{\bar m^\epsilon(0)}_E^6}^{\frac 1 3},
	\end{split}
\end{multline*}
by stationarity of $\bar m$. Using the Cauchy-Schwarz inequality, \cref{hyp:m_L2_all_starting_point,hyp:f0}, we get
\begin{align*}
	\esp{\int_{\tau_1 \wedge \tauLe}^{\tau_2 \wedge \tauLe} r_0(X^\epsilon(s))ds}
		&\lesssim_{\phi,\Lambda} \int_0^T \esp{\indset{[\tau_1,\tau_2]}(s)}^{\frac 1 2} ds\\
		&\lesssim_{\phi,\Lambda,T} \parths{\int_0^T \esp{\indset{[\tau_1,\tau_2]}(s)} ds}^{\frac 1 2}\\
		&\lesssim_{\phi,\Lambda,T} \delta^{\frac 1 2} \to 0,
\end{align*}
uniformly in $\epsilon$, $\tau_1$ and $\tau_2$. This concludes the proof of \cref{eq:aim_aldous2}.

We are now in position to apply Aldous' criterion, which proves that the family $\seq{(\rho\stoppedLz,\zeta\stoppedLz)}{\epsilon}$ is tight in $C_T H^{-\sigma}_x \times C_T C^1_x$. This concludes the proof of \cref{item:second_claim}, and of \cref{prop:tightness}.

%\section{Convergence of the processes} \label{sec:convergence}
\section{Identification of the limit points} \label{sec:convergence}

In this section, we establish the first convergence result stated in \cref{thm:main_result}.

We start by proving the convergence of the auxiliary process $\zeta^\epsilon$ in \cref{subsec:convergence_of_the_auxiliary_process}, using the convergence of a simplified martingale problem. Then, in \cref{subsec:convergence_martingale_problem}, we determine the stopped martingale problem solved by a limit point of the stopped process. In \cref{subsec:identification_limit_point}, we use this stopped martingale to identify the limit point of the stopped process. We conclude on the convergence of the unstopped process in \cref{subsec:convergence}.

% In this section, we prove that the stopped process $(\rho\stoppedL,\zeta\stoppedL)$ converges in distribution in $C([0,T],H^{-\sigma}_x) \times C([0,T],C^1_x)$ when $\epsilon \to 0$ to the process introduced in \cref{thm:main_result}. By \cref{prop:tightness}, this stopped process is tight. Let $(\rho_\Lambda,\zeta_\Lambda)$ be a limit point. To get the aimed convergence, it is sufficient to prove the uniqueness of its distribution. In \cref{subsec:convergence_martingale_problem}, we identify the martingale problem solved by $(\rho_\Lambda,\zeta_\Lambda)$. In \cref{subsec:identification_limit_point}, we deduce the uniqueness in distribution of $(\rho_\Lambda,\zeta_\Lambda)$ using this martingale problem.

% By \cref{prop:tightness}, $\seq{\zeta^\epsilon}{\epsilon}$ is also tight. Thus, in this section, up to another extraction, we take $\epsilon$ along a sequence such that we have the following convergence in distribution
% \begin{equation*}
% 	(\rho\stoppedL,\zeta\stoppedL) \xrightarrow[\epsilon \to 0]{} (\rho_\Lambda,\zeta_\Lambda), \zeta^\epsilon \xrightarrow[\epsilon \to 0]{} \zeta.
% \end{equation*}

\subsection{Convergence of the auxiliary process} \label{subsec:convergence_of_the_auxiliary_process}

Proving the convergence of $\zeta^\epsilon$ is much simpler than for the coupled process $\bar X^\epsilon$. Indeed, as seen in particular in \cref{prop:tightness}, the only stopping time we need is $\taue$, and $\taue \xrightarrow[\epsilon \to 0]{} +\infty$. Therefore, the convergence of martingale problems is a little more intricate than the proof used in \cite{debussche2011diffusion}, but it remains straightforward.

\begin{proposition} \label{prop:convergence_of_the_auxiliary_process}
	The process $\zeta^\epsilon$ converges in distribution in $C_T C^1_x$ to a Wiener process of covariance $Q$ when $\epsilon \to 0$.
\end{proposition}

\begin{proof}
	Owing to the tightness established in \cref{prop:tightness}, there exists a sequence $\epsilon_i \xrightarrow[i \to \infty]{} 0$ and $\zeta \in C_T C^1_x$ such that $\zeta^{\epsilon_i}$ converges in distribution to $\zeta$ when $i \to \infty$. We start by proving that $\zeta$ solves the martingale problem associated with the generator $\mathcal L$. 

	Let $\phi \in \Theta$ with $\psi = 1$. Let $0 \leq s \leq s_1 \leq ... \leq s_n \leq t$, let $g$ be a continuous bounded function and for $z \in C_T C^1_x$, let $G(z) = g(z(s_1), ..., z(s_n))$ and
	$$\Phi(z) = \parths{\phi(z(t)) - \phi(z(s)) - \int_s^t \mathcal L \phi(z(u)) du}G(z).$$
	Note that $G$ and $\Phi$ are continuous and bounded on $C_T C^1_x$, so $\esp{\Phi(\zeta^{\epsilon_i})} \xrightarrow[i \to \infty]{} \esp{\Phi(\zeta)}$. Let us establish that $\esp{\Phi(\zeta^{\epsilon_i})}$ also converges to $0$.

	Let $\phi^{\epsilon_i}$ be the perturbed test function introduced in \cref{prop:corrector} associated to $\phi$. Since $\phi^{\epsilon_i}$ is a good test function, and since $G(\zeta\stoppedi)$ is $\mathcal F^{\epsilon_i}_s$-measurable, \cref{prop:good_test_function_martingale} yields
	\begin{equation*}
			\esp{\parths{\phi^{\epsilon_i}(\zeta\stoppedi(t)) - \phi^{\epsilon_i}(\zeta\stoppedi(s)) - \int_{s \wedge \tauei}^{t \wedge \tauei} \mathcal L^{\epsilon_i} \phi^{\epsilon_i}(\zeta^{\epsilon_i}(u))du} G(\zeta\stoppedi)} = 0.
	\end{equation*}
	Owing to \cref{eq:def_phi_epsilon}, this leads to
	\begin{equation*}
		\abs{\esp{\Phi(\zeta\stoppedi)}} \lesssim_g \sum_{j=1}^4 \esp{\abs{r_j}},
	\end{equation*}
	with
	\begin{gather*}
		r_1 = {\epsilon_i} (\phi_1(\zeta\stoppedi(t),\bar m\stoppedi(t)) - \phi_1(\zeta\stoppedi(s),\bar m\stoppedi(s))),\\
		r_2 = {\epsilon_i}^2 (\phi_2(\zeta\stoppedi(t),\bar m\stoppedi(t)) - \phi_2(\zeta\stoppedi(s),\bar m\stoppedi(s))),\\
		r_3 = - \int_{s \wedge \tauei}^{t \wedge \tauei} \parths{\mathcal L^{\epsilon_i} \phi^{\epsilon_i}(\zeta^{\epsilon_i}(u)) - \mathcal L \phi (\zeta^{\epsilon_i}(u))} du,\\
		r_4 = \int_{t \wedge \tauei}^{t} \mathcal L \phi(\zeta\stoppedi(u)) du - \int_{s \wedge \tauei}^{s} \mathcal L \phi(\zeta\stoppedi(u)) du.
	\end{gather*}
	Using \cref{eq:hyp_alpha}, \cref{eq:control_phi_1}, \cref{eq:control_phi_2}, \cref{eq:control_perturbed_generator}, \cref{hyp:m_L2_all_starting_point,hyp:f0}, we have for $j \in \set{1,2,3}$, $\esp{\abs{r_j}} \xrightarrow[\epsilon \to 0]{} 0$. It remains to prove that $\esp{\abs{r_4}} \to 0$. The term $r_4$ does not appear in \cite{debussche2011diffusion}, but is simple to manage since $\taue \xrightarrow[\epsilon \to 0]{} \infty$. The Cauchy-Schwarz inequality and \cref{lem:taue_to_infty} lead to
	\begin{align*}
		\esp{\abs{r_4}}^2
			&\lesssim_{\phi} \esp{\abs{t - t \wedge \tauei}^2 + \abs{s - s \wedge \tauei}^2}\\
			&\lesssim_{\phi} T^2 \proba{\tauei < T} \xrightarrow[i \to \infty]{} 0
	\end{align*}
	Thus, we get $\esp{\Phi(\zeta\stoppedi)} \xrightarrow[i \to \infty]{} 0$, hence $\esp{\Phi(\zeta)} = 0$. The same proof can be adapted when replacing $\phi$ by $\phi^2$. Therefore, the processes $M_\phi$ and $M_{\phi^2}$ defined in \cref{prop:limit_generator} are martingales. Owing to \cref{prop:limit_generator}, $\zeta$ satisfies \cref{eq:limit_equation_weak_zeta} and is a $Q$-Wiener process.

	This limit point being unique in distribution, $\zeta^\epsilon$ converges in distribution to this Wiener process.
\end{proof}

\subsection{Convergence of the stopped martingale problems} \label{subsec:convergence_martingale_problem}

In this section, we use \cref{prop:convergence_of_the_auxiliary_process} to establish the convergence of the stopped martingale problems satisfied by $X\stoppedL$. The proof is similar to the proof of \cref{prop:convergence_of_the_auxiliary_process}, but this time the stopping time persists when $\epsilon \to 0$ because of the fixed threshold $\Lambda$.

Let us introduce the path space $\ul \Omega = C_T H^{-\sigma}_x \times C_T C^1_x \times C_T C^1_x$, equipped with its Borel $\sigma$-algebra. We denote by $(\ul \rho,\ul \zeta,\ul \zeta')$ the canonical process on $\ul \Omega$ and by $\seq{\ul{\mathcal F}_t}{t \in \tSet}$ its associated filtration.

Define $\law{\epsilon,\Lambda}$ the distribution of $(\rho\stoppedL,\zeta\stoppedL,\zeta^\epsilon)$ and $\mathbb E_{\epsilon,\Lambda}$ the expectation under this distribution (on $\ul \Omega$). By \cref{prop:tightness}, the family $\seq{\law{\epsilon,\Lambda}}{\epsilon}$ is tight. Thus, in this section, we consider a sequence $\seq{\epsilon_i}{i \in \mathbb N}$ such that $\epsilon_i \to 0$ and $\law{\epsilon_i,\Lambda} \to \law{0,\Lambda}$ weakly when $i \to \infty$, for some limit point $\law{0,\Lambda}$. Note that under $\law{0,\Lambda}$, owing to \cref{prop:convergence_of_the_auxiliary_process}, $\ul \zeta'$ is a $Q$-Wiener process whose distribution $\law{Q}$ does not depend on $\Lambda$.

We now state two continuity lemmas.
\begin{lemma} \label{lem:tauL_continuity}
	For any fixed $\Lambda \in \tSet$, the mapping $\tau_\Lambda(\cdot)$ defined by \cref{eq:def_tauLze} is lower semi-continuous on $C_T C^1_x$. Moreover, it is continuous at every $z$ such that $\tau_\cdot(z)$ is continuous at $\Lambda$.
\end{lemma}
\begin{lemma} \label{lem:tauLze_as_continuity}
	The set $\set{\Lambda \geq 0 \mid \probal{Q}{\tau_\cdot(\ul \zeta') \mbox{ is not continuous at $\Lambda$}} >0}$ is at most countable. Let $\mathfrak L$ be its complementary.
\end{lemma}
We refer to \cite{hofmanova2012weak} (Lemma 3.5, 3.6 and Appendix) for the proofs of \Cref{lem:tauL_continuity,lem:tauLze_as_continuity}. These results can be applied here since $\norm{\ul \zeta'}_{C^1_x}$ is a continuous finite dimensional process and its distribution $\law{Q}$ under $\law{0,\Lambda}$ does not depend on $\Lambda$.

Owing to \cref{lem:tauLze_as_continuity}, there exist arbitrarily large numbers $\Lambda \in \mathfrak L$ and for all $\Lambda \in \mathfrak L$, $\tau_\cdot(\ul \zeta')$ is $\law{0,\Lambda}$-almost surely continuous at $\Lambda$ and by \cref{lem:tauL_continuity}, $\tau_\Lambda(\cdot)$ is $\law{0,\Lambda}$-a.s. continuous at $\ul \zeta'$. From now on, it is assumed that $\Lambda \in \mathfrak L$.

\begin{proposition} \label{prop:limit_martingale}
	Let $\Lambda \in \mathfrak L$. For all $\phi \in \Theta$, the process
	\begin{equation*}
		t \mapsto \phi(\ul \rho(t),\ul \zeta(t)) - \phi(\ul \rho(0),\ul \zeta(0)) - \int_0^{t \wedge \tau_\Lambda(\ul \zeta')} \mathcal L \phi(\ul \rho(u),\ul \zeta(u)) du
	\end{equation*}
	is a $\seq{\ul{\mathcal F}_t}{t \in \tSet}$-martingale under $\law{0,\Lambda}$.
\end{proposition}

\begin{proof}
	Let $\phi \in \Theta$. As for \cref{prop:convergence_of_the_auxiliary_process}, let $0 \leq s \leq s_1 \leq ... \leq s_n \leq t$, let $g$ be a continuous bounded function, and let
	\begin{equation*}
		G(\ul \rho,\ul \zeta,\ul \zeta') = g(\ul \rho(s_1), \ul \zeta(s_1), \ul \zeta'(s_1), ..., \ul \rho(s_n), \ul \zeta(s_n), \ul \zeta'(s_n)),
	\end{equation*}
	and
	\begin{equation*}
		\Phi(\ul \rho,\ul \zeta,\ul \zeta') = \parths{\phi(\ul \rho(t),\ul \zeta(t)) - \phi(\ul \rho(s),\ul \zeta(s)) - \int_{s \wedge \tau_\Lambda(\ul \zeta')}^{t \wedge \tau_\Lambda(\ul \zeta')} \mathcal L \phi(\ul \rho(u),\ul \zeta(u)) du}G(\ul \rho,\ul \zeta,\ul \zeta').
	\end{equation*}
	As for \cref{prop:convergence_of_the_auxiliary_process}, we establish that $\espl{\epsilon_i,\Lambda}{\Phi(\ul\rho,\ul\zeta,\ul\zeta')}$ converges, when $i \to \infty$, to both $\espl{0,\Lambda}{\Phi(\ul\rho,\ul\zeta,\ul\zeta')}$ and $0$.

	On the one hand, since $\Phi$ is continuous $\law{0,\Lambda}$-almost everywhere, $\law{\epsilon_i,\Lambda} \circ \Phi^{-1} \to \law{0,\Lambda} \circ \Phi^{-1}$ weakly when $i \to \infty$ (see \cite{bourbaki2004integrationII} Proposition IX.5.7). Moreover, $\seq{\law{\epsilon_i,\Lambda} \circ \Phi^{-1}}{\epsilon_i}$ is uniformly integrable. Indeed, using \cref{eq:def_L_phi}, we have
	\begin{equation*}
		\sup_\epsilon \espl{\epsilon,\Lambda}{\abs{\Phi(\ul\rho,\ul\zeta,\ul\zeta')}^2} \lesssim_{T,\Lambda,\phi,g} \sup_\epsilon \esp{1 + \norm{f^\epsilon_0}_\fSet^4} < \infty.
	\end{equation*}
	Uniform integrability and convergence in distribution yield (see \cite{billingsley1999convergence}, Theorem 5.4)
	\begin{equation*}
		\espl{\epsilon_i,\Lambda}{\Phi(\ul\rho,\ul\zeta,\ul\zeta')} \xrightarrow[i \to \infty]{} \espl{0,\Lambda}{\Phi(\ul\rho,\ul\zeta,\ul\zeta')}.
	\end{equation*}

	On the other hand, define the perturbed test function $\phi^{\epsilon_i}$ as in \cref{prop:corrector}. As for \cref{prop:convergence_of_the_auxiliary_process}, we have
	\begin{multline*}
			\esp{\parths{\phi^{\epsilon_i}(X\stoppedLi(t)) - \phi^{\epsilon_i}(X\stoppedLi(s)) - \int_{s \wedge \tauLei}^{t \wedge \tauLei} \mathcal L^{\epsilon_i} \phi^{\epsilon_i}(X^{\epsilon_i}(u))du} \right. \\
			\left. \vphantom{\parths{\int_{s \wedge \tauLei}^{t \wedge \tauLei} \mathcal L^{\epsilon_i} \phi^{\epsilon_i}(X^{\epsilon_i}(u))du}} G(\rho\stoppedLi,\zeta\stoppedLi,\zeta\stoppedi)} = 0,
	\end{multline*}
	and
	\begin{equation*}
		\abs{\espl{{\epsilon_i},\Lambda}{\Phi(\ul\rho,\ul\zeta,\ul\zeta')}} = \abs{\esp{\Phi(\rho\stoppedLi,\zeta\stoppedLi,\zeta^{\epsilon_i})}} \lesssim_g \sum_{j=1}^4 \esp{\abs{r_j}},
	\end{equation*}
	with
	\begin{gather*}
		r_1 = {\epsilon_i} (\phi_1(X\stoppedLi(t)) - \phi_1(X\stoppedLi(s))) \to 0\\
		r_2 = {\epsilon_i}^2 (\phi_2(X\stoppedLi(t)) - \phi_2(X\stoppedLi(s))) \to 0\\
		r_3 = - \int_{s \wedge \tauLei}^{t \wedge \tauLei} \parths{\mathcal L^{\epsilon_i} \phi^{\epsilon_i}(X\stoppedLi(u)) - \mathcal L \phi (\bar X\stoppedLi(u))} du \to 0\\
		r_4 = \int_{t \wedge \tauLei}^{t \wedge \tauLzei} \mathcal L \phi(\bar X\stoppedi(u)) du - \int_{s \wedge \tauLei}^{s \wedge \tauLzei} \mathcal L \phi(\bar X\stoppedi(u)) du.
	\end{gather*}
	For the last term $r_4$, we have
	\begin{align*}
		\esp{\abs{r_4}}^2
			&\lesssim_{\phi,\Lambda} \esp{\abs{t \wedge \tauLzei - t \wedge \tauLei}^2 + \abs{s \wedge \tauLzei - s \wedge \tauLei}^2}\\
			&\lesssim_{\phi,\Lambda} T^2 \proba{\tauei < T \wedge \tau_\Lambda(\zeta^{\epsilon_i})} \mbox{ using \cref{eq:def_tauLe}}\\
			&\lesssim_{\phi,\Lambda} T^2 \proba{\tauei < T} \xrightarrow[i \to \infty]{} 0.
	\end{align*}
	% Moreover, using \cref{eq:def_L_phi}, we have
	% \begin{equation*}
	% 	\esp{\abs{r_4}} \lesssim_{\phi,\Lambda} \esp{(1 + \norm{f^{\epsilon_i}_0}_\fSet^2) (\abs{t \wedge \tauLzei - t \wedge \tauLei} + \abs{s \wedge \tauLzei - s \wedge \tauLei})}.
	% \end{equation*}
	% Owing to the Cauchy-Schwarz inequality, to \cref{eq:def_tauLe} and to \cref{lem:taue_to_infty}, we get
	% \begin{align*}
	% 	\esp{\abs{r_4}}^2
	% 		&\lesssim_{\phi,\Lambda} \esp{\abs{t \wedge \tauLzei - t \wedge \tauLei}^2 + \abs{s \wedge \tauLzei - s \wedge \tauLei}^2}\\
	% 		&\lesssim_{\phi,\Lambda} T^2 \proba{\tauei < T} \xrightarrow[i \to \infty]{} 0.
	% \end{align*}
	Thus, we get $\espl{\epsilon_i,\Lambda}{\Phi(\ul\rho,\ul\zeta,\ul\zeta')} \xrightarrow[i \to \infty]{} 0$, which concludes the proof of \cref{prop:limit_martingale}.
\end{proof}

\subsection{Identification of the limit point} \label{subsec:identification_limit_point}

In \cref{subsec:convergence_of_the_auxiliary_process}, solving the martingale problem is sufficient to characterize the distribution of the Markov process as a solution of a limit equation, under a uniqueness condition. However, the limit point $\law{0,\Lambda}$ solves a martingale problem only until a stopping time $\tau_\Lambda(\ul \zeta')$. The goal of this section is to explain how to identify $\law{0,\Lambda}$ using this stopped martingale problem.

Let us come back to the space $\Omega$ to state more precise results. Recall that the distribution of $(\rho\stoppedL,\zeta\stoppedL,\zeta^\epsilon)$ is $\law{\epsilon,\Lambda}$, and define $(\rho_\Lambda,\zeta_\Lambda,\zeta')$ following the limit distribution $\law{0,\Lambda}$ (we assume $\Omega$ is large enough to define such a process). Recall that $\bar X\stoppedL \doteq (\rho\stoppedL,\zeta\stoppedL)$. Define $\bar X_\Lambda \doteq (\rho_\Lambda,\zeta_\Lambda)$ and $\bar X$ a solution of \cref{eq:limit_equation}.

In this section, we construct a process $Y_\Lambda$ that extends $\bar X_\Lambda$ after the stopping time $\tau_\Lambda(\zeta')$ (in distribution) and that solves the martingale problem associated to $\mathcal L$. It is similar to the proof of Theorem 6.1.2 in \cite{stroock2006multidimensional}, but we adapt this proof to see precisely how $\tau_\Lambda(\zeta')$ is linked to the extended process.

\paragraph{Extension after a stopping time}

We first need a result to assert that $\tau_\Lambda(\zeta')$ is a hitting time for $\bar X_\Lambda$. Note that until here, we did not use $\zeta^\epsilon$  when considering the coupled process $(\rho^\epsilon,\zeta^\epsilon)$.% and computing the perturbed test functions and the limit generator and that considering $(\rho^\epsilon,\zeta^\epsilon)$ instead of $\rho^\epsilon$ made the computations heavier.
But had we considered $\rho^\epsilon$ alone, the stopping time $\tau_\Lambda(\zeta')$ would not be a hitting time for $\rho_\Lambda$ (as a matter of fact, $\tau_\Lambda(\zeta')$ is not even a stopping time for the filtration generated by $\rho_\Lambda$).
\begin{lemma} \label{lem:zeta_lambda_rho_lambda}
	Let $\Lambda \in \mathfrak L$.

	The processes $\zeta_\Lambda$ and $(\zeta')^{\tau_\Lambda(\zeta')}$ are indistinguishable. In particular, $\tau_\Lambda(\zeta_\Lambda) = \tau_\Lambda(\zeta')$. Moreover, the processes $\rho_\Lambda$ and $\rho_\Lambda^{\tau_\Lambda(\zeta')}$ are indistinguishable.
\end{lemma}
This result was expected, given the construction of the stopping times and the fact that $\zeta_\Lambda$ and $\zeta'$ are the limit of the same process, respectively with and without a stopping time. The choice $\Lambda \in \mathfrak L$ is here necessary to retrieve this result by taking the limit $\epsilon \to 0$.
\begin{proof}
	Since $\tauei \to \infty$ in probability by \cref{lem:taue_to_infty}, Slutsky's Lemma yields the following convergence in distribution
	\begin{equation*}
	(\zeta\stoppedLi,\zeta\stoppedi,\zeta^{\epsilon_i}) \xrightarrow[i \to \infty]{} (\zeta_\Lambda, \zeta', \zeta')
	\end{equation*}

	Now, for $z_1,z_2,z_3 \in C_T C^1_x$, let
	$$\Phi(z_1,z_2,z_3) = \norm{z_1 - z_2^{\tau_\Lambda(z_3)}}_{C_T C^1_x}.$$
	Owing to \cref{lem:tauL_continuity}, the mapping $\Phi$ is almost surely continuous at $(\zeta_\Lambda, \zeta', \zeta')$. Thus $\Phi(\zeta\stoppedLi, \zeta\stoppedi, \zeta^{\epsilon_i}) = 0$ converges in distribution to $\Phi(\zeta_\Lambda, \zeta', \zeta')$. Hence, we have almost surely $\zeta_\Lambda = (\zeta')^{\tau_\Lambda(\zeta')}$.

	The proof for $\rho_\Lambda$ uses similar arguments with $\Phi(\rho,z) = \norm{\rho - \rho^{\tau_\Lambda(z)}}_{C_T H^{-\sigma}_x}$.
\end{proof}

From now on, for any process $Y = (\rho,\zeta)$, we write $\tau_\Lambda(Y) = \tau_\Lambda(\zeta)$ so that $\tau_\Lambda(\bar X_\Lambda) = \tau_\Lambda(\zeta') \in [0,\infty]$. We shorten the notation to $\tau_\Lambda \doteq \tau_\Lambda(\bar X_\Lambda)$. Introduce the measurable function $S_\Lambda$ that stops a process at the level $\Lambda$, namely $S_\Lambda(Y) = Y^{\tau_\Lambda(Y)}$. Owing to \cref{lem:zeta_lambda_rho_lambda}, we have $S_\Lambda(\bar X_\Lambda) = \bar X_\Lambda$.

This section is devoted to \emph{extend} $\bar X_\Lambda$ after $\tau_\Lambda$ into a solution of the martingale problem associated to $\mathcal L$. Namely, we define a process $Y_\Lambda$ such that $S_\Lambda(Y_\Lambda) \stackrel{d}{=} \bar X_\Lambda$ and such that $Y_\Lambda$ solves the aforementioned martingale problem.

Fix $\omega' \in \Omega$. Define the process $\bar X_{\Lambda,\omega'}$ as follows:
\begin{itemize}
	\item $\forall \omega \in \Omega, \forall t \leq \tau_\Lambda\parths{\omega'}, \bar X_{\Lambda,\omega'}(t)\parths{\omega} = \bar X_\Lambda(t)\parths{\omega'}$. Note that $\tau_\Lambda(\bar X_{\Lambda,\omega'}) = \tau_\Lambda\parths{\omega'}$ almost surely. In particular, the distribution of $S_\Lambda(\bar X_{\Lambda,\omega'})$ is the Dirac distribution at $\bar X_\Lambda\parths{\omega'}$.
	\item On $[\tau_\Lambda\parths{\omega'},T]$ (this interval can be empty), $\bar X_{\Lambda,\omega'}\parths{\omega}$ is the solution of \cref{eq:limit_equation} starting at time $\tau_\Lambda\parths{\omega'}$ from the initial state $\bar X_\Lambda(\tau_\Lambda\parths{\omega'})\parths{\omega'}$.
\end{itemize}
It is straightforward to check that
\begin{equation*}
	\omega' \mapsto \proba{\bar X_{\Lambda,\omega'} \in \mathcal C}
\end{equation*}
is measurable for $\mathcal C = \set{Y \in C_T H^{-\sigma}_x \times C_T C^1_x \mid Y(t_1) \in \Gamma_1, ..., Y(t_n) \in \Gamma_n}$ with $0 \leq t_1 < ... < t_n \leq T$ and $\Gamma_i$ measurable. Since those sets generate the Borel $\sigma$-algebra of $C_T H^{-\sigma}_x \times C_T C^1_x$, and since a pointwise limit of measurable functions is measurable, we can take the limit when the subdivision become thiner to get that the mapping is still measurable for any measurable $\mathcal C$. Thus, we can define a mapping $\mathcal C \mapsto \mathbb E' \proba{\bar X_{\Lambda,\omega'} \in \mathcal C}$, where $\mathbb E'$ denotes the integration with respect to $\omega'$. It is also straightforward to check that this mapping is a probability measure, thus we can define on $\Omega$ a process $Y_\Lambda$ following this distribution, namely
\begin{equation*}
	\proba{Y_\Lambda \in \mathcal C} = \mathbb E' \proba{\bar X_{\Lambda,\omega'} \in \mathcal C}.
\end{equation*}
In particular, since $S_\Lambda^{-1}(\mathcal C)$ is a measurable set, we have
\begin{align*}
	\proba{S_\Lambda(Y_\Lambda) \in \mathcal C}
		&= \mathbb E' \proba{S_\Lambda(\bar X_{\Lambda,\omega'}) \in \mathcal C}\\
		&= \mathbb E' \ind{\bar X_{\Lambda}\parths{\omega'} \in \mathcal C}\\
		&= \mathbb P \parths{\bar X_{\Lambda} \in \mathcal C},
\end{align*}
hence $Y_\Lambda$ extends $\bar X_\Lambda$ as announced beforehand, in the sense that $S_\Lambda(Y_\Lambda) \stackrel{d}{=} \bar X_\Lambda$. Moreover, for any measurable function $\Phi$ such that $\mathbb E' \esp{\abs{\Phi(\bar X_{\Lambda,\omega'})}} < \infty$, we have
\begin{equation} \label{eq:law_Y}
	\esp{\Phi(Y_\Lambda)} = \mathbb E' \esp{\Phi(\bar X_{\Lambda,\omega'})}.
\end{equation}

% \begin{remark}
% 	We prove later that the distribution of $Y_\Lambda$ is unique, in the sense that it does not depend on the sequence $\seq{\epsilon_i}{i}$, hence so is the distribution of $\bar X_\Lambda$. By \cref{prop:tightness}, this is sufficient to get that $\bar X\stoppedL$ converges in distribution. But the uniqueness of $Y_\Lambda$ is only known in $C_T H^{-\sigma}_x \times C_T L^2_x$, hence the convergence is a priori only obtained in the space $C_T H^{-\sigma}_x \times C_T L^2_x$. This is not a problem, wince the convergence of the stopped process is not required: in \cref{subsec:convergence}, we only use the convergence of a subsequence (of the stopped process) in $C_T H^{-\sigma}_x \times C_T C^1_x$.
% \end{remark}

\paragraph{Identification of the extended process}

It remains to prove that $Y_\Lambda$ solves the martingale problem associated to $\mathcal L$.

For $\phi \in \Theta$, and a process $Y \in C_T H^{-\sigma}_x \times C_T C^1_x$, define the process
\begin{equation*}
	M^{Y}(t) = \phi(Y(t)) - \phi(Y(0)) - \int_0^t \mathcal L \phi(Y(u))du.
\end{equation*}
Let $0 \leq s_1 \leq ... \leq s_n \leq s < t$ and $g$ be a bounded measurable function. Let $G : Y \mapsto g(Y(s_1),...,Y(s_n))$.

Owing to \cref{prop:limit_generator}, for almost all $\omega' \in \Omega$, the process
\begin{equation*}
	N_{\Lambda,\omega'}(t) = M^{\bar X_{\Lambda,\omega'}}(t) - M^{\bar X_{\Lambda,\omega'}}(t \wedge \tau_\Lambda\parths{\omega'})
\end{equation*}
satisfies the martingale property
\begin{equation*}
	\esp{N_{\Lambda,\omega'}(t) G(\bar X_{\Lambda,\omega'})} = \esp{N_{\Lambda,\omega'}(s) G(\bar X_{\Lambda,\omega'})}.
\end{equation*}
Indeed, for $t \in [0,\tau\parths{\omega'}]$, $N_{\Lambda,\omega'}(t) = 0$ and after the time $\tau\parths{\omega'}$, this process solves the martingale problem starting at time $\tau\parths{\omega'}$ by construction. Using \cref{eq:moment_limit_solution} and that $\phi$ and $\mathcal L \phi$ have at most quadratic growth, it is straightforward to establish
$$\mathbb E' \esp{\abs{N_{\Lambda,\omega'}(t)  G(\bar X_{\Lambda,\omega'})}} < \infty.$$
Thus, \cref{eq:law_Y} and the identity above yield
\begin{equation*}
	\esp{(M^{Y_\Lambda}(t) - M^{Y_\Lambda}(t \wedge \tau_\Lambda(Y_\Lambda))) G(Y_\Lambda)} = \esp{(M^{Y_\Lambda}(s) - M^{Y_\Lambda}(s \wedge \tau_\Lambda(Y_\Lambda))) G(Y_\Lambda)},
\end{equation*}
which can be rewritten as
\begin{multline} \label{eq:mart_temp}
	\esp{M^{Y_\Lambda}(t) G(Y_\Lambda)} =\\
	\esp{M^{Y_\Lambda}(s) \ind{\tau_\Lambda(Y_\Lambda) \leq s} G(Y_\Lambda)} + \esp{M^{Y_\Lambda}(t \wedge \tau_\Lambda(Y_\Lambda)) \ind{\tau_\Lambda(Y_\Lambda) > s} G(Y_\Lambda)}.
\end{multline}
Using that the process $Y_\Lambda$ and $S_\Lambda(Y_\Lambda)$ are equal until the time $\tau_\Lambda(Y_\Lambda) = \tau_\Lambda(S_\Lambda(Y_\Lambda))$, and that $S_\Lambda(Y_\Lambda)$ and $\bar X_\Lambda$ are equal in distribution, we get for the second term
\begin{multline*}
	\esp{M^{Y_\Lambda}(t \wedge \tau_\Lambda(Y_\Lambda)) \ind{\tau_\Lambda(Y_\Lambda) > s} G(Y_\Lambda)}\\
	\begin{split}
		&= \esp{M^{S_\Lambda(Y_\Lambda)}(t \wedge \tau_\Lambda(S_\Lambda(Y_\Lambda))) \ind{\tau_\Lambda(S_\Lambda(Y_\Lambda)) > s} G(S_\Lambda(Y_\Lambda))}\\
		&= \esp{M^{\bar X_\Lambda}(t \wedge \tau_\Lambda) \ind{\tau_\Lambda > s} G(\bar X_\Lambda)}.
	\end{split}
\end{multline*}
Owing to \cref{prop:limit_martingale}, $t \mapsto M^{\bar X_\Lambda}(t \wedge \tau_\Lambda)$ is a martingale for the filtration $\mathcal F^{\bar X_\Lambda}$ generated by $\bar X_\Lambda$. Moreover $\ind{\tau_\Lambda > s} G(\bar X_\Lambda)$ is $\mathcal F^{\bar X_\Lambda}_s$-measurable, hence the martingale property yields
\begin{equation*}
	\esp{M^{\bar X_\Lambda}(t \wedge \tau_\Lambda) \ind{\tau_\Lambda > s} G(\bar X_\Lambda)} = \esp{M^{\bar X_\Lambda}(s) \ind{\tau_\Lambda > s} G(\bar X_\Lambda)}.
\end{equation*}
Using again that $S_\Lambda(Y_\Lambda) \stackrel{d}{=} \bar X_\Lambda$, we get
\begin{equation*}
	\esp{M^{\bar X_\Lambda}(s) \ind{\tau_\Lambda > s} G(\bar X_\Lambda)} = \esp{M^{Y_\Lambda}(s) \ind{\tau_\Lambda(Y_\Lambda) > s} G(Y_\Lambda)}.
\end{equation*}
Finally, owing to \cref{eq:mart_temp}, we have
\begin{align*}
	\esp{M^{Y_\Lambda}(t) G(Y_\Lambda)}
		&= \esp{M^{Y_\Lambda}(s) \ind{\tau_\Lambda(Y_\Lambda) \leq s} G(Y_\Lambda)} + \esp{M^{Y_\Lambda}(s) \ind{\tau_\Lambda(Y_\Lambda) > s} G(Y_\Lambda)}\\
		&= \esp{M^{Y_\Lambda}(s) G(Y_\Lambda)},
\end{align*}
which proves that $Y_\Lambda$ solves the martingale associated to $\mathcal L$. Owing to \cref{prop:limit_generator}, it solves \cref{eq:limit_equation} and since the solution is unique $Y_\Lambda \stackrel{d}{=} \bar X$ the solution of \cref{eq:limit_equation}. Therefore, the limit point is unique (and does not depend on $\Lambda$). This concludes the proof that $X\stoppedL$ converges in distribution to $\bar X$.

% By uniqueness of its limit point, the tight family $\seq{(\rho\stoppedL,\zeta\stoppedL)}{\epsilon}$ converges in distribution to $\bar X_\Lambda = (\rho_\Lambda,\zeta_\Lambda)$ when $\epsilon \to 0$. As mentioned in \cref{sec:tightness}, using \cref{lem:taue_to_infty} and Slutsky's Lemma, it implies the convergence for $\seq{(\rho\stoppedLz,\zeta\stoppedLz)}{\epsilon}$ to the same limit.

\subsection{Convergence of the unstopped process} \label{subsec:convergence}

This section is devoted to the proof that the process $\bar X^\epsilon \doteq (\rho^\epsilon,\zeta^\epsilon)$ converges in distribution to $\bar X \doteq (\rho,\zeta)$ solution of \cref{eq:limit_equation}, in $C_T H^{-\sigma}_x \times C_T C^1_x$.

Let $\Phi$ be a continuous bounded mapping from $C_T H^{-\sigma}_x \times C_T C^1_x$ to $\mathbb R$. There exists a sequence $\epsilon_i$ such that $\epsilon_i \to 0$ when $i \to \infty$ and
\begin{gather*}
	\limsup_{\epsilon \to 0} \abs{\esp{\Phi(\bar X^{\epsilon})} - \esp{\Phi(\bar X)}} = \lim_{i \to \infty} \abs{\esp{\Phi(\bar X^{\epsilon_i})} - \esp{\Phi(\bar X)}}.
\end{gather*}
Let $\Lambda \in \mathfrak L$. Owing to \cref{prop:tightness}, up to the extraction of another subsequence, we can assume that $(\bar X\stoppedLi,\zeta^{\epsilon_i})$ converges in distribution to some $(\bar X_\Lambda,\zeta)$ in $(C_T H^{-\sigma} \times C_T C^1_x) \times C_T C^1_x$. Now we write
\begin{multline*}
	\abs{\esp{\Phi(\bar X^{\epsilon_i})} - \esp{\Phi(\bar X)}} \leq \abs{\esp{\Phi(\bar X^{\epsilon_i})} - \esp{\Phi(\bar X\stoppedLzi)}}\\
	+ \abs{\esp{\Phi(\bar X\stoppedLzi)} - \esp{\Phi(\bar X)}}.
\end{multline*}

First, we have
\begin{equation*}
	\abs{\esp{\Phi(\bar X^{\epsilon_i})} - \esp{\Phi(\bar X\stoppedLzi)}} \lesssim_\Phi \proba{\tauLzei \leq T}.
\end{equation*}
By \cref{lem:tauL_continuity,lem:tauLze_as_continuity}, since $\Lambda \in \mathfrak L$, $\tauLzei \wedge 2T$ converges in distribution to $\tau_\Lambda(\zeta') \wedge 2T$. Then, by Portmanteau's Theorem for closed sets, we have
\begin{equation*}
	\limsup_i \proba{\tauLzei \leq T} \leq \proba{\tau_\Lambda(\zeta') \leq T}.
\end{equation*}

Since $\Phi$ is a continuous bounded function, we have
\begin{equation*}
	\lim_i \abs{\esp{\Phi(\bar X\stoppedLzi)} - \esp{\Phi(\bar X)}} = \abs{\esp{\Phi(\bar X_\Lambda)} - \esp{\Phi(\bar X)}}.
\end{equation*}
Recall that $\bar X_\Lambda \stackrel{d}{=} S_\Lambda(Y_\Lambda)$, and that $Y_\Lambda \stackrel{d}{=} \bar X$ (by \cref{subsec:identification_limit_point}). Thus, we get
\begin{align*}
	\abs{\esp{\Phi(\bar X_\Lambda)} - \esp{\Phi(\bar X)}}
		&= \abs{\esp{\Phi(S_\Lambda(Y_\Lambda))} - \esp{\Phi(Y_\Lambda)}}\\
		&\lesssim_\Phi \proba{\tau_\Lambda(Y_\Lambda) \leq T}.
\end{align*}
Since $\tau_\Lambda(Y_\Lambda) \stackrel{d}{=} \tau_\Lambda(\bar X_\Lambda) = \tau_\Lambda(\zeta')$ by \cref{lem:zeta_lambda_rho_lambda}, we finally get for $\Lambda \in \mathfrak L$
\begin{equation*}
	\limsup_{\epsilon \to 0} \abs{\esp{\Phi(\bar X^{\epsilon})} - \esp{\Phi(\bar X)}} \lesssim_\Phi \proba{\tau_\Lambda(\zeta') \leq T}.
\end{equation*}
Since $\zeta' \in C_T C^1_x$, we have $\proba{\tau_\Lambda(\zeta') \leq T} \xrightarrow[\Lambda \to \infty]{} 0$. Recall that we can take this limit since $\mathcal L$ contains arbitrarily large $\Lambda$'s. Therefore, we have
\begin{equation*}
	\esp{\Phi(\bar X^{\epsilon})} \xrightarrow[\epsilon \to 0]{} \esp{\Phi(\bar X)}.
\end{equation*}
This concludes the proof that $\bar X^\epsilon$ converges in distribution to $\bar X$, and in particular that $\rho^\epsilon$ converges in distribution to $\rho$ in $C_T H^{-\sigma}_x$.

\section{Strong convergence} \label{sec:strong_tightness}

In this section, we establish the second convergence result stated in \cref{thm:main_result}, namely the convergence in $L^2_T L^2_x$. Given \cref{sec:convergence} and \cref{prop:tightness}, it is sufficient to prove that the sequence $\seq{\rho\stoppedL}{\epsilon>0}$ is tight in $L^2_T L^2_x$.

Recall that $w_\rho$ denotes the modulus of continuity of a $H^{-\sigma}_x$-valued continuous process $\rho$. Then, using Theorem 5 in \cite{simon1987compact}, the set
\begin{equation*}
	K_R \doteq \set{\rho \in L^2_T L^2_x \mid \norm{\rho}_{L^2_T H^{\sigma'}_x} \leq R \mbox{ and } \forall \delta > 0, w_\rho(\delta) < \eta(\delta)}
\end{equation*}
where $R>0$, ${\sigma'}>0$ and $\eta(\delta) \xrightarrow[\delta \to 0]{} 0$, is compact in $L^2_T L^2_x$. Using Prokhorov's Theorem, the tightness of $\seq{\rho\stoppedL}{\epsilon>0}$ in $L^2_T L^2_x$ will follow if we prove that for all $\eta > 0$, there exists $R > 0$ and ${\sigma'} > 0$ such that
\begin{equation} \label{eq:strong_tightness_claim2}
	\lim_{\delta \to 0} \limsup_{\epsilon \to 0} \proba{w_{\rho\stoppedL}(\delta) > \eta} = 0,
\end{equation}
and
\begin{equation} \label{eq:strong_tightness_claim1}
	\sup_\epsilon \proba{\norm{\rho\stoppedL}_{L^2_T H^{\sigma'}_x} > R} < \eta.
\end{equation}

Equation \cref{eq:strong_tightness_claim2} is a direct consequence of \cref{eq:aldous}. It remains to prove \cref{eq:strong_tightness_claim1}. Owing to the Markov Inequality, it is sufficient to prove that, for some ${\sigma'} > 0$, we have
\begin{equation} \label{eq:strong_tightness_claim1bis}
	\sup_\epsilon \esp{\norm{\rho\stoppedL}_{L^2_T H^{\sigma'}_x}} \lesssim_\Lambda 1.
\end{equation}
Let $g^\epsilon = \epsilon \partial_t f^\epsilon + \scald{a(v)}{\nabla_x f^\epsilon}$. Owing to \cref{hyp:a2}, we can use an averaging lemma (Theorem 2.3 in \cite{bouchut1999averaging} with $f(t) = f^{\epsilon}(\epsilon t)$, $g(t) = g^{\epsilon}(\epsilon t)$ and $h = 0$ until the time $T \wedge \tauLe$) and by rescaling the time, we get
\begin{align*}
	\norm{\rho\stoppedL}_{L^2_T H^{1/4}_x}^2
		&= \int_0^{T \wedge \tauLe} \norm{\rho\stoppedL(t)}_{H^{\sigma^*/4}_x}^2 dt\\
		&\lesssim \epsilon \norm{f_0^\epsilon}_{L^2_x}^2 + \int_0^{T \wedge \tauLe} \norm{f\stoppedL(t)}_{\fSet}^2 dt + \int_0^{T \wedge \tauLe}\norm{g\stoppedL(t)}_{\fSet}^2 dt,
\end{align*}
where, using the Cauchy-Schwarz inequality,
\begin{align*}
	\norm{g\stoppedL(t)}_{\fSet}
		&= \norm{f\stoppedL(t) \bar m\stoppedL(t) + \frac 1 \epsilon Lf\stoppedL(t)}_{\fSet}\\
		&\leq \norm{f\stoppedL(t)}_{\fSet} \norm{\bar m\stoppedL(t)}_{L^2_x} + \frac 1 \epsilon \norm{Lf\stoppedL(t)}_{\fSet}.
\end{align*}
Then \cref{hyp:f0,eq:taue_bounds_m,prop:L2_bound} lead to \cref{eq:strong_tightness_claim1bis} with $\sigma' = \frac{\sigma^*} 4$. Since the sets $K_R$ are compacts, Prokhorov's Theorem yields, using \cref{eq:strong_tightness_claim1,eq:strong_tightness_claim2}, that $\seq{\rho\stoppedL}{\epsilon>0}$ is tight in $L^2_T L^2_x$.

Given \cref{sec:convergence}, this concludes the proof of the convergence in distribution of $\rho^\epsilon$ in $L^2_T L^2_x$ to $\rho$ the solution of \cref{eq:limit_equation}.

\bibliographystyle{plain}
\bibliography{fullbiblio}
\end{document}